\documentclass[reqno]{amsart}
\usepackage[T1]{fontenc}
\usepackage{amsmath,a4wide}
\usepackage{amssymb}
\usepackage{amscd}
\usepackage{epsfig}
\usepackage[T1]{fontenc}
\usepackage{epsfig}
\usepackage{amsmath}
\usepackage{graphicx}
\usepackage{subfigure}
\usepackage{mathrsfs}
\usepackage{epstopdf}
\usepackage{color}
\usepackage{booktabs}
\usepackage{ifpdf}
\usepackage{cite}
\usepackage{amsthm}
\usepackage{amsmath}
\usepackage{dsfont}
\definecolor{myblue}{rgb}{0.2,0,0.9}
\definecolor{blue-violet}{rgb}{0.54, 0.17, 0.89}
\usepackage{enumitem}
\setlist[enumerate]{label={\upshape(\roman*)}}
\usepackage{algorithm,algorithmic}

\usepackage{pgfplots}
\usepgfplotslibrary{groupplots}
\pgfplotsset{compat=1.12}

\usepackage{mathtools}

\usepackage{tikz}
\usetikzlibrary{arrows,intersections}
\usepackage{boldline}
\usepackage{multirow}
\usepackage[margin=1.32in]{geometry}

\usepackage{varioref}
\usepackage{xr-hyper}
\usepackage{hyperref}
\hypersetup{
colorlinks,linkcolor=blue,citecolor=blue,filecolor=red,urlcolor=blue}

\usepackage{todonotes}

\def\d{{\rm d}}

\def\K{\mathcal{K}}

\def\N{\mathbb{N}}

\def\R{\mathbb{R}}
\def\Rd{\mathbb{R}^d}

\def\A{\mathcal{A}}
\def\B{\mathcal{B}}
\def\C{{\rm C}}
\def\c{{\rm c}}
\def\E{\mathbb{E}}
\def\H{\mathcal{H}}
\def\L{\mathcal{L}}
\def\M{\mathcal{M}}
\def\cN{\mathcal{N}}
\def\cO{\mathcal{O}}
\def\cP{\mathcal{P}}
\def\cU{\mathcal{U}}
\def\W{\mathcal{W}}
\def\epsilon{\varepsilon}
\def\one{\mathds{1}}
\definecolor{myblue}{rgb}{0.2,0,0.9}
\definecolor{blue_violet}{rgb}{0.54, 0.17, 0.89}
\definecolor{darkgreen}{rgb}{0,0.35,0}

\newcommand{\llsim}{%
\mathrel{%
	\raisebox{0.4ex}[0pt][0pt]{%
		$\underset{\raisebox{0.6ex}{$\sim$}}{\ll}$%
	}%
}%
}

\makeatletter
\DeclareRobustCommand*\cal{\@fontswitch\relax\mathcal}
\makeatother

\makeatletter
\newcommand{\labeltext}[3][]{%
\@bsphack%
\csname phantomsection\endcsname
\def\tst{#1}%
\def\labelmarkup{\emph}
\def\refmarkup{}%
\ifx\tst\empty\def\@currentlabel{\refmarkup{#2}}{\label{#3}}%
\else\def\@currentlabel{\refmarkup{#1}}{\label{#3}}\fi%
\@esphack%
\labelmarkup{#2}
}
\makeatother

\newtheorem{theorem}{Theorem}[section]
\newtheorem{proposition}[theorem]{Proposition}

\newtheorem{corollary}[theorem]{Corollary}
\newtheorem{lemma}[theorem]{Lemma}
\numberwithin{equation}{section}

\theoremstyle{definition}
\newtheorem{remark}[theorem]{Remark}
\newtheorem{example}[theorem]{Example}
\newtheorem{definition}[theorem]{Definition}
\newtheorem{assumption}[theorem]{Assumption}

\DeclareMathOperator*{\argmin}{arg\,min}

\DeclareMathOperator{\Lipb}{Lip_b}

\DeclareMathOperator{\ReLU}{ReLU}

\DeclareMathOperator{\supp}{supp}
\DeclareMathOperator{\sgn}{sgn}

\DeclareMathOperator*{\linspan}{span}
\DeclareMathOperator*{\trace}{tr}
\DeclareMathOperator*{\id}{id}
\DeclareMathOperator*{\Cb}{C_b}
\DeclareMathOperator*{\Ck}{C_\kappa}
\DeclareMathOperator*{\Cbi}{C_b^\infty}
\DeclareMathOperator*{\Cci}{C_c^\infty}
\DeclareMathOperator{\cov}{cov}

\usepackage{xparse}
\makeatletter
\RenewDocumentCommand{\title}{om}{%
\IfNoValueTF{#1}
{\gdef\shorttitle{}}
{\gdef\shorttitle{#1}}%
\gdef\@title{#2}%
}
\makeatother

\makeatletter
\def\@tocline#1#2#3#4#5#6#7{\relax
\ifnum #1>\c@tocdepth 
\else
\par \addpenalty\@secpenalty\addvspace{#2}%
\begingroup \hyphenpenalty\@M
\@ifempty{#4}{%
	\@tempdima\csname r@tocindent\number#1\endcsname\relax
}{%
	\@tempdima#4\relax
}%
\parindent\z@ \leftskip#3\relax \advance\leftskip\@tempdima\relax
\rightskip\@pnumwidth plus4em \parfillskip-\@pnumwidth
#5\leavevmode\hskip-\@tempdima
\ifcase #1
\or\or \hskip 2em \or \hskip 2em \else \hskip 3em \fi%
#6\nobreak\relax
\hfill\hbox to\@pnumwidth{\@tocpagenum{#7}}\par
\nobreak
\endgroup
\fi}
\makeatother

\title[Neural operators approximate convex monotone semigroups]{Neural operators approximate strongly 
continuous convex monotone semigroups}

\author[J.~Blessing]{Jonas Blessing}
\address{Department of Mathematics, ETH Zurich, Switzerland}
\email{jonas.blessing@math.ethz.ch}

\author[P.~Schmocker]{Philipp Schmocker}
\address{Department of Mathematics, ETH Zurich, Switzerland}
\email{philipp.schmocker@math.ethz.ch}

\author[A.~Sgarabottolo]{Alessandro Sgarabottolo}
\address{Department of Mathematics, LMU Munich, Germany}
\email{sgarabottolo@math.lmu.de}

\thanks{\textit{MSC2020 Subject Classification:} primary 47H20, 65M15, 68T07; secondary 41A81, 49M25, 65J08}

\date{\today}

\begin{document}

\begin{abstract}
	We approximate strongly continuous convex monotone semigroups by learning their Chernoff-type one-step operators with neural operators.\ First, we introduce the general class of so-called Chernoff-neural operators and show in a universal approximation theorem that they can approximate the Chernoff one-step operators arbitrarily well.\ By using stability estimates between weighted H\"older spaces, the one-step approximation error can be propagated through the iterations which yields universal approximation of the corresponding semigroup.\ Second, we introduce the more specialized class of envelope-neural operators for envelope semigroups which allows us to derive quantitative approximation rates.\ Finally, we illustrate the effectiveness of these neural operators in several numerical examples arising from non-linear partial differential equations, stochastic optimal control and stochastic processes under model uncertainty.
\end{abstract}
\vspace{-1.em}
\maketitle

\vspace{-1.em}
{
	\hypersetup{linkcolor=black}
	\tableofcontents
}
\vspace{-0.5cm}

\section{Introduction}
\label{sec:intro}

In this article, we show that strongly continuous convex monotone semigroups
on spaces of continuous functions can be approximated by iterating a sequence
of neural operators.\ Strongly continuous convex monotone semigroups have been
studied in a series of
articles~\cite{DKN20,nendel21,blessing23,kupper23,kupper25,denk25,blessing25,goldys24}.
Starting from a sequence $(I_n)_{n\in\N}$ of Chernoff-type one-step operators $I_n\colon\Ck\to\Ck$
and a sequence $(h_n)_{n\in\N}\subset (0,\infty)$ with $h_n\to 0$, one can construct a semigroup 
\begin{equation} \label{eq:intro1}
	S(t)f:=\lim_{n\to\infty}I_n^{k_n^t}f
	=\lim_{n\to\infty}\underbrace{(I_n\circ\cdots\circ I_n)}_{k^t_n \text{ times}}f,
\end{equation}
with $k_n^t:=\max\{k\in\N_0\colon kh_n\leq t\}$, whose infinitesimal generator is given by 
\begin{equation} \label{eq:intro1b}
	Af=I'(0)f:=\lim_{n\rightarrow\infty}\frac{I_n f-f}{h_n} \quad\text{for all } f\in\Cbi.
\end{equation}
The comparison principle in~\cite{denk25} guarantees that the semigroup $(S(t))_{t\geq 0}$
is uniquely determined by the infinitesimal behavior of the one-step operators
in equation~\eqref{eq:intro1b}.\

The framework of strongly continuous convex monotone semigroups covers various applications
such as stochastic optimal control, Hamilton--Jacobi--Bellman (HJB) equations,
large deviations and stochastic processes under model uncertainty,
see~\cite{DKN20,nendel21,kupper25,blessing23,goldys24,bartl2021,fuhrmann23,
	Kupper25RiskMeasuresWOT,kupper26,sgarabottolo24}.\ It is also worth mentioning
that convex monotone semigroups do not rely 
on the theory of viscosity solutions~\cite{CIL92,CL83,Lions82} but are consistent
with it, i.e., $u(t):=S(t)f$ is a viscosity solution of the HJB 
equation
\begin{equation} \label{eq:hjb}
	\begin{cases}
		\partial_t u(t)=Au(t), \quad t \in (0,\infty), \\
		\;\;\; u(0)=f,
	\end{cases}
\end{equation}
see~\cite[Theorem~6.3]{goldys24}.\ In particular, if the viscosity solution
is unique, the Chernoff-type approximation~\eqref{eq:intro1} can be seen 
as a monotone approximation scheme~\cite{barles91,CS08,BJ2005,BJ07,Krylov2000,DK05,Jiang2022} 
for the solution operator $f \mapsto u(t):=S(t)f$ of the HJB equation~\eqref{eq:hjb}.

In this paper, we approximate the strongly continuous convex monotone semigroup 
$(S(t))_{t\geq 0}$ by learning, for sufficiently large~$n\in\N$, the Chernoff-type 
one-step operator $I_n$ with a neural operator $\Phi_n$ on a sufficiently rich subset 
of $\C^\alpha_\kappa$ so that equation~\eqref{eq:intro1} yields
\begin{equation} \label{eq:intro2}
	S(t)f\approx I_n^{k_n^t}f\approx\Phi_n^{k_n^t}f \quad\text{for all } f\in\C^\alpha_\kappa,
\end{equation}
where $\C^\alpha_\kappa$ is a weighted H\"older space.\ First, we introduce the general
class of so-called \emph{Chernoff-neural operators}, which are based on the 
infinite-dimensional generalization of neural networks on weighted spaces in~\cite{cuchiero26}.\ 
These neural operators map the infinite-dimensional input $f\in\C^\alpha_\kappa$ 
via suitable hidden layer maps to the real hidden layer space, on which a scalar
non-linear activation function is applied, followed by the multiplication with 
linear readout functions to obtain an output in $\C^\alpha_\kappa$.\ By generalizing 
the weighted Stone--Weierstrass theorem of~\cite{cuchiero26} to the scale of weighted 
H\"older spaces, we establish a universal approximation theorem for the one-step 
operators $(I_n)_{n\in\N}$, see Theorem~\ref{thm:uat_chernoff}.\ The approximation 
error can then be propagated through the iterations to derive the approximation of 
the semigroup in equation~\eqref{eq:intro2}, see Theorem~\ref{thm:approx_chernoff}.

Second, we introduce the more specialized but structure-preserving class of so-called 
\emph{envelope-neural operators} if the sequence of one-step operators $(I_n)_{n\in\N}$
is of the form
\begin{equation} \label{eq:intro3}
	(I_n f)(x)=\sup_{\lambda\in\Lambda}\big((I_{n,\lambda}f)(x)-\eta_n(\lambda)h_n\big),
\end{equation}
where $\Lambda$ is a parameter set, $I_{n,\lambda}$ are linear operators and 
$\eta_n\colon\Lambda\to [0,\infty]$ are penalization functions.\ This framework covers, 
for instance, stochastic optimal control problems, HJB equations and transition semigroups 
of stochastic processes under model uncertainty.\ To make the supremum on the right-hand 
side of equation~\eqref{eq:intro3} tractable, we replace the supremum over $\Lambda$ 
by a maximum over finitely many trainable parameters $\lambda_1,\ldots,\lambda_M\in\Lambda$.\
Since finite maxima can be represented by ReLU networks, we define envelope-neural 
operators as compositions of ReLU networks with finite subcollections of the operators
$(I_{n,\lambda})_{\lambda \in \Lambda}$.\ Under suitable regularity assumptions on the
parameter dependence, we show that envelope-neural operators can approximate the one-step
operators $(I_n)_{n\in\N}$ arbitrarily well and subsequently propagate the resulting
one-step approximation error through the iterations, see Theorem~\ref{thm:uat_nisio}
and Theorem~\ref{thm:approx_nisio}.\ Moreover, we derive explicit approximation rates 
in terms of the size of the neural operator and combine them with the convergence rates 
of the one-step operator in \cite{blessing23} to obtain quantitative error bounds for 
the approximation of the strongly continuous convex monotone semigroup, see
Theorem~\ref{thm:rate_nisio} and Theorem~\ref{thm:approx_nisio_rate}.

Compared to the literature, our approach differs both from deep learning methods that
approximate the solution of a partial differential equation (PDE) for fixed initial 
data and from operator-learning methods that approximate the full solution operator
of the underlying PDE.\ For fixed initial data, our approach is closely related to 
deep splitting schemes~\cite{beck21,wu24} and, more broadly, to neural PDE solvers 
such as the deep BSDE method~\cite{han18}, the deep Galerkin method~\cite{sirignano18},
the deep Ritz method~\cite{e18}, physics-informed neural networks~\cite{raissi19}, 
deep backward dynamic programming~\cite{hure20} and random-feature models~\cite{nelsen21,gonon23,schmocker26},
see also~\cite{han16,grohs21,nakamura21,nusken21,zhou21,cai25} in the context of HJB equations.\
For varying initial data, our approach is related to operator-learning methods for PDEs 
such as Fourier neural operators~\cite{li20}, neural integral operators~\cite{stuart21}, 
DeepONets~\cite{lu21,lanthaler22}, generative equilibrium operators~\cite{kratsios25}
and kernel methods~\cite{ortega26}, see also \cite{lanthaler26} for HJB equations.\ 
These results rely on universal approximation theorems for neural operators between 
infinite-dimensional function spaces starting with~\cite{chen95,mhaskar97} and followed 
by results, e.g., on non-Euclidean domains~\cite{kratsios20,galimberti22,kratsios23}, 
Fr\'echet spaces~\cite{benth21}, and weighted infinite-dimensional spaces~\cite{cuchiero26,teichmann26}.

In contrast to these approaches, we do not approximate the solution operator 
$f\mapsto S(t)f$ directly but instead learn the Chernoff-type one-step operator 
$I_n$ with a neural operator $\Phi_n$ and recover $S(t)$ by iteration.\ 
Hence, the training of $\Phi_n$ only requires evaluations of $I_n$, often 
available in closed form, instead of supervised PDE solution data.\
We illustrate our approach in three numerical examples from non-linear partial 
differential equations, stochastic optimal control, and stochastic processes under
model uncertainty.\ While the Chernoff-neural operator provides a more general
architecture and is straightforward to evaluate, the envelope-neural operator
is particularly efficient in training, achieving accurate approximations with small architecture.

The rest of the article is organized as follows.\ In Section~\ref{sec:notation}, 
we introduce strongly continuous convex monotone semigroups.\
In Section~\ref{sec:cn}, we establish a universal approximation theorem for 
Chernoff-neural operators followed by the one for envelope-neural operators 
and their approximation rates in Section~\ref{sec:nisio}.\ Finally,
we present several numerical experiments in Section~\ref{sec:numerics} while
some proofs are given in Appendices~\ref{app:holder}--\ref{app:ass}.

\section{Setup and notation}
\label{sec:notation}

Let $\kappa\colon\Rd\to (0,\infty)$ be a bounded continuous function satisfying
\begin{equation} \label{eq:weight_cond}
	\sup_{x\in\Rd}\sup_{|y|\leq 1}\frac{\kappa(x)}{\kappa(x-y)}<\infty.
\end{equation}
Typical examples that we have in mind are polynomial weights $\kappa(x):=(1+|x|^q)^{-1}$
with $q \geq 0$. We define $\Ck$ as the space of all continuous functions 
$f\colon\Rd\to\R$ satisfying 
\[ \|f\|_\kappa:=\sup_{x\in\Rd}|f(x)|\kappa(x)<\infty. \]
An operator $S\colon\Ck\to\Ck$ is called convex if
$S(\lambda f+(1-\lambda)g)\leq\lambda Sf+(1-\lambda)Sg$ for all $\lambda\in [0,1]$ 
and $f,g\in\Ck$ and monotone if $Sf\leq Sg$ for all $f,g\in\Ck$ with $f\leq g$,
where $f\leq g$ means that $f(x)\leq g(x)$ for all $x\in\Rd$.\ Furthermore, 
we write $f_n\downarrow f$ if a sequence $(f_n)_{n\in\N}\subset\Ck$ decreases
pointwise to another function $f\in\Ck$ and define 
$B_{\Ck}(r):=\{f\in\Ck\colon\|f\|_\kappa\leq r\}$.\ Let $\Lipb$ be the space of all 
bounded Lipschitz functions $f\colon\Rd\to\R$ and denote by $\Lipb(r)$ the set of 
all $r$-Lipschitz functions $f\colon\Rd\to\R$ with $\|f\|_\infty\leq r$. Moreover, 
the space $\C^1_b$ (resp., $\Cbi$; resp. $\Cci$) consists of all bounded differentiable 
(resp., infinitely differentiable; resp. compactly supported infinitely differentiable) 
functions $f\colon\Rd\to\R$ such that all existing derivatives are bounded.

Throughout this article, we endow $\Ck$ with the mixed topology between $\|\cdot\|_\kappa$ 
and the topology of uniform convergence on compact sets, i.e., the strongest locally 
convex topology on $\Ck$ that coincides on $\|\cdot\|_\kappa$-bounded sets with the
topology of uniform convergence on compact subsets, see~\cite{wiweger61} for a detailed 
introduction. Choosing the mixed topology rather than the norm topology is crucial 
for the analysis of strongly continuous convex monotone semigroups, see~\cite{blessing23,denk25,kupper25}.\
Although the mixed topology is not metrizable, for convex monotone operators $S\colon\Ck\to\Ck$, 
continuity w.r.t.\ the mixed topology is equivalent to sequential continuity and to continuity 
from above, see~\cite{delbaen23,nendel25}. Moreover, a sequence $(f_n)_{n\in\N}\subset\Ck$ 
converges to $f\in\Ck$ if and only if 
\[ \sup_{n\in\N}\|f_n\|_\kappa<\infty \quad\mbox{and}\quad \lim_{n\to\infty}\|f-f_n\|_{\infty,K}=0 \]
for all compact subsets $K\Subset\Rd$, where $\|f\|_{\infty,K}:=\sup_{x\in K}|f(x)|$,
see~\cite[Proposition~A.4]{goldys24}. If not stated otherwise, all limits in $\Ck$ are
taken w.r.t.\ the mixed topology and compact subsets are denoted by $K\Subset\Rd$.\ 
The mixed topology can also be seen as a strict topology, 
see~\cite{fremlin72,Kunze09,sentilles72}.

In order to derive a universal approximation result, we further introduce weighted 
H\"older spaces as follows. For every $\alpha \in [0,1]$, we denote by $\C^\alpha_\kappa$ 
the space of all functions $f\colon\Rd\to\R$ satisfying
\[ \|f\|_{\alpha,\kappa}:=\max\bigg\{\sup_{x\in\Rd}|f(x)|\kappa(x),\,
\sup_{x\neq y}\frac{|f(x)-f(y)|}{|x-y|^\alpha}\bar{\kappa}(x,y)\bigg\}<\infty, \]
where $\bar{\kappa}(x,y) := \frac{2\kappa(x)\kappa(y)}{\kappa(x)+\kappa(y)}=\frac{2}{\kappa(x)^{-1}+\kappa(y)^{-1}}$ satisfies $\min(\kappa(x),\kappa(y))\leq\bar{\kappa}(x,y)\leq 2\min(\kappa(x),\kappa(y))$. By Theorem~\ref{thm:hol_ban}, the space $(\C^\alpha_\kappa,\|\cdot\|_{\alpha,\kappa})$ 
is a Banach space satisfying $\C^1_\kappa=\Lipb$ if $\kappa\equiv 1$. Moreover, 
since the norms $\|\cdot\|_{0,\kappa}$ and $\|\cdot\|_\kappa$ are equivalent, 
we have $\C^0_\kappa=\Ck$ as sets, but we keep the notation $\C^0_\kappa$ to 
distinguish the norm topology from the mixed topology on $\Ck$.\ In addition,
Proposition~\ref{thm:hol_cont} guarantees that the embedding 
$\C^\alpha_\kappa\hookrightarrow\C^{\alpha'}_{\kappa'}$ is continuous if 
$0\leq\alpha'\leq\alpha\leq 1$ and $\kappa'\lesssim\kappa$, where $\kappa'\lesssim\kappa$ 
means that $\kappa'\leq c\kappa$ for some constant $c\geq 0$.\ By Theorem~\ref{thm:hol_cpt},
the embedding $\C^\alpha_\kappa\hookrightarrow\C^{\alpha'}_{\kappa'}$ is even compact if 
$0\leq\alpha'<\alpha\leq 1$ and $\kappa'\llsim\kappa$, where $\kappa'\llsim\kappa$ means 
that $\lim_{|x|\to\infty}\frac{\kappa'(x)}{\kappa(x)}=0$. Moreover, we denote by 
$\c^\alpha_\kappa$ the little weighted H\"older space consisting of all functions 
$f\in\C^\alpha_\kappa$ with
\[ \lim_{|x|\to\infty}|f(x)|\kappa(x) = 0 \quad\text{and}\quad
\lim_{\delta\to 0}\sup_{x\neq y,\atop |x-y|<\delta}\frac{|f(x)-f(y)|}{|x-y|^\alpha}\bar{\kappa}(x,y)=0. \]
Note that $\c^\alpha_\kappa\subset\C^\alpha_\kappa$ is a closed linear subspace and 
therefore itself a Banach space.\ In particular, the embedding 
$\C^\alpha_\kappa\hookrightarrow\c^{\alpha'}_{\kappa'}$ is continuous and dense if 
$0\leq\alpha'<\alpha\leq 1$ and $\kappa'\llsim\kappa$, see Theorem~\ref{thm:hol_cont} 
and Theorem~\ref{thm:hol_dense}. We define 
$B_{\C^\alpha_\kappa}(r):=\{f\in \C^\alpha_\kappa\colon\|f\|_{\alpha,\kappa}\leq r\}$
for all $r\geq 0$.

Throughout this article, we also use the following notations.\ Let 
$\R_+:=\{x\in\R\colon x\geq 0\}$.\ We define $B_{\Rd}(r):=\{x\in\Rd\colon |x|\leq r\}$
for all $r\geq 0$, where $|\cdot|$ is the Euclidean norm.\ The transpose of 
a vector $x\in\Rd$ is denoted by $x^\top$, compact subsets are denoted by 
$K\Subset\Rd$ and the translation of a function $f\colon\Rd\to\R$ by a vector
$x\in\Rd$ is denoted by $(\tau_x f)(y):=f(x+y)$ for all $y\in\Rd$.\ The set~$\M_\kappa$
consists of all signed Borel measures $\mu$ with $\int_{\Rd}\frac{1}{\kappa(x)}\,\vert\mu\vert(\d x)<\infty$.\
Moreover, for topological spaces $X,Y$, the space $\Cb(X;Y)$ consists of all uniformly
bounded continuous functions $f\colon X\to Y$, while $\C(X)$ consists of all
continuous functions $f\colon X\to\R$.\ For a Banach space~$X$, we denote by $\C^1(X)$
the space of all continuously differentiable functions $f\colon X\to\R$. If $(X,d)$ 
is a metric space, we define $\mathring{B}_d(x,r):=\{y\in X\colon d(x,y)<r\}$ 
for all $r>0$ and $x\in X$. Finally, the set $\mathbb{S}^d_+\subset\R^{d\times d}$ 
consists of all symmetric positive semi-definite matrices.

\subsection{Strongly continuous convex monotone semigroups}

In this section, we recall the results from~\cite[Section~5]{denk25} 
and~\cite[Section~2]{blessing23} about Chernoff-type approximations of strongly 
continuous convex monotone semigroups.

\begin{definition} \label{def:semigroup}
	A family $(S(t))_{t\geq 0}$ of operators $S(t)\colon\Ck\to\Ck$ is called 
	\emph{strongly continuous convex monotone semigroup} if the following conditions are satisfied:
	\begin{enumerate}
		\item $S(t)$ is convex and monotone with $S(t) f_n\downarrow 0$ 
		for all $t\geq 0$ and $f_n\downarrow 0$.
		\item $S(0)f=f$ and $S(s+t)f=S(s)S(t)f$ for all $s,t\geq 0$ and $f\in\Ck$.
		\item $\sup_{t\in [0,T]}\|S(t)\frac{r}{\kappa}\|_\kappa<\infty$
		for all $r,T\geq 0$.
		\item $f=\lim_{t\downarrow 0}S(t)f$ for all $f\in\Ck$.
	\end{enumerate}
	Furthermore, the generator of the semigroup is defined by
	\[ A\colon D(A)\to\Ck,\; f\mapsto\lim_{h\downarrow 0}\frac{S(h)f-f}{h}, \]
	where the domain consists of all $f\in\Ck$ such that the previous limit exists.
\end{definition}

Let $(I_n)_{n\in\N}$ be a sequence of operators $I_n\colon\Ck\to\Ck$ and 
$(h_n)_{n\in\N}\subset (0,\infty)$ be a sequence with $h_n\to 0$. For every $t\geq 0$,
$n\in\N$ and $f\in\Ck$, we define 
\[ I(\pi^t_n)f:= I_n^{k^t_n}f=\underbrace{(I_n\circ\cdots\circ I_n)}_{k^t_n \text{ times}}f, \]
where $k^t_n:=\max\{k\in\N_0\colon kh_n\leq t\}$ and $\pi^t_n:=\{0,h_n,\ldots,k_n^t h_n\}$.\
In addition, for every $f\in\Ck$ such that the following limit exists, we define
\[ I'(0)f:=\lim_{n\to\infty}\frac{I_n f-f}{h_n}\in\Ck. \]
The following conditions guarantee that the sequence of iterated operators converges
to a strongly continuous convex monotone semigroup which is uniquely determined by
the infinitesimal behavior of the one-step operators.

\begin{assumption} \label{ass:chernoff}
	Suppose that the following conditions are satisfied:
	\begin{enumerate}
		\item $I_n$ is convex and monotone with $I_n 0=0$ for all $n\in\N$.
		\item There exists $\omega\geq 0$ with 
		$\|I_n f-I_n g\|_\kappa\leq e^{\omega h_n}\|f-g\|_\kappa$ for all $f,g\in\Ck$ and $n\in\N$. 
		\item For every $\epsilon>0$, $r,T\geq 0$, and $K\Subset\Rd$ there exist
		$c\geq 0$ and $K'\Subset\Rd$ with 
		\[ \|I(\pi^t_n)f-I(\pi^t_n)g\|_{\infty,K}\leq c\|f-g\|_{\infty,K'}+\epsilon \]
		for all $t\in [0,T]$, $n\in\N$ and $f,g\in B_{\Ck}(r)$.
		\item For every $\epsilon>0$, there exists $\delta>0$ and $n_0\in\N$ with
		\[ I_n (\tau_x f)\leq\tau_x I_n f+\frac{r\epsilon h_n}{\kappa} \]
		for all $r\geq 0$, $f\in\Lipb(r)$, $x\in B_{\Rd}(\delta)$ and $n\geq n_0$.
		\item It holds $I_n\colon\Lipb(r)\to\Lipb(e^{\omega h_n}r)$ for all $r\geq 0$ and $n\in\N$. 
		\item The limit $I'(0)f\in\Ck$ exists for all $f\in\Cbi$.
	\end{enumerate}
\end{assumption}

It follows from~\cite[Lemma~C.2]{denk25} that condition~(iii) is equivalent to
the following: for every $T\geq 0$, $K\Subset\Rd$ and $(f_k)_{k\in\N}\subset\Ck$ 
with $f_k\downarrow 0$, it holds 
\[ \sup_{(t,x)\in [0,T]\times K}\sup_{n\in\N}\big(I(\pi^t_n)f_k\big)(x)\downarrow 0. \]
Moreover, we refer to~\cite[Section~2.4]{blessing23} for sufficient conditions
on the one-step operators~$I_n$ which guarantee that condition~(vi) is satisfied
and which can easily be verified in applications.\ The next theorem follows immediately 
from~\cite[Theorem~2.11 and Corollary~2.15]{blessing23}.

\begin{theorem} \label{thm:chernoff}
	Suppose that Assumption~\ref{ass:chernoff} is satisfied. Then, there exists a strongly 
	continuous convex monotone semigroup $(S(t))_{t \geq 0}$ on $\Ck$ with generator 
	$A\colon D(A)\to\Ck$ given by
	\[ S(t)f:=\lim_{n\to\infty}I(\pi_n^t)f \quad\text{for all } t\geq 0 \text{ and } f\in\Ck. \]
	In addition, the following statements are true:
	\begin{enumerate}
		\item It holds $f\in D(A)$ and $Af=I'(0)f$ for all $f\in\Ck$ such that $I'(0)f\in\Ck$ exists.
		In particular, this is valid for all $f\in\Cbi$.
		\item It holds $\|S(t)f-S(t)g\|_\kappa\leq e^{\omega t}\|f-g\|_\kappa$ for all $t\geq 0$
		and $f,g\in\Ck$.
		\item For every $\epsilon>0$, $r,T\geq 0$ and $K\Subset\Rd$, there exists $c\geq 0$ and 
		$K'\Subset\Rd$ with 
		\[ \|S(t)f-S(t)g\|_{\infty,K}\leq c\|f-g\|_{\infty,K'} + \epsilon \]
		for all $t\in [0,T]$ and $f,g\in B_{\Ck}(r)$.
		\item For every $\epsilon>0$, there exists $\delta>0$ with
		\[ S(t)(\tau_x f)\leq\tau_x S(t)f+\frac{e^{\omega t}r\epsilon t}{\kappa} \]
		for all $r,t\geq 0$, $f\in\Lipb(r)$ and $x\in B_{\Rd}(\delta)$.
		\item It holds $S(t)\colon\Lipb(r)\to\Lipb(e^{\omega t}r)$ for all $r,t\geq 0$.
	\end{enumerate}
	Furthermore, let $(T(t))_{t\geq 0}$ be another strongly continuous convex monotone 
	semigroup on~$\Ck$ with generator $B\colon D(B)\to\Ck$ satisfying the conditions~(iv)
	and~(v), $\Cbi\subset D(B)$ and 
	\[ Af=Bf \quad\text{for all } f\in\Cbi. \]
	Then, it holds $S(t)f=T(t)f$ for all $t\geq 0$ and $f\in\Ck$. 
\end{theorem}

\begin{corollary} \label{cor:chernoff}
	Let Assumption~\ref{ass:chernoff} be satisfied and denote by $(S(t))_{t\geq 0}$
	the strongly continuous convex monotone semigroup from Theorem~\ref{thm:chernoff}.\
	Assume that there exists $\alpha\in (0,1]$ and $\tilde{\omega}\geq 0$ with
	\[ \|I_n f\|_{\alpha,\kappa}\leq e^{\tilde{\omega}h_n}\|f\|_{\alpha,\kappa} 
	\quad\text{for all } n\in\N \text{ and } f\in \C^\alpha_\kappa. \]
	Then, it holds $\|S(t)f\|_{\alpha,\kappa}\leq e^{\tilde{\omega}t}\|f\|_{\alpha,\kappa}$ 
	for all $t\geq 0$ and $f \in \C^\alpha_\kappa$. Furthermore, 
	\[ \lim_{n\to\infty}\sup_{f\in B_{\C^\alpha_\kappa}(r)}\sup_{t\in [0,T]}
	\|S(t)f-I(\pi_n^t)f\|_{\alpha',\kappa'}=0 \]
	for all $r,T\geq 0$, $\alpha' \in [0,\alpha)$ and $\kappa' \llsim \kappa$.
\end{corollary}
\begin{proof}
	For every $t\geq 0$ and $f\in\C^\alpha_\kappa$, it follows by induction that
	\begin{equation} \label{eq:cor:chernoff0a}
		\|I(\pi_n^t)f\|_{\alpha,\kappa}\leq e^{\tilde{\omega} t}\|f\|_{\alpha,\kappa}
		\quad\text{for all } n\in\N,
	\end{equation}
	which, by taking the limit, implies
	\begin{equation} \label{eq:cor:chernoff0b}
		\|S(t)f\|_{\alpha,\kappa}\leq e^{\tilde{\omega} t}\|f\|_{\alpha,\kappa}.
	\end{equation}
	Now, let $r,T\geq 0$, $\alpha'\in [0,\alpha)$, $\kappa'\llsim\kappa$ and $\epsilon'>0$.\
	Since $\kappa'\llsim\kappa$, Assumption~\ref{ass:chernoff}(iii) and Theorem~\ref{thm:chernoff}(iii)
	imply that there exists $\delta>0$ with 
	\begin{equation} \label{eq:cor:chernoff1}
		\|S(t)f-S(t)g\|_{\kappa'}<\tfrac{\epsilon'}{5} \quad\text{and}\quad 
		\|I(\pi_n^t)f-I(\pi_n^t)g\|_{\kappa'}<\tfrac{\epsilon'}{5}
	\end{equation}
	for all $t\in [0,T]$, $n\in\N$ and $f,g\in B_{\Ck}(r)$ with $\|f-g\|_{\kappa'}<\delta$.\
	By Theorem~\ref{thm:hol_cpt}, there exist $f_1,\ldots,f_N\in B_{\C^\alpha_\kappa}(r)$ 
	such that, for every $f\in B_{\C^\alpha_\kappa}(r)$, there exists $j\in\{1,\ldots,N\}$ with
	\begin{equation} \label{eq:cor:chernoff2}
		\|f-f_j\|_{\kappa'}<\delta. 
	\end{equation}
	In addition, by using $\kappa'\llsim\kappa$ and the strong continuity, there exist
	$t_1,\ldots,t_M\in [0,T]$ such that, for every $t\in [0,T]$, there exists $i\in\{1,\ldots,M\}$ with
	\begin{equation} \label{eq:cor:chernoff3}
		\|S(t)f_j-S(t_i)f_j\|_{\kappa'}<\tfrac{\epsilon'}{5} \quad\text{and}\quad
		\|I(\pi_n^t)f_j-I(\pi_n^{t_i})f_j\|_{\kappa'}<\tfrac{\epsilon'}{5} \quad\text{for all } j\in\{1,\ldots,N\}.
	\end{equation}
	Theorem~\ref{thm:hol_cpt} yields the existence of a subsequence $(n_l)_{l\in\N}\subset\N$ 
	and functions $g_{ij}\in\C^\alpha_\kappa$ with 
	\[ \lim_{l\to\infty}\|I(\pi_{n_l}^{t_i})f_j-g_{ij}\|_{\kappa'}=0
	\quad\text{for all } i\in\{1,\ldots,M\} \text{ and } j\in\{1,\ldots,N\}. \]
	Since $I(\pi_{n_l}^{t_i})f_j\to S(t_i)f_j$ in the mixed topology, we obtain $g_{ij}=S(t_i)f_j$.\
	In particular, the limit does not depend on the choice of the convergent subsequence, whence
	there exists $n_0\in\N$ with 
	\begin{equation} \label{eq:cor:chernoff4}
		\|I(\pi_n^{t_i})f_j-S(t_i)f_j\|_{\kappa'}<\tfrac{\epsilon'}{5}
		\quad\text{for all } i\in\{1,\ldots,M\}, \, j\in\{1,\ldots,N\} \text{ and } n\geq n_0.
	\end{equation}
	For every $t\in [0,T]$, $f\in B_{\C^\alpha_\kappa}(r)$ and $n\geq n_0$, 
	the inequalities~\eqref{eq:cor:chernoff1}-\eqref{eq:cor:chernoff4} imply
	\begin{align*}
		\|I(\pi_n^{t})f-S(t)f\|_{\kappa'}
		&\leq\|I(\pi_n^{t})f-I(\pi_n^t)f_j\|_{\kappa'}+\|I(\pi_n^t)f_j-I(\pi_n^{t_i})f_j\|_{\kappa'} \\
		&\quad\; +\|I(\pi_n^{t_i})f_j-S(t_i)f_j\|_{\kappa'}+\|S(t_i)f_j-S(t)f_j\|_{\kappa'} \\
		&\quad\; +\|S(t)f_j-S(t)f\|_{\kappa'} \\
		&<\epsilon'.
	\end{align*}
	In order to derive the convergence w.r.t. $\Vert \cdot \Vert_{\alpha',\kappa'}$, we observe that
	\[ \mathcal{F}:=\overline{\mathcal{E}}^{\|\cdot\|_{\alpha',\kappa'}}\cup\{0\} 
	\quad\text{with}\quad 
	\mathcal{E}:=\big\{I(\pi^t_n)f-S(t)f\colon t\in [0,T], n\in\N, f\in B_{\C^\alpha_\kappa}(r)\big\} \]
	is bounded in $\C^\alpha_\kappa$ due to the inequalities~\eqref{eq:cor:chernoff0a}
	and~\eqref{eq:cor:chernoff0b} and therefore compact in $\C^{\alpha'}_{\kappa'}$ by 
	Theorem~\ref{thm:hol_cpt}.\ Since 
	$(\mathcal{F},\|\cdot\|_{\alpha',\kappa'})\hookrightarrow (\mathcal{F},\|\cdot\|_{\kappa'})$
	is a continuous bijection from a compact space into a Hausdorff space, it is a homeomorphism.
	Hence, for every $\epsilon>0$, there exists $\epsilon'>0$ such that, for every $g\in\mathcal{F}$, 
	we have
	\[ \|g\|_{\kappa'}<\epsilon' \quad \Rightarrow \quad \|g\|_{\alpha',\kappa'}<\epsilon. \]
	Combining this estimate with the first part of the proof yields the claim.
\end{proof}

\section{Chernoff-neural operators}
\label{sec:cn}

In this section, we introduce a very general neural operator architecture to approximate 
the one-step operators $(I_n)_{n\in\N}$ from the previous section. In particular, by learning the 
one-step operators uniformly over different functions, we can then iterate the neural 
operators to approximate the corresponding strongly continuous convex monotone semigroup
$(S(t))_{t\geq 0}$.

\subsection{Universal approximation}
\label{sec:uat}

Let $\alpha\in (0,1]$, $p \in (1,\infty)$ and $\kappa,\kappa_0\colon\Rd\to (0,\infty)$ be bounded
continuous functions satisfying $\kappa_0\llsim\kappa$.\ Furthermore, we fix
a sequence $(I_n)_{n\in\N}$ of operators $I_n\colon\Ck\to\Ck$ and a sequence 
$(h_n)_{n\in\N}\subset (0,\infty)$ with $h_n\to 0$ such that Assumption~\ref{ass:chernoff}
is valid.

\begin{assumption} \label{ass:cn_uat}
	The operators $(I_n)_{n\in\N}$ can be extended to operators $I_n\colon\C_{\kappa_0}\to\C_{\kappa_0}$.\
	Furthermore, there exists a set~$\K$ of bounded continuous functions $\kappa'\colon\Rd\to (0,\infty)$
	with $\kappa_0\lesssim\kappa'\lesssim\kappa$ such that $\kappa\in\K$ and the following conditions are satisfied:
	\begin{enumerate}
		\item For every $\kappa_1,\kappa_2\in\K$ with $\kappa_1\llsim\kappa_2$,
		there exists $\kappa_3\in\K$ with $\kappa_1\llsim\kappa_3\llsim \kappa_2$.
		\item It holds $\|I_n f\|_{\alpha',\kappa'}\leq e^{\omega h_n}\|f\|_{\alpha',\kappa'}$ 
		for all $\alpha'\in (0,\alpha]$, $\kappa'\in\K$, $n\in\N$ and $f\in\C^{\alpha'}_{\kappa'}$.
		\item For every $\alpha'\in[0,\alpha)$, $\K\ni\kappa'\llsim\kappa$ and $n\in\N$, there exists $c\geq 0$ with 
		\[ \|I_n f-I_n g\|_{\alpha',\kappa'}\leq c\|f-g\|_{\alpha',\kappa'}
		\quad\mbox{for all } f,g\in\C^{\alpha'}_{\kappa'}. \]
	\end{enumerate}
\end{assumption}

In order to approximate the operators $(I_n)_{n\in\N}$ on $\C^\alpha_\kappa$,
we follow the infinite-dimensional generalization of neural networks over 
weighted spaces in~\cite{cuchiero26}. For every $\alpha'\in [0,\alpha)$ and
$\K\ni\kappa'\llsim\kappa$, we consider the normed vector space 
$\C^{\alpha,\alpha'}_{\kappa,\kappa'}:=(\C^\alpha_\kappa,\|\cdot\|_{\alpha',\kappa'})$
and define the weight function 
\[ \psi_{\alpha',\kappa',p}\colon\C^{\alpha,\alpha'}_{\kappa,\kappa'}\to (0,\infty),\;
f\mapsto (1+\|f\|_{\alpha,\kappa})^p. \]
By Theorem~\ref{thm:hol_cpt}, the set
$K_R:=\psi_{\alpha',\kappa',p}^{-1}((0,R])\subset\C^{\alpha,\alpha'}_{\kappa,\kappa'}$ is compact
and thus $\psi_{\alpha',\kappa',p}$ is an admissible weight in the sense 
of~\cite[Definition~2.1]{cuchiero26}.\ For a normed space $Y$,
we denote by $\B_p(\C^{\alpha,\alpha'}_{\kappa,\kappa'};Y)$ the completion 
of the space of all uniformly bounded continuous operators 
$T\colon\C^{\alpha,\alpha'}_{\kappa,\kappa'}\to Y$ w.r.t. 
\[ \|T\|_{\B_p(\C^{\alpha,\alpha'}_{\kappa,\kappa'};Y)} 
:=\sup_{f\in\C^\alpha_\kappa}\frac{\|Tf\|_Y}{(1+\|f\|_{\alpha,\kappa})^p}. \]
Here, an operator $T\colon\C^{\alpha,\alpha'}_{\kappa,\kappa'}\to Y$ is called uniformly
bounded if $\sup\{\|Tf\|_Y\colon f\in\C^{\alpha,\alpha'}_{\kappa,\kappa'}\}<\infty$.\ 
In the case of $Y:=\R$, it follows from~\cite[Lemma~2.7]{cuchiero26} that 
$T\in\B_p(\C^{\alpha,\alpha'}_{\kappa,\kappa'};\R)$ if and only if 
$T\vert_{K_R}$ is continuous for all $R>0$ and 
$\lim_{R\to\infty}\sup_{f\in\C^\alpha_\kappa\setminus K_R}\frac{|Tf|}{\psi_{\alpha',\kappa',p}(f)}=0$.

Now, we introduce Chernoff-neural operators mapping the input $f\in\C^\alpha_\kappa$
via some mappings $h\in\H$ to the real hidden layer space, on which a non-linear 
activation function $\rho\in\C(\R)$ is applied, and then return the output via 
some linear readouts $\L\subset\C^\alpha_\kappa$ into $\C^\alpha_\kappa$, 
see Figure~\ref{fig:cn}.

\begin{definition}
	For a set $\H$ consisting of mappings $h\colon\C_{\kappa_0}\to\R$, $\rho\in\C(\R)$
	and a subset $\L\subset\C^\alpha_\kappa$, a \emph{Chernoff-neural operator} is a mapping 
	of the form
	\[ \Phi\colon\C_{\kappa_0}\to\C^\alpha_\kappa,\; f \mapsto\sum_{m=1}^M y_m\rho (h_m(f)),\]
	where $M\in\N$ is the number of neurons, $h_1,\ldots,h_M\in\H$ are hidden layer maps
	and $y_1,\ldots,y_M\in\L$ are linear readouts. We collect all operators of this form
	in the set $\mathcal{CN}^{\H,\rho,\L}_{\C_{\kappa_0},\C^\alpha_\kappa}$.
\end{definition}

\begin{figure}[ht]
	\centering
	\begin{tikzpicture}[
		inputnode/.style={circle, draw=green!60, fill=green!5, very thick, minimum size=4mm},
		hiddennode/.style={circle, draw=blue!60, fill=blue!5, very thick, minimum size=4mm},
		outputnode/.style={circle, draw=red!60, fill=red!5, very thick, minimum size=4mm},
		node distance=7mm,
		]
		\node[inputnode] (x1) {};
		\node[hiddennode] (y2) [right = 2cm of x1] {};
		\node[hiddennode] (y1) [above of = y2] {};
		\node[hiddennode] (y3) [below of = y2] {};
		\node[outputnode] (o1) [right = 2cm of y2] {};
		
		\draw[shorten >=0.1cm,shorten <=0.1cm, ->] (x1.east) -- (y1.west);	
		\draw[shorten >=0.1cm,shorten <=0.1cm, ->] (x1.east) -- (y2.west);
		\draw[shorten >=0.1cm,shorten <=0.1cm, ->] (x1.east) -- (y3.west);
		\draw[shorten >=0.1cm,shorten <=0.1cm, ->] (y1.east) -- (o1.west);
		\draw[shorten >=0.1cm,shorten <=0.1cm, ->] (y2.east) -- (o1.west);
		\draw[shorten >=0.1cm,shorten <=0.1cm, ->] (y3.east) -- (o1.west);
		\draw[shorten >=0.2cm,shorten <=0.2cm, ->] (y3) to[out=-120,in=-60,loop] ();
		
		\draw[] (-0.1,0.9) node[anchor=center, align=center] {\footnotesize Input Layer};
		\draw[] (0,0.6) node[anchor=center, align=center] {\footnotesize $\C_{\kappa_0}$};
		\draw[] (2.45,1.5) node[anchor=center, align=center] {\footnotesize Hidden Layer};
		\draw[] (2.45,1.2) node[anchor=center, align=center] {\footnotesize $\R$};
		\draw[] (4.85,0.9) node[anchor=center, align=center] {\footnotesize Output Layer};
		\draw[] (4.85,0.6) node[anchor=center, align=center] {\footnotesize $\C^\alpha_\kappa$};
		\draw[] (-0.85,0) node[anchor=center, align=center] {\footnotesize $f \in \C_{\kappa_0}$};
		\draw[] (6.0,0) node[anchor=center, align=center] {\footnotesize $\Phi(f) \in \C^\alpha_\kappa$};
		\draw[] (1.25,-0.7) node[anchor=center, align=center] {\footnotesize $\mathcal{H}$};
		\draw[] (3.6,-0.7) node[anchor=center, align=center] {\footnotesize $\mathcal{L}$};
		\draw[] (2.43,-1.2) node[anchor=center, align=center] {\footnotesize $\rho$};
	\end{tikzpicture}
	\caption{A Chernoff-neural operator $\Phi\colon\C_{\kappa_0}\to\C^\alpha_\kappa$
		with additive family $\H$, activation function $\rho\in\C(\R)$, linear readout 
		$\L\subset\C^\alpha_\kappa$ and $M=3$ neurons.}
	\label{fig:cn}
\end{figure}
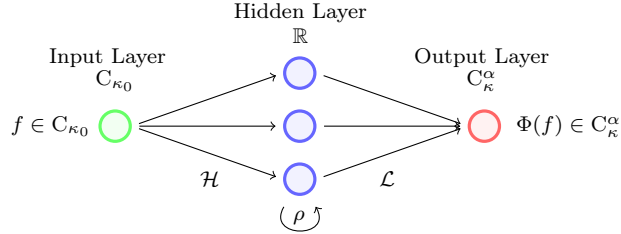

\begin{assumption} \label{ass:cn_operator}
	Suppose that the following conditions are satisfied:
	\begin{enumerate}
		\item The set $\H$ is a vector space, contains all constant mappings and is point 
		separating on~$\C_{\kappa_0}$, i.e., for every $f,g\in\C_{\kappa_0}$ 
		with $f\neq g$ there exists $h\in\H$ with $h(f)\neq h(g)$.\ Moreover, for every 
		$0\leq\alpha'<\alpha''\leq\alpha$, $\kappa'\llsim\kappa''\in\K$ and $h\in\H$, the mapping
		$h\colon\C^{\alpha'',\alpha'}_{\kappa'',\kappa'} \rightarrow \R$ is continuous and satisfies
		\[ c_{\alpha'',\kappa'',h}:=\sup_{f\in\C^{\alpha''}_{\kappa''}}
		\frac{1+|h(f)|}{1+\|f\|_{\alpha'',\kappa''}}<\infty. \]
		\item The function $\rho\in\C(\R)$ is non-polynomial with 
		$\lim_{|x|\to\infty}\frac{|\rho(x)|}{(1+|x|)^p}=0$. 
		\item The set $\L\subset\C^\alpha_\kappa$ is a dense linear subspace of $\c^{\alpha'}_{\kappa'}$
		for all $0\leq\alpha'<\alpha$ and $\K\ni\kappa'\llsim\kappa$.
	\end{enumerate}
\end{assumption}

\begin{example}
	As an example for Assumption~\ref{ass:cn_operator}(i), we consider the set
	\[ \H:=\left\{\C_{\kappa_0}\to\R,\; f\mapsto\int_{\Rd} f(x)\,\nu(\d x)+b\colon 
	\nu\in\M_{\kappa_0}, b\in\R\right\}. \]
	Since $\H$ contains the Dirac measures $\{\delta_x\colon x\in\Rd\}$, it is point separating.\
	Moreover, for every $h\in\H$ of the form $h(f):=\int_{\Rd} f(x)\,\nu(\d x)+b$, $\alpha''\in [0,\alpha]$
	and $\kappa''\in\K$, it holds 
	\[ \sup_{f\in\C^{\alpha''}_{\kappa''}}\frac{1+|h(f)|}{1+\|f\|_{\alpha'',\kappa''}}
	\leq\sup_{f\in\C^{\alpha''}_{\kappa''}}\frac{1+\|f\|_{\kappa_0}c_\nu+|b|}{1+\|f\|_{\alpha'',\kappa''}} 
	\leq\sup_{f\in\C^{\alpha''}_{\kappa''}}\frac{1+c_{\alpha'',\kappa''}\|f\|_{\alpha'',\kappa''}c_\nu+|b|}
	{1+\|f\|_{\alpha'',\kappa''}}<\infty, \]
	where $c_\nu:=\int_{\Rd}\frac{1}{\kappa_0(x)}|\nu|(\d x)$ and $c_{\alpha'',\kappa''}$ is 
	the operator norm of the embedding $\C^{\alpha''}_{\kappa''} \hookrightarrow \C_{\kappa_0}$.
	
	We further observe that, for a bounded and non-polynomial $\tilde{\rho}\in\C(\R)$ and 
	a fixed probability measure $\tilde{\nu}\in\M_{\kappa_0}$ having full support, the subset
	\[ \tilde{\H}:=\left\{\C_{\kappa_0}\to\R,\; 
	f\mapsto\int_{\Rd}f(x)\varphi(x)\,\tilde{\nu}(\d x)+b\colon
	\varphi\in\cN^{\tilde{\rho}}_{\Rd,\R}\right\} \]
	also satisfies Assumption~\ref{ass:cn_operator}(i).\ Note that $\tilde{\H}$ can be
	implemented on a computer by approximating the integral with samples from $\tilde{\nu}$.
	Here, the space $\cN^{\tilde{\rho}}_{\Rd,\R}$ consists of all neural networks
	\[ \varphi\colon\Rd\to\R,\; x \mapsto\sum_{k=1}^K y_k\tilde{\rho}\left(a_k^\top x+b_k\right), \]
	where $K\in\N$ is the number of neurons, $a_1,\ldots,a_K\in\Rd$ are weights, 
	$b_1,\ldots,b_K\in\R$ are biases, $y_1,\ldots,y_K\in\R$ are linear readouts and 
	$\tilde{\rho}$ is a (bounded) activation function.\ Indeed, for every $f,g\in\C_{\kappa_0}$
	with $f \neq g$, we define $u:=f-g \neq 0$, $c:=\int_{\Rd}|u|\,\d\tilde{\nu}>0$ and
	$\d\mu:=\frac{|u|}{c}\d\tilde{\nu}$. Since $\cN^{\tilde{\rho}}_{\Rd,\R}$ is
	dense in $L^1(\mu)$ by~\cite[Proposition~1]{leshno93}, there is $\varphi\in\cN^{\tilde{\rho}}_{\Rd,\R}$
	with $\|\varphi-\sgn(u(\cdot))\|_{L^1(\mu)}<1$. Hence,
	\[ \left|\frac{1}{c}\int_{\Rd}u(x)\varphi(x)\,\tilde{\nu}(\d x)-1\right|
	=\left|\int_{\Rd}\sgn(u(x))(\varphi(x)-\sgn(u(x)))\,\mu(\d x)\right| 
	\leq\|\varphi-\sgn(u(\cdot))\|_{L^1(\mu)}<1 \] 
	and therefore $\int_{\Rd}f(x)\varphi(x)\,\tilde{\nu}(\d x)\neq\int_{\Rd}g(x)\varphi(x)\,\tilde{\nu}(\d x)$.
	
	As an example for Assumption~\ref{ass:cn_operator}(iii), we can choose 
	$\L:=\cN^{\tilde{\rho}}_{\Rd,\R}\subset\C^\alpha_\kappa$. For instance, if $\alpha=1$, 
	$\kappa(x)=(1+|x|^q)^{-1}$ with $q>0$ and $\tilde{\rho}\in\C^1(\R)$ is non-polynomial 
	with $\lim_{|x|\to\infty}\frac{|\tilde{\rho}(x)|}{(1+|x|)^q} = 0$, the set 
	$\cN^{\tilde{\rho}}_{\Rd,\R}$ is dense in $\C^1_\kappa$ by~\cite[Theorem~2.7]{neufeld26} 
	and therefore also dense in $\c^{\alpha'}_{\kappa'}$ for all $0\leq\alpha'<1$ 
	and $\K\ni\kappa'\llsim\kappa$, see Theorem~\ref{thm:hol_dense}.
\end{example}

In order to obtain the following universal approximation result, we adapt 
the arguments from~\cite[Theorem~4.13]{cuchiero26} to establish the universality
of Chernoff-neural operators.

\begin{theorem} \label{thm:uat_chernoff}
	Suppose that the Assumptions~\ref{ass:cn_uat} and~\ref{ass:cn_operator} are satisfied.\ 
	Then, for every $k,n\in\N$, $0\leq\alpha_l'<\alpha_l\leq\alpha$
	and $\kappa_l'\llsim\kappa_l\in\K$ for $l=1,\ldots,k$ and $\epsilon>0$, there exists 
	$\Phi_n\in\mathcal{CN}^{\H,\rho,\L}_{\C_{\kappa_0},\C^\alpha_\kappa}$ with
	\[ \max_{l=1,\ldots,k}\sup_{f\in\C^{\alpha_l}_{\kappa_l}}
	\frac{\|I_n f-\Phi_n f\|_{\alpha_l',\kappa_l'}}{(1+\|f\|_{\alpha_l,\kappa_l})^p}<\epsilon.\]
\end{theorem}
\begin{proof}
	First, we show that $I_n\in\B_p(\C^{\alpha_l,\alpha_l'}_{\kappa_l,\kappa_l'};\c^{\alpha_l'}_{\kappa_l'})$ 
	for all $n\in\N$ and $l = 1,\ldots,k$.\ To do so, for every $R\in\N$, we define 
	$I_n^{(R)}:=T^{(R)}\circ I_n\colon\C^{\alpha_l,\alpha_l'}_{\kappa_l,\kappa_l'}\to\c^{\alpha_l'}_{\kappa_l'}$, 
	where 
	\[ T^{(R)}\colon\c^{\alpha_l'}_{\kappa_l'}\to\c^{\alpha_l'}_{\kappa_l'},\; g\mapsto
	\begin{cases}
		g & \text{if } \|g\|_{\alpha_l',\kappa_l'}\leq R, \\
		\frac{R}{\|g\|_{\alpha_l',\kappa_l'}}\, g & \text{if } \|g\|_{\alpha_l',\kappa_l'}>R.
	\end{cases} \]
	Assumption~\ref{ass:cn_uat}(iii) and Theorem~\ref{thm:hol_cont} imply that 
	$I_n^{(R)}:=T^{(R)}\circ I_n\colon\C^{\alpha_l,\alpha_l'}_{\kappa_l,\kappa_l'}\to\c^{\alpha_l'}_{\kappa_l'}$ 
	is well-defined and continuous, while the definition of $T^{(R)}$ ensures that $I_n^{(R)}$ 
	is uniformly bounded on $\C^{\alpha_l,\alpha_l'}_{\kappa_l,\kappa_l'}$ which implies that 
	$I_n^{(R)}\in\Cb(\C^{\alpha_l,\alpha_l'}_{\kappa_l,\kappa_l'};\c^{\alpha_l'}_{\kappa_l'})$.\ 
	Moreover, we use Assumption~\ref{ass:cn_uat}(ii), Theorem~\ref{thm:hol_cont} with $c\geq 0$ 
	denoting the operator norm of the embedding $\C^{\alpha_l}_{\kappa_l} \hookrightarrow \C^{\alpha_l'}_{\kappa_l'}$
	and $p>1$ to obtain
	\begin{align*}
		\lim_{R\to\infty}\|I_n-I_n^{(R)}\|_{\B_p(\C^{\alpha_l,\alpha_l'}_{\kappa_l,\kappa_l'};\c^{\alpha_l'}_{\kappa_l'})} 
		&=\lim_{R\to\infty}\sup_{f\in\C^{\alpha_l}_{\kappa_l}}
		\frac{\|I_nf-I_n^{(R)}f\|_{\alpha_l',\kappa_l'}}{(1+\|f\|_{\alpha_l,\kappa_l})^p} \\
		&\leq\lim_{R\to\infty}\sup_{\|I_nf\|_{\alpha_l',\kappa_l'}\geq R} 
		\frac{\|I_n f\|_{\alpha_l',\kappa_l'}+\|I_n^{(R)}f\|_{\alpha_l',\kappa_l'}}{(1+\|f\|_{\alpha_l,\kappa_l})^p} \\
		&\leq\lim_{R\to\infty}\sup_{e^{\omega h_n}\|f\|_{\alpha_l',\kappa_l'}\geq R} 
		\frac{e^{\omega h_n}\|f\|_{\alpha_l',\kappa_l'}+R}{(1+\|f\|_{\alpha_l,\kappa_l})^p} \\
		&\leq 2 c e^{\omega h_n} \lim_{R\to\infty}\sup_{ce^{\omega h_n}\|f\|_{\alpha_l,\kappa_l}\geq R} \frac{\|f\|_{\alpha_l,\kappa_l}}{(1+\|f\|_{\alpha_l,\kappa_l})^p}=0.
	\end{align*}
	This shows that $I_n\in\B_p(\C^{\alpha_l,\alpha_l'}_{\kappa_l,\kappa_l'};\c^{\alpha_l'}_{\kappa_l'})$.\
	Following the arguments in the proof of~\cite[Theorem~4.13]{cuchiero26}, one can further show that 
	$y\rho(h(\cdot))\in\B_p(\C^{\alpha_l,\alpha_l'}_{\kappa_l,\kappa_l'};\c^{\alpha_l'}_{\kappa_l'})$ to obtain 
	$\mathcal{CN}^{\H,\rho,\L}_{\C_{\kappa_0},\C^\alpha_\kappa}\subset 
	\B_p(\C^{\alpha_l,\alpha_l'}_{\kappa_l,\kappa_l'};\c^{\alpha_l'}_{\kappa_l'})$.
	
	Second, we define $\A:=\linspan(\{\cos(h(\cdot))\colon h\in\H\}\cup\{\sin(h(\cdot))\colon h\in\H\})$
	and show that the vector space $\W:=\linspan(\{y a(\cdot)\colon a\in\A, y\in\L\})$ is contained 
	in the closure of $\mathcal{CN}^{\H,\rho,\L}_{\C_{\kappa_0},\C^\alpha_\kappa}$ w.r.t.
	\begin{equation} \label{eq:thm:uat_chernoff:proof1}
		\max_{l=1,\ldots,k}\|T\|_{\B_p(\C^{\alpha_l,\alpha_l'}_{\kappa_l,\kappa_l'};\c^{\alpha_l'}_{\kappa_l'})} 
		:=\max_{l=1,\ldots,k}\sup_{f\in\C^{\alpha_l}_{\kappa_l}}
		\frac{\|Tf\|_{\alpha_l',\kappa_l'}}{(1+\|f\|_{\alpha_l,\kappa_l})^p}.
	\end{equation}
	For $y\in\L$, $h\in\H$, and $\epsilon>0$, we define $c_h:=\max_{l=1,\ldots,k}c_{\alpha_l,\kappa_l,h}$ 
	and $c_y:=\max_{l=1,\ldots,k}\|y\|_{\alpha_l',\kappa_l'}$. Then, by~\cite[Proposition~4.4(A3)]{cuchiero26}
	or~\cite[Theorem~2.7]{neufeld26}, there exists $\varphi_1\in\mathcal{NN}^\rho_{\R,\R}$ with
	\[ \sup_{t\in\R}\frac{|\cos(t)-\varphi_1(t)|}{(1+|t|)^p}<\frac{\epsilon}{\max(1,c_h^p c_y)}.\]
	Hence, the Chernoff-neural 
	operator $\varphi:=y\varphi_1(h(\cdot))\in\mathcal{CN}^{\H,\rho,\L}_{\C_{\kappa_0},\C^\alpha_\kappa}$
	satisfies
	\begin{align*}
		& \max_{l=1,\ldots,k}\|y\cos(h(\cdot))-\varphi(\cdot)\|_
		{\B_p(\C^{\alpha_l,\alpha_l'}_{\kappa_l,\kappa_l'};\c^{\alpha_l'}_{\kappa_l'})} 
		=\max_{l=1,\ldots,k}\sup_{f\in\C^{\alpha_l}_{\kappa_l}} 
		\frac{\|y\cos(h(f))-y\varphi_1(h(f))\|_{\alpha_l',\kappa_l'}}{(1+\|f\|_{\alpha_l,\kappa_l})^p} \\
		&\leq\Big(\max_{l=1,\ldots,k}\|y\|_{\alpha_l',\kappa_l'}\Big) 
		\bigg(\max_{l=1,\ldots,k}\sup_{f\in\C^{\alpha_l}_{\kappa_l}} \
		\frac{(1+|h(f)|)^p}{(1+\|f\|_{\alpha_l,\kappa_l})^p}\bigg) 
		\sup_{t\in\R}\frac{|\cos(t)-\varphi_1(t)|}{(1+|t|)^p}
		<\frac{c_y c_h^p \epsilon}{\max(1,c_h^p c_y)}\leq\epsilon.
	\end{align*}
	Since the same estimate holds true for $y\sin(h(\cdot))$, the set~$\W$ is contained 
	in the closure of $\mathcal{CN}^{\H,\rho,\L}_{\C_{\kappa_0},\C^\alpha_\kappa}$ 
	w.r.t. the norm defined in equation~\eqref{eq:thm:uat_chernoff:proof1}.
	
	Third, due to some fundamental trigonometric identities, see~\cite[Equation~4.3]{cuchiero26},
	the set~$\A$ is an algebra.\ By definition, the set~$\W$ is an $\A$-submodule.\
	Since~$\H$ is point separating and contains all constant mappings, the set~$\A$ is 
	point separating and nowhere vanishing.\ Finally, by Assumption~\ref{ass:cn_operator}(iii),
	the set $\W(f):=\{w(f)\colon w\in\W\}=\L$ is dense in $\c^{\alpha'}_{\kappa'}$ for all 
	$0\leq\alpha'<\alpha''\leq\alpha$, $\kappa'\llsim\kappa''\in\K$ and $f\in\C^{\alpha''}_{\kappa''}$.
	Hence, the claim follows from Theorem~\ref{thm:stone_wstrass_vector}.
\end{proof}

\subsection{Qualitative approximation of the semigroup}

We now present the main result of this article which shows that iterations of Chernoff-neural 
operators~$\Phi_n\in\mathcal{CN}^{\H,\rho,\L}_{\C_{\kappa_0},\C^\alpha_\kappa}$ 
are able to approximate the strongly continuous convex monotone semigroup $(S(t))_{t\geq 0}$.

\begin{theorem} \label{thm:approx_chernoff}
	Suppose that the Assumptions~\ref{ass:cn_uat} and~\ref{ass:cn_operator} are satisfied.\ 
	Then, for every $\alpha'\in [0,\alpha)$, $\K\ni\kappa'\llsim\kappa$, $T\geq 0$, $n\in\N$ 
	and $\epsilon>0$, there exists $\Phi_n\in\mathcal{CN}^{\H,\rho,\L}_{\C_{\kappa_0},\C^\alpha_\kappa}$ with
	\[ \|I(\pi_n^t)f-\Phi(\pi^t_n)f\|_{\alpha',\kappa'}\leq (1+\|f\|_{\alpha,\kappa})^{p^{k_n^t}}\epsilon
	\quad\text{for all } t\in [0,T] \text{ and } f\in\C^\alpha_\kappa. \]
	Moreover, let Assumption~\ref{ass:chernoff} be valid and let
	$(S(t))_{t \geq 0}$ be the semigroup from Theorem~\ref{thm:chernoff}.\
	Then, for every $r,T\geq 0$, $\alpha'\in [0,\alpha)$, $\K\ni\kappa'\llsim\kappa$ and $\epsilon>0$, 
	there exist $n\in\N$ and $\Phi_n\in\mathcal{CN}^{\H,\rho,\L}_{\C_{\kappa_0},\C^\alpha_\kappa}$ with
	\[ \sup_{f\in B_{\C^\alpha_\kappa}(r)}\sup_{t\in [0,T]}\|S(t)f-\Phi(\pi^t_n)f\|_{\alpha',\kappa'}<\epsilon. \]
\end{theorem}
\begin{proof}
	Let $\alpha'\in [0,\alpha)$, $\kappa'\llsim\kappa\in\K$, $p\in (1,\infty)$, $T\geq 0$, 
	$n\in\N$ and $\epsilon>0$.\ By Assumption~\ref{ass:cn_uat}(i), there exist 
	$\alpha'=\alpha_{k_n^T}<\ldots<\alpha_0=\alpha$ and $\kappa'=\kappa_{k_n^T}\llsim\ldots\llsim\kappa_0=\kappa$.\
	Since Assumption~\ref{ass:cn_uat}(ii) guarantees that $I_n^k f\in\C^\alpha_\kappa$ 
	for all $k\in\N$ and $f\in\C^\alpha_\kappa$, we can further apply 
	Assumption~\ref{ass:cn_uat}(ii)--(iii) and Theorem~\ref{thm:hol_cont} to obtain a 
	constant $c\geq 1$ with 
	\begin{align} 
		\|I_nf\|_{\alpha_l,\kappa_l} &\leq c\|f\|_{\alpha_{l-1},\kappa_{l-1}}, \label{eq:cor:uat_proof2a} \\
		\|I_n^k f-I_n^k g\|_{\alpha_l,\kappa_l}
		&\leq c\|f-g\|_{\alpha_l,\kappa_l} \label{eq:cor:uat_proof2b}
	\end{align}
	for all $f,g\in\C^\alpha_\kappa$ and $k,l\in\{1,\ldots,k_n^T\}$.\ By Theorem~\ref{thm:uat_chernoff},
	there exists $\Phi_n\in\mathcal{CN}^{\H,\rho,\L}_{\C_{\kappa_0},\C^\alpha_\kappa}$ with
	\begin{equation} \label{eq:cor:uat_proof3}
		\max_{1\leq l\leq k^T_n}\sup_{f\in\C^{\alpha_{l-1}}_{\kappa_{l-1}}} 
		\frac{\|I_nf-\Phi_nf\|_{\alpha_l,\kappa_l}}{(1+\|f\|_{\alpha_{l-1},\kappa_{l-1}})^p}
		<\min\left\{1,\frac{\epsilon}{C}\right\},
	\end{equation}
	where $C:=c\sum_{l=1}^{k_n^T}(2c)^\frac{p(p^l-1)}{p-1}$. We use $\Phi_n^{l-1}f\in\C^\alpha_\kappa$,
	inequality~\eqref{eq:cor:uat_proof2a} and inequality~\eqref{eq:cor:uat_proof3} to obtain
	\begin{align*}
		1+\|\Phi_n^l f\|_{\alpha_l,\kappa_l}
		&=1+\|\Phi_n\Phi_n^{l-1}f\|_{\alpha_l,\kappa_l} \\
		&\leq1+\|I_n\Phi_n^{l-1}f\|_{\alpha_l,\kappa_l}
		+\|I_n\Phi_n^{l-1}f-\Phi_n\Phi_n^{l-1}f\|_{\alpha_l,\kappa_l}\\
		&\leq1+c\|\Phi_n^{l-1}f\|_{\alpha_{l-1},\kappa_{l-1}}+(1+\|\Phi_n^{l-1}f\|_{\alpha_{l-1},\kappa_{l-1}})^p \\
		&\leq 2c(1+\|\Phi_n^{l-1}f\|_{\alpha_{l-1},\kappa_{l-1}})^p
	\end{align*}
	for all $l\in\{1,\ldots,k_n^T\}$ and $f\in\C^\alpha_\kappa$. Hence, by induction on 
	$l\in\{1,\ldots,k_n^T\}$, it follows that
	\begin{equation} \label{eq:cor:uat_proof4}
		1+\|\Phi_n^l f\|_{\alpha_l,\kappa_l}\leq(2c)^\frac{p^l-1}{p-1}
		(1+\|f\|_{\alpha,\kappa})^{p^l} \quad\text{for all } f\in\C^\alpha_\kappa.
	\end{equation}
	For every $t\in [0,T]$ and $f\in\C^\alpha_\kappa$, we use $\Phi_n^{l-1}f\in\C^\alpha_\kappa$ 
	and the inequalities~\eqref{eq:cor:uat_proof2b}--\eqref{eq:cor:uat_proof4} to obtain
	\begin{align*}
		\|I(\pi^t_n)f-\Phi(\pi^t_n)f\|_{\alpha',\kappa'}
		&=\big\|I_n^{k^t_n}f-\Phi_n^{k^t_n}f\big\|_{\alpha',\kappa'} 
		\leq\sum_{l=1}^{k^t_n}\big\|I_n^{k^t_n-l}I_n\Phi_n^{l-1}f
		-I_n^{k^t_n-l}\Phi_n\Phi_n^{l-1}f\big\|_{\alpha',\kappa'} \\
		&\leq c\sum_{l=1}^{k^t_n}\|I_n\Phi_n^{l-1}f-\Phi_n\Phi_n^{l-1}f\|_{\alpha_l,\kappa_l} 
		\leq \frac{c\epsilon}{C}\sum_{l=1}^{k^t_n}(1+\|\Phi_n^{l-1}f\|_{\alpha_{l-1},\kappa_{l-1}})^p \\
		&\leq \frac{c\epsilon}{C}\sum_{l=1}^{k_n^t}(2c)^\frac{p(p^l-1)}{p-1}(1+\|f\|_{\alpha,\kappa})^{p^l}
		\leq (1+\|f\|_{\alpha,\kappa})^{p^{k_n^t}}\epsilon.
	\end{align*}
	The second part follows from the first part and Corollary~\ref{cor:chernoff}.
\end{proof}

\section{Envelope-neural operators}
\label{sec:nisio}

We now introduce a more specific neural operator architecture for the case that 
the one-step operators are given as the supremum over a family of linear operators.\ 
This covers, for instance, the important class of Nisio semigroups which have 
originally been introduced in~\cite{Nisio76} and adapted to transition semigroups 
of stochastic processes under model uncertainty in~\cite{DKN20,nendel21}.\ In the
sequel, we fix $\alpha\in (0,1]$, a bounded continuous function $\kappa\colon\Rd\to (0,\infty)$ 
satisfying inequality~\eqref{eq:weight_cond}, an index set~$\Lambda$, a family 
$\{I_{n,\lambda}\colon n\in\N,\,\lambda\in\Lambda\}$ of operators $I_{n,\lambda}\colon\Ck\to\Ck$, 
penalization functions $\eta,\eta_n\colon\Lambda\to[0,\infty]$ and a sequence 
$(h_n)_{n\in\N}\subset (0,\infty)$  with $h_n\to 0$. We define
\begin{equation} \label{eq:nisio}
	(I_nf)(x):=\sup_{\lambda\in\Lambda}\big((I_{n,\lambda}f)(x)-\eta_n(\lambda)h_n\big) 
\end{equation}
for all $n\in\N$, $f\in\Ck$ and $x\in\R^d$. Moreover, for every $\lambda\in\Lambda$, 
we define
\[ A_\lambda\colon D(A_\lambda)\to\Ck,\; f\mapsto\lim_{n\rightarrow\infty}\frac{I_{n,\lambda}f-f}{h_n}, \] 
where the domain $D(A_\lambda)$ consists of all $f\in\Ck$ such that the previous limit exists.\
To guarantee that $I_n\colon\Ck\to\Ck$ and that Assumption~\ref{ass:chernoff} is satisfied, 
we impose the following conditions.

\begin{assumption} \label{ass:nisio}
	Suppose that the conditions~(i)--(vii) or that the conditions~(i)--(vi) and~(vii') 
	of the following list are satisfied:
	\begin{enumerate}
		\item For every $n\in\N$, there exists $\lambda_{0,n}\in\Lambda$ with $\eta_n(\lambda_{0,n})=0$.
		\item The operators $I_{n,\lambda}$ are linear and monotone
		for all $n\in\N$ and $\lambda\in\Lambda$.
		\item There exists $\omega\geq 0$ with $\|I_{n,\lambda}f\|_\kappa\leq e^{\omega h_n}\|f\|_\kappa$
		for all $n\in\N$, $\lambda\in\Lambda$ and $f\in\Ck$.
		\item For every $\epsilon>0$, there exists $\delta>0$ and $n_0\in\N$ with
		\[ I_{n,\lambda}(\tau_x f)\leq\tau_x (I_{n,\lambda}f)+\frac{r\epsilon h_n}{\kappa} \]
		for all $n\geq n_0$, $r\geq 0$, $f\in\Lipb(r)$, $x\in B_{\Rd}(\delta)$ and $\lambda\in\Lambda$.
		\item For every $\epsilon>0$, $r,T\geq 0$ and $K\Subset\Rd$, there exist
		$c\geq 0$ and $K'\Subset\Rd$ with 
		\[ \|I(\pi^t_n)f-I(\pi^t_n)g\|_{\infty,K}\leq c\|f-g\|_{\infty,K'}+\epsilon \]
		for all $t\in [0,T]$, $n\in\N$ and $f,g\in B_{\Ck}(r)$.
		\item It holds $I_{n,\lambda}\colon\Lipb(r)\to\Lipb(e^{\omega h_n}r)$ 
		for all $n \in \N$, $\lambda\in\Lambda$ and $r\geq 0$.
		\item The limit $I'(0)f\in\Ck$ exists for all $f\in\Cbi$.
		\item[(vii')] It holds $\Cbi\subset D(A_\lambda)$ for all $\lambda\in\Lambda$.\ In addition, 
		for every $f\in\Cbi$, there exist $c\geq 0$, $n_0\in\N$ and $\Lambda_0\subset\Lambda$ 
		such that the following statements are true:
		\begin{enumerate}[label=(\alph*)]
			\item $\sup_{\lambda\in\Lambda}(I_{n,\lambda}f-\eta_n(\lambda)h_n)
			=\sup_{\lambda\in\Lambda_0}(I_{n,\lambda}f-\eta_n(\lambda)h_n)$ for all $n\geq n_0$,
			\item $\sup_{\lambda\in\Lambda}(A_\lambda f-\eta(\lambda))
			=\sup_{\lambda\in\Lambda_0}(A_\lambda f-\eta(\lambda))$,
			\item $\sup_{n\in\N}\sup_{\lambda\in\Lambda_0}\eta_n(\lambda)<\infty$,
			$\sup_{\lambda\in\Lambda_0}\|I_{n,\lambda}f-f\|_\kappa\leq ch_n$ for all $n\geq n_0$ and 
			\[ \lim_{n\rightarrow\infty}\sup_{\lambda\in\Lambda_0}
			\left\|\frac{I_{n,\lambda}f-f}{h_n}-A_\lambda f\right\|_{\infty,K}=0 
			\quad\text{for all } K\Subset\Rd. \]
			Furthermore, it holds $\lim_{n\to\infty}\sup_{\lambda\in\Lambda_0}|\eta_n(\lambda)-\eta(\lambda)|=0$.
		\end{enumerate}
	\end{enumerate}
\end{assumption}

Note that condition~(vii') immediately implies condition~(vii).

\begin{theorem} \label{thm:nisio}
	Suppose that the Assumptions~\ref{ass:nisio}(i)--(vii) are valid.\ Then, there 
	exists a strongly continuous convex monotone semigroup $(S(t))_{t \geq 0}$ on~$\Ck$ 
	given by
	\[ S(t)f:=\lim_{n\to\infty}I(\pi_n^t)f \quad\text{for all } t\geq 0 \text{ and } f\in\Ck \]
	whose generator satisfies $Af=I'(0)f$ for all $f\in\Cbi$.\ If the 
	Assumption~\ref{ass:nisio}(i)--(vi) and~(vii') are valid, we further obtain
	\[ Af=\sup_{\lambda\in\Lambda}\big(A_\lambda f-\eta(\lambda)\big) \quad\text{for all } f\in\Cbi. \]
\end{theorem}
\begin{proof}
	This is a particular case of Theorem~\ref{thm:chernoff}.
\end{proof}

\begin{corollary} \label{cor:nisio}
	Let Assumption~\ref{ass:nisio} be satisfied and denote by $(S(t))_{t\geq 0}$
	the strongly continuous convex monotone semigroup from Theorem~\ref{thm:nisio}.
	Assume that there exists $\tilde{\omega}\geq 0$ with
	\[ \|I_{n,\lambda}f\|_{\alpha,\kappa}\leq e^{\tilde{\omega}h_n}\|f\|_{\alpha,\kappa} 
	\quad\text{for all } n\in\N, \lambda\in\Lambda \text{ and } f\in \C^\alpha_\kappa. \]
	Then, it holds $\|S(t)f\|_{\alpha,\kappa}\leq e^{\tilde{\omega}t}\|f\|_{\alpha,\kappa}$
	for all $t\geq 0$ and $f\in\C^\alpha_\kappa$. Furthermore,  
	\[ \lim_{n\to\infty}\sup_{f\in B_{\C^\alpha_\kappa}(r)}\sup_{t\in [0,T]}
	\|S(t)f-I(\pi_n^t)f\|_{\alpha',\kappa'}=0 \]
	for all $r,T\geq 0$, $\alpha'\in [0,\alpha)$ and $\kappa'\llsim\kappa$.
\end{corollary}
\begin{proof}
	This is a particular case of Corollary~\ref{cor:chernoff}.
\end{proof}

\subsection{Universal approximation}

In order to approximate the operators $(I_n)_{n\in\N}$ defined by 
equation~\eqref{eq:nisio}, we replace the supremum over $\lambda\in\Lambda$ 
by the maximum over finitely many trainable parameters 
$\lambda_0,\ldots,\lambda_M\in\Lambda$ which can be implemented using
a ReLU-neural network, see Figure~\ref{fig:nn}.

\begin{definition}
	For every $n\in\N$, an envelope-neural operator is a mapping of the form
	\[ \Phi_n\colon\Ck\to\Ck,\;
	f\mapsto\varphi\big((I_{n,\lambda_0} f)(\cdot),(I_{n,\lambda_1} f)(\cdot),\ldots,(I_{n,\lambda_M} f)(\cdot)\big), \]
	where $\lambda_0:=\lambda_{0,n}\in\Lambda$ is the fixed parameter from Assumption~\ref{ass:nisio}(i), 
	$\lambda_1,\ldots,\lambda_M\in\Lambda$ are the (hidden-layer) parameters and 
	$\varphi\in\mathcal{NN}^{\ReLU}_{\R^{M+1},\R}$ is a ReLU-neural network. 
	We collect all operators of this form in the set $\mathcal{EN}^{\ReLU}_{\Ck,\Ck}$.
\end{definition}

\begin{figure}[ht]
	\centering
	\begin{tikzpicture}[
		inputnode/.style={circle, draw=green!60, fill=green!5, very thick, minimum size=4mm},
		hiddennode/.style={circle, draw=blue!60, fill=blue!5, very thick, minimum size=4mm},
		outputnode/.style={circle, draw=red!60, fill=red!5, very thick, minimum size=4mm},
		node distance=7mm,
		]
		\node[inputnode] (x1) {};
		\node[hiddennode] (y2) [right = 2cm of x1] {};
		\node[hiddennode] (y1) [above of = y2] {};
		\node[hiddennode] (y3) [below of = y2] {};
		\node[outputnode] (o1) [right = 2cm of y2] {};
		
		\draw[shorten >=0.1cm,shorten <=0.1cm, ->] (x1.east) -- (y1.west);	
		\draw[shorten >=0.1cm,shorten <=0.1cm, ->] (x1.east) -- (y2.west);
		\draw[shorten >=0.1cm,shorten <=0.1cm, ->] (x1.east) -- (y3.west);
		\draw[shorten >=0.1cm,shorten <=0.1cm, ->] (y1.east) -- (o1.west);
		\draw[shorten >=0.1cm,shorten <=0.1cm, ->] (y2.east) -- (o1.west);
		\draw[shorten >=0.1cm,shorten <=0.1cm, ->] (y3.east) -- (o1.west);
		
		\draw[] (-0.1,0.9) node[anchor=center, align=center] {\footnotesize Input Layer};
		\draw[] (0,0.6) node[anchor=center, align=center] {\footnotesize $\Ck$};
		\draw[] (2.45,1.5) node[anchor=center, align=center] {\footnotesize Hidden Layer};
		\draw[] (2.45,1.2) node[anchor=center, align=center] {\footnotesize $\Ck$};
		\draw[] (4.85,0.9) node[anchor=center, align=center] {\footnotesize Output Layer};
		\draw[] (4.85,0.6) node[anchor=center, align=center] {\footnotesize $\Ck$};
		\draw[] (-0.85,0) node[anchor=center, align=center] {\footnotesize $f\in\C^\alpha_\kappa$};
		\draw[] (6.0,0) node[anchor=center, align=center] {\footnotesize $\Phi(f)\in\C^\alpha_\kappa$};
		\draw[] (1.1,-0.7) node[anchor=center, align=center] {\footnotesize $I_{n,\lambda}$};
		\draw[] (3.6,-0.7) node[anchor=center, align=center] {\footnotesize $\varphi$};
	\end{tikzpicture}
	\caption{An envelope-neural operator $\Phi_n\colon\Ck\to\Ck$ 
		with hidden-layer maps from the family $(I_{n,\lambda})_{\lambda\in\Lambda}$, 
		ReLU-neural network $\varphi\in\mathcal{NN}^{\ReLU}_{\R^3,\R}$ as linear readout
		and $M+1=3$ neurons.}
	\label{fig:nn}
\end{figure}
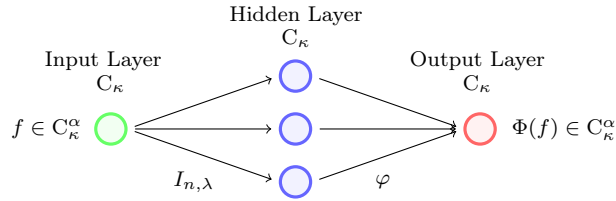

\begin{assumption} \label{ass:en}
	Suppose that the following conditions are satisfied:
	\begin{enumerate}
		\item It holds $\|I_{n,\lambda}f\|_{\alpha,\kappa}\leq e^{\omega h_n}\|f\|_{\alpha,\kappa}$
		for all $n\in\N$, $\lambda\in\Lambda$ and $f\in\C^\alpha_\kappa$.
		\item For every $r\geq 0$, there exists a totally bounded metric space $(\Lambda_r,d_r)$
		with $\Lambda_r\subset\Lambda$ such that the set $\Lambda_{r,n}:=\Lambda_r\cap\{\eta_n<\infty\}$ 
		satisfies
		\[ \sup_{\lambda\in\Lambda}\,(I_{n,\lambda}f-\eta_n(\lambda)h_n) 
		=\sup_{\lambda\in\Lambda_{r,n}}(I_{n,\lambda}f-\eta_n(\lambda)h_n)
		\quad \text{for all } n\in\N \text{ and } f\in B_{\C^\alpha_\kappa}(r). \]
		In addition, for every $r\geq 0$ and $\epsilon>0$, there exists $\delta>0$ such that
		\[ \|(I_{n,\lambda_1}f-\eta_n(\lambda_1)h_n)-(I_{n,\lambda_2}f-\eta_n(\lambda_2)h_n)\|_\kappa<\epsilon \]
		for all $n\in\N$, $f\in B_{\C^\alpha_\kappa}(r)$ and $\lambda_1,\lambda_2\in\Lambda_{r,n}$
		with $d_r(\lambda_1,\lambda_2)<\delta$.\ Finally, the element~$\lambda_{0,n}$
		from Assumption~\ref{ass:nisio}(i) satisfies $\lambda_{0,n}\in\Lambda_r$ for all $n\in\N$.
	\end{enumerate}
\end{assumption}

\begin{theorem} \label{thm:uat_nisio}
	Suppose that the Assumptions~\ref{ass:nisio} and \ref{ass:en} are satisfied.\ 
	Then, for every $n\in\N$, $r\geq 0$, $\alpha'\in [0,\alpha)$ and $\epsilon>0$,
	there exists $\Phi_n\in\mathcal{EN}^{\ReLU}_{\Ck,\Ck}$ with 
	\[ \sup_{f\in B_{\C^\alpha_\kappa}(r)} \|I_n f-\Phi_n f\|_{\alpha',\kappa}<\epsilon. \]
\end{theorem}
\begin{proof}
	Let $\epsilon>0$ and $\epsilon'\in (0,\epsilon)$ with
	$(2e^{\omega h_n}r)^{\alpha'/\alpha}(2\epsilon')^{1-\alpha'/\alpha}<\epsilon/2$.
	Moreover, let~$\Lambda_{r,n}$ and $\delta>0$ satisfy Assumption~\ref{ass:en}(ii) with $\epsilon'$.\
	Since $\Lambda_{r,n}$ is also totally bounded, there exist $\lambda_1,\ldots,\lambda_M\in\Lambda_{r,n}$ with 
	$\Lambda_{r,n}\subset\bigcup_{m=1}^M \mathring{B}_{d_r}(\lambda_m,\delta)$. In particular, 
	choosing $m(\lambda)\in\argmin_{m=1,\ldots,M}d_r(\lambda,\lambda_m)$ yields 
	\[ \|(I_{n,\lambda}f-\eta_n(\lambda)h_n)-(I_{n,\lambda_{m(\lambda)}}f-\eta_n(\lambda_{m(\lambda)})h_n)\|_\kappa
	<\epsilon'<\epsilon \quad\text{for all } f\in B_{\C^\alpha_\kappa}(r). \]
	We implement the function $\R^{M+1}\to\R,\; y \mapsto\max_{m=0,\ldots,M}(y_m-\eta_n(\lambda_m)h_n)$
	with $\lambda_0:=\lambda_{0,n}$ as ReLU-neural network $\varphi\in\mathcal{NN}^{\ReLU}_{\R^{M+1},\R}$ 
	and define $\Phi_n\in\mathcal{EN}^{\ReLU}_{\Ck,\Ck}$ by 
	$\Phi_n(f):=\varphi\circ I_{n,\lambda_{0:M}}f$, where 
	$I_{n,\lambda_{0:M}}f:=((I_{n,\lambda_0}f)(\cdot),\ldots,(I_{n,\lambda_M}f)(\cdot))$. Then,
	for every $f\in B_{\C^\alpha_\kappa}(r)$, it holds
	\begin{equation} \label{eq:thm:uat_nisio:proof1}
		\|I_n f-\Phi_n f\|_\kappa 
		=\Big\|\sup_{\lambda\in\Lambda_{r,n}}(I_{n,\lambda}f-\eta_n(\lambda)h_n)
		-\max_{m=0,\ldots,M}(I_{n,\lambda_m}f-\eta_n(\lambda_m)h_n)\Big\|_\kappa
		<\epsilon'<\epsilon.
	\end{equation}
	Let $f\in B_{\C^\alpha_\kappa}(r)$ and define $g_n:=I_n f-\varphi\circ I_{n,\lambda_{0:M}}f$.
	Assumption~\ref{ass:en}(i) implies
	\[ \sup_{x\neq y}\frac{|g_n(x)-g_n(y)|}{|x-y|^\alpha}\bar{\kappa}(x,y)
	\leq\sup_{\lambda\in\Lambda}\|I_{n,\lambda}f\|_{\alpha,\kappa} 
	+\max_{m=0,\ldots,M}\|I_{n,\lambda_m}f\|_{\alpha,\kappa} 
	\leq 2e^{\omega h_n}\|f\|_{\alpha,\kappa}\leq 2e^{\omega h_n}r. \]
	Hence, it follows from inequality~\eqref{eq:thm:uat_nisio:proof1} and $\bar{\kappa}(x,y)\leq 2\min(\kappa(x),\kappa(y))$ that
	\begin{align*}
		\sup_{x\neq y}\frac{|g_n(x)-g_n(y)|}{|x-y|^{\alpha'}}\bar{\kappa}(x,y) 
		&\leq\sup_{x\neq y}\left(\frac{|g_n(x)-g_n(y)|}{|x-y|^\alpha} \bar{\kappa}(x,y)\right)^\frac{\alpha'}{\alpha} 
		\big(|g_n(x)-g_n(y)|\bar{\kappa}(x,y)\big)^{1-\frac{\alpha'}{\alpha}} \\
		&\leq (2e^{\omega h_n}r)^\frac{\alpha'}{\alpha}(2\epsilon')^{1-\frac{\alpha'}{\alpha}}
		<\epsilon. \qedhere
	\end{align*}
\end{proof}

The following theorem states that iterations of an envelope-neural operator~$\Phi_n$
are able to approximate the strongly continuous convex monotone semigroup $(S(t))_{t\geq 0}$.

\begin{theorem} \label{thm:approx_nisio}
	Suppose that the Assumptions~\ref{ass:nisio} and~\ref{ass:en} are satisfied.\
	Then, for every $r,T\geq 0$, $n\in\N$, $\alpha' \in [0,\alpha)$ and $\epsilon>0$, 
	there exists $\Phi_n\in\mathcal{EN}^{\ReLU}_{\Ck,\Ck}$ with
	\[ \sup_{f\in B_{\C^\alpha_\kappa}(r)}\sup_{t\in [0,T]}\|I(\pi^t_n)f-\Phi(\pi^t_n)f\|_{\alpha',\kappa}<\epsilon. \]
	Denote by $(S(t))_{t\geq 0}$ the semigroup from Theorem~\ref{thm:nisio}.\ Then, 
	for every $r,T\geq 0$, $\alpha'\in [0,\alpha)$, $\kappa'\llsim\kappa$ and $\epsilon>0$, 
	there exist $n\in\N$ and $\Phi_n\in\mathcal{EN}^{\ReLU}_{\Ck,\Ck}$ with
	\[ \sup_{f \in B_{\C^\alpha_\kappa}(r)} \sup_{t \in [0,T]} \|S(t)f-\Phi(\pi^t_n)f\|_{\alpha',\kappa'}<\epsilon. \]
\end{theorem}
\begin{proof}
	Let $r,T\geq 0$, $n\in\N$, $\alpha'\in [0,\alpha)$, $\epsilon>0$ and $\epsilon'\in (0,\epsilon)$ 
	with $(2e^{\omega T}r)^{\alpha'/\alpha}(2\epsilon')^{1-\alpha'/\alpha}<\epsilon$.
	Assumption~\ref{ass:nisio}(iii) guarantees that
	\[ \|I_{n,\lambda}f\|_\kappa\leq e^{\omega h_n}\|f\|_\kappa \quad\text{and}\quad
	\|I_n^k f-I_n^k g\|_\kappa\leq e^{\omega kh_n}\|f-g\|_\kappa \]
	for all $f,g\in\C^\alpha_\kappa$, $k\in\{1,\ldots,k^T_n\}$ and $\lambda\in\Lambda$.
	By Theorem~\ref{thm:uat_nisio}, there exists $\Phi_n\in\mathcal{EN}^{\ReLU}_{\Ck,\Ck}$ with
	\[ \|I_n f-\Phi_n f\|_\kappa<\frac{\epsilon'}{e^{\omega T}k^T_n}
	\quad\text{for all } f\in B_{\C^\alpha_\kappa}(e^{\omega T} r), \]
	where $\Phi_n f:=\max_{m=0,\ldots,M}(I_{n,\lambda_m}f-\eta_n(\lambda_m)h_n)$
	with $\lambda_0 := \lambda_{0,n}$ and some $\lambda_1,\ldots,\lambda_M\in\Lambda$.\
	For every $t\in [0,T]$ and $f\in B_{\C^\alpha_\kappa}(r)$, we use 
	$\|\Phi_n^{l-1} f\|_\kappa\leq\|\Phi_n^{l-1} f\|_{\alpha,\kappa}\leq e^{\omega (l-1)h_n}r\leq e^{\omega T}r$ to obtain
	\begin{align} 
		\|I(\pi^t_n)f-\Phi(\pi^t_n)f\|_\kappa
		&=\big\|I_n^{k^t_n}f-\Phi_n^{k^t_n}f\big\|_\kappa
		\leq\sum_{l=1}^{k^t_n}\big\|I_n^{k^t_n-l}I_n\Phi_n^{l-1}f-I_n^{k^t_n-l}\Phi_n\Phi_n^{l-1}f\big\|_\kappa 
		\nonumber \\
		&\leq e^{\omega T}\sum_{l=1}^{k^t_n}\|I_n\Phi_n^{l-1}f-\Phi_n\Phi_n^{l-1}f\|_\kappa
		<\frac{e^{\omega T}k^t_n\epsilon'}{e^{\omega T} k^T_n}\leq\epsilon'.
		\label{eq:approx_nisio:uat_proof2}
	\end{align}
	Let $f\in B_{\C^\alpha_\kappa}(r)$ and define $g^t_n:=I(\pi^t_n)f-\Phi(\pi^t_n)f$.
	It holds 
	\[ \sup_{x\neq y}\frac{|g^t_n(x)-g^t_n(y)|}{|x-y|^\alpha}\bar{\kappa}(x,y)
	\leq\|I(\pi^t_n)f\|_{\alpha,\kappa}+\|\Phi(\pi^t_n)f\|_{\alpha,\kappa} 
	\leq 2e^{\omega k^t_n h_n}\|f\|_{\alpha,\kappa}\leq 2e^{\omega T}r. \] 
	Hence, it follows from inequality~\eqref{eq:approx_nisio:uat_proof2} and $\bar{\kappa}(x,y)\leq 2\min(\kappa(x),\kappa(y))$ that
	\begin{align*}
		\sup_{x\neq y}\frac{|g^t_n(x)-g^t_n(y)|}{|x-y|^{\alpha'}}\bar{\kappa}(x,y) 
		&\leq\sup_{x\neq y}\left(\frac{|g^t_n(x)-g^t_n(y)|}{|x-y|^\alpha}\bar{\kappa}(x,y)\right)^\frac{\alpha'}{\alpha}
		\big(|g^t_n(x)-g^t_n(y)|\bar{\kappa}(x,y)\big)^{1-\frac{\alpha'}{\alpha}} \\
		&\leq (2e^{\omega T}r)^\frac{\alpha'}{\alpha}(2\epsilon')^{1-\frac{\alpha'}{\alpha}}<\epsilon.
	\end{align*}
	This shows the first part of the claim and the second part follows from Corollary~\ref{cor:nisio}.
\end{proof}

It remains to quantify the approximation error in terms of the number of neurons $M\in\N$.

\begin{assumption} \label{ass:en_rate}
	For every $r\geq 0$, there exists a totally bounded metric space $(\Lambda_r,d_r)$ with 
	$\Lambda_r\subset\Lambda$ such that the set $\Lambda_{r,n}:=\Lambda_r\cap\{\eta_n<\infty\}$ 
	satisfies
	\[ \sup_{\lambda\in\Lambda}\,(I_{n,\lambda}f-\eta_n(\lambda)h_n) 
	=\sup_{\lambda\in\Lambda_{r,n}}(I_{n,\lambda}f-\eta_n(\lambda)h_n)
	\quad \text{for all } n\in\N \text{ and } f\in B_{\C^\alpha_\kappa}(r). \]
	In addition, for every $r\geq 0$, there exist $L_r,\beta_r\geq 0$ with
	\[ \|(I_{n,\lambda_1}f-\eta_n(\lambda_1)h_n)-(I_{n,\lambda_2}f-\eta_n(\lambda_2)h_n)\|_\kappa 
	\leq L_r d_r(\lambda_1,\lambda_2)^{\beta_r} \]
	for all $n\in\N$, $f\in B_{\C^\alpha_\kappa}(r)$ and $\lambda_1,\lambda_2\in\Lambda_{r,n}$.\
	Finally, the element~$\lambda_{0,n}$ from Assumption~\ref{ass:nisio}(i) satisfies 
	$\lambda_{0,n}\in\Lambda_r$ for all $n\in\N$.
\end{assumption}

For every $M\in\N$ and $r\geq 0$, the \emph{fill-in distance} of $(\Lambda_r,d_r)$ 
is defined by 
\[ \delta_r(M):=\inf_{\lambda_1,\ldots,\lambda_M\in\Lambda_r} 
\sup_{\lambda\in\Lambda_r}\min_{m=1,\ldots,M}d_r(\lambda,\lambda_m). \] 
For example, if $\Lambda_r\subset\R^m$ and $d_r$ is equivalent to the
Euclidean metric, then $\delta_r(M)\asymp M^{-1/m}$.

\begin{theorem} \label{thm:rate_nisio}
	Suppose that the Assumptions~\ref{ass:nisio},~\ref{ass:en}(i) and~\ref{ass:en_rate} 
	are satisfied.\ Then, for every $r\geq 0$, $\alpha'\in [0,\alpha)$ and $M,n\in\N$, 
	there exists $\Phi_{n,M}\in\mathcal{EN}^{\ReLU}_{\Ck,\Ck}$ with~$M+1$
	neurons such that
	\[ \|I_n f-\Phi_{n,M}f\|_{\alpha',\kappa}
	\leq\max\left\{L_r(4\delta_r(M))^{\beta_r},
	\big(C_{r,\alpha',\kappa}(4\delta_r(M))^{\beta_r}\big)^{1-\frac{\alpha'}{\alpha}}\right\}
	\quad\text{for all } f\in B_{\C^\alpha_\kappa}(r), \]
	where $C_{r,\alpha',\kappa}>0$ is a constant depending only on $r$, $\alpha'$ and $\kappa$.\
	Furthermore, the output ReLU-neural network of $\Phi_{n,M}$ has depth $\cO(\log_2(M))$ and width $\cO(M)$.
\end{theorem}
\begin{proof}
	Since $(\Lambda_r,d_r)$ is totally bounded, by definition of the infimum in $\delta_r(M)$, there exist some $\tilde{\lambda}_1,\ldots,\tilde{\lambda}_M\in\Lambda_r$ with $\sup_{\lambda\in\Lambda_r}\min_{m=1,...,M}d_r(\lambda,\tilde{\lambda}_m)\leq 2\delta_r(M)$. For every $m\in\{1,...,M\}$ with $\Lambda_{r,n} \cap B_{d_r}(\tilde{\lambda}_m,2\delta_r(M)) \neq \emptyset$, choose some $\lambda_m \in \Lambda_{r,n} \cap B_{d_r}(\tilde{\lambda}_m,2\delta_r(M))$. By repeating selected points if necessary, we obtain $\lambda_1,...,\lambda_M \in \Lambda_{r,n}$, for which the triangle inequality ensures
	\[ \sup_{\lambda\in\Lambda_{r,n}}\min_{m=1,\ldots,M}d_r(\lambda,\lambda_m)\leq 4\delta_r(M).\]
	By~\cite[Lemma~5.11]{petersen24}, we can implement 
	$\R^{M+1}\to\R,\; y\mapsto\max_{m=0,\ldots,M}(y_m-\eta_n(\lambda_m)h_n)$ with $\lambda_0:=\lambda_{0,n}$
	as ReLU network $\varphi\in\mathcal{NN}^{\ReLU}_{\R^{M+1},\R}$ of depth~$\cO(\log_2(M))$ 
	and width~$\cO(M)$.\ We define $\Phi_{n,M}\in\mathcal{EN}^{\ReLU}_{\Ck,\Ck}$ 
	by $\Phi_{n,M}f:=\varphi ((I_{n,\lambda_0} f)(\cdot),\ldots,(I_{n,\lambda_M}f)(\cdot))$
	and use Assumption~\ref{ass:en_rate} to obtain
	\begin{align*}
		\|I_n f-\Phi_{n,M} f\|_\kappa 
		&=\Big\|\sup_{\lambda\in\Lambda_{r,n}}(I_{n,\lambda}f-\eta_n(\lambda)h_n)
		-\max_{m=0,\ldots,M}(I_{n,\lambda_m}f-\eta_n(\lambda_m)h_n)\Big\|_\kappa \\
		&\leq\sup_{\lambda\in\Lambda_{r,n}}
		\|(I_{n,\lambda}f-\eta_n(\lambda)h_n)-(I_{n,\lambda_{m(\lambda)}}f-\eta_n(\lambda_{m(\lambda)})h_n)\|_\kappa \\
		&\leq L_r\sup_{\lambda\in\Lambda_{r,n}}d_r(\lambda,\lambda_{m(\lambda)})^{\beta_r} 
		\leq L_r(4\delta_r(M))^{\beta_r}
	\end{align*}
	for all $f\in B_{\C^\alpha_\kappa}(r)$, where $m(\lambda)\in\argmin_{m=1,\ldots,M}d_r(\lambda,\lambda_m)$.\
	Estimating the H\"older seminorm similarly to the proof of Theorem~\ref{thm:uat_nisio}
	yields the claim.
\end{proof}

Finally, by scaling the number of neurons $(M_n)_{n\in\N}$ depending on the step-size
$(h_n)_{n\in\N}$, we obtain a sequence $(\Phi_{n,M_n})_{n\in\N}$ of neural operators 
which generate the same strongly continuous convex monotone semigroup $(S(t))_{t \geq 0}$ 
as the Chernoff one-step operators $(I_n)_{n\in\N}$.\ To this end, we recall that
the covering number of $\Lambda_r$ is defined by
\[ \mathfrak{N}_r(\epsilon):=\min\left\{M\in\N\colon\exists\lambda_1,\ldots,\lambda_M\in\Lambda_r
\;\text{ s.t. } \bigcup_{m=1}^M \mathring{B}_{d_r}(\lambda_m,\epsilon)=\Lambda_r\right\}
\quad\text{for all } \epsilon>0. \]
It holds $\mathfrak{N}_r(\epsilon)=\min\{M\in\N\colon\delta_r(M)<\epsilon\}$ 
and $\delta_r(M)=\inf\{\epsilon>0\colon\mathfrak{N}_r(\epsilon)\leq M\}$.

\begin{proposition} \label{prop:approx_nisio_rate}
	Suppose that the Assumptions~\ref{ass:nisio} and~\ref{ass:en_rate} are satisfied 
	and denote by $(S(t))_{t \geq 0}$ the semigroup from Theorem~\ref{thm:nisio}.\ 
	Let $(r_n)_{n\in\N}\subset (0,\infty)$ be a sequence with $r_n\to\infty$. Choose 
	$\epsilon_n>0$ with $L_{r_n}\epsilon_n^{\beta_{r_n}}=o(h_n)$ and $M_n\geq\mathfrak{N}_{r_n}(\epsilon_n/4)$.\
	Then, there exists a sequence $(\Phi_{n,M_n})_{n\in\N}$ of envelope-neural
	operators $\Phi_{n,M_n}\in\mathcal{EN}^{\ReLU}_{\Ck,\Ck}$ 
	with $M_n+1$ neurons such that
	\[ S(t)f=\lim_{n\to\infty}\Phi_{n,M_n}^{k_n^t}f 
	\quad\text{for all } t\geq 0 \text{ and } f\in\Ck. \]
\end{proposition}
\begin{proof}
	For every $f\in\Cbi$ and $K\Subset\Rd$, Assumption~\ref{ass:nisio}(vii) 
	and Theorem~\ref{thm:rate_nisio} imply
	\begin{align*}
		\lim_{n\to\infty}\left\|\frac{\Phi_{n,M_n}f-f}{h_n}-Af\right\|_{\infty,K}
		&\leq c\lim_{n\to\infty}\frac{\|\Phi_{n,M_n}f-I_n f\|_\kappa}{h_n} 
		+\lim_{n\to\infty}\left\|\frac{I_n f-f}{h_n}-Af\right\|_{\infty,K} \\
		&\leq c\lim_{n\to\infty}\frac{L_{r_n}(4\delta_{r_n}(M_n))^{\beta_{r_n}}}{h_n} 
		\leq c\lim_{n\to\infty}\frac{L_{r_n}\epsilon_n^{\beta_{r_n}}}{h_n}=0,
	\end{align*}
	where $c:=\sup_{x\in K}\frac{1}{\kappa(x)}$.\ Furthermore, Assumption~\ref{ass:nisio}
	guarantees that $(\Phi_{n,M_n})_{n\in\N}$ is a family of one-step 
	operators $\Phi_{n,M_n}\colon\Ck\to\Ck$ satisfying Assumption~\ref{ass:chernoff}.\
	By Theorem~\ref{thm:chernoff}, there exists another strongly continuous convex
	monotone semigroup $(T(t))_{t\geq 0}$ on~$\Ck$ given by
	\[ T(t)f:=\lim_{n\to\infty}\Phi_{n,M_n}^{k_n^t}f 
	\quad\text{for all } t\geq 0 \text{ and } f\in\Ck \]
	with generator $B\colon D(B)\to\Ck$ satisfying (iv) and~(v) of Theorem~\ref{thm:chernoff}, 
	$\Cbi\subset D(B)$ and 
	\[ Af=Bf \quad\text{for all } f\in\Cbi. \]
	Consequently, we obtain $S(t)f=T(t)f$ for all $t\geq 0$ and $f\in\Ck$. 
\end{proof}

\subsection{Quantitative approximation of the semigroup}

Finally, by combining Theorem~\ref{thm:rate_nisio} with the convergence rates
in~\cite{blessing25}, we obtain an explicit error bound for the approximation 
of the strongly continuous convex monotone semigroup $(S(t))_{t\geq 0}$ by the
envelope-neural operators.\ The next assumption is a particular case
of~\cite[Assumption~2.5 and~2.7]{blessing25}.\ Let $\zeta\colon\R_+\times\Rd\to\R$ 
be an infinitely differentiable function satisfying
\[ \supp(\zeta)\subset [0,1]\times B_{\Rd}(1) \quad\text{and}\quad
\int_{\R_+\times\Rd}\zeta(t,x)\,\d t\,\d x=1. \]
For every locally bounded measurable function $u\colon\R_+\times\Rd\to\R$,
$\epsilon=(\epsilon_1,\epsilon_2)\in (0,\infty)^2$, $t\geq 0$ and $x\in\Rd$,
we define 
$\zeta^\epsilon(t,x):=\epsilon_1^{-1}\epsilon_2^{-d}\zeta(\epsilon_1^{-1}t,\epsilon_2^{-1}x)$
and 
\[ u^\epsilon(t,x):=(u*\zeta^\epsilon)(t,x)
:=\int_{\R_+\times\Rd}u(s+t,x+y)\zeta^\epsilon(s,y)\,\d s\,\d y. \]

\begin{assumption} \label{ass:rate_semigroup}
	Suppose that the following conditions are satisfied:
	\begin{enumerate}
		\item There exist $c\geq 0$ and $\epsilon_0\in (0,1]$ with 
		\[ \|I_n (\tau_x f)-\tau_x I_n f\|_\kappa\leq crh_n|x| \]
		for all $x\in B_{\Rd}(\epsilon_0)$, $r\geq 0$, $f\in\Lipb(r)$ and $n\in\N$.
		\item There exist a function $a_1\colon\R_+\to\R_+$ and constants $a_2,p \geq 0$ with
		\[ \|I_n f-f\|_\kappa
		\leq \left( a_1\big(d^{-\frac{1}{2}}\|\nabla f\|_\infty\big)+a_2\big(d^{-1}\|\nabla^2 f\|_\infty\big)^p \right) h_n 
		\quad \text{for all } n\in\N \text{ and } f\in\Cbi. \]
		\item There exist a function $\theta\colon\R_+^2\to\R_+$ and constants
		$\gamma^\pm_1\ldots,\gamma^\pm_N\geq 0$ with
		\begin{align*}
			\partial_t u^\epsilon(t)-Au^\epsilon(t)-\frac{u^\epsilon(t)-I_n u^\epsilon(t-h_n)}{h_n}
			&\leq\frac{\theta(r,t)}{\kappa}\max_{i=1,\ldots,N}h_n^{\gamma^+_i}\epsilon_2^{-\gamma^-_i}, \\
			\partial_t u_n^\epsilon(t)-Au_n^\epsilon(t)-\frac{u_n^\epsilon(t)-I_n u_n^\epsilon(t-h_n)}{h_n} 
			&\geq -\frac{\theta(r,t)}{\kappa}\max_{i=1,\ldots,N}h_n^{\gamma^+_i}\epsilon_2^{-\gamma^-_i}
		\end{align*}
		for all $n\in\N$, $r\geq 0$, $f\in\Lipb(r)$, $t\geq h_n$ and $\epsilon=(\epsilon_1,\epsilon_2)\in (0,\epsilon_0]^2$ 
		with $\epsilon_1=\epsilon_2^{1+p}\geq h_n$, where $u(t):=S(t)f$ and 
		$u_n(t):=I(\pi_n^t)f$.
	\end{enumerate}
\end{assumption}

\begin{theorem} \label{thm:approx_nisio_rate}
	Let the Assumptions~\ref{ass:nisio},~\ref{ass:en}(i),~\ref{ass:en_rate} and~\ref{ass:rate_semigroup}
	be satisfied and denote by $(S(t))_{t \geq 0}$ the semigroup from Theorem~\ref{thm:nisio}.\ 
	Then, for every $r,T\geq 0$, there exists $c_{r,T}\geq 0$ such that, for every 
	$M_n\geq\mathfrak{N}_{e^{\omega T}r}(\epsilon_n/4)$ with 
	$\epsilon_n^{\beta_{e^{\omega T}r}}=\cO(h_n^\gamma/k^T_n)$,
	there exists $\Phi_{n,M_n}\in\mathcal{EN}^{\ReLU}_{\Ck,\Ck}$
	with $M_n+1$ neurons such that
	\[ \big\|S(t)f-\Phi_{n,M_n}^{k^t_n}f\big\|_\kappa\leq c_{r,T}h_n^{\gamma} 
	\quad \text{for all } t\in [0,T] \text{ and } f\in\Lipb(r), \]
	where $\gamma:=\min\big\{\frac{1}{1+p},\frac{\gamma^+_1}{1+\gamma^-_1},\ldots,\frac{\gamma^+_N}{1+\gamma^-_N}\big\}$.
\end{theorem}
\begin{proof}
	We use~\cite[Theorem~2.9]{blessing25}, argue similarly as in the proof of 
	Theorem~\ref{thm:approx_nisio} and apply Theorem~\ref{thm:rate_nisio} to obtain
	constants $\tilde{c}_{r,T}, c_{r,T}\geq 0$ with
	\begin{align*}
		\big\|S(t)f-\Phi_{n,M_n}^{k^t_n}f\big\|_\kappa
		&\leq\|S(t)f-I(\pi^t_n)f\|_\kappa+\big\|I_n^{k^t_n}f-\Phi_{n,M_n}^{k^t_n}f\big\|_\kappa \\
		&\leq\tilde{c}_{r,T}h_n^\gamma+\sum_{l=1}^{k^t_n}
		\big\|I_n^{k^t_n-l}I_n\Phi_{n,M_n}^{l-1}f-I_n^{k^t_n-l}\Phi_{n,M_n}\Phi_{n,M_n}^{l-1}f\big\|_\kappa \\
		&\leq \tilde{c}_{r,T}h_n^\gamma+e^{\omega T}\sum_{l=1}^{k^t_n}
		\|I_n\Phi_{n,M_n}^{l-1}f-\Phi_{n,M_n}\Phi_{n,M_n}^{l-1}f\|_\kappa \\
		&\leq\tilde{c}_{r,T}h_n^\gamma
		+e^{\omega T}k^t_n L_{e^{\omega T}r}(4\delta_{e^{\omega T}r}(M_n))^{\beta_{e^{\omega T}r}} \\
		&\leq\tilde{c}_{r,T}h_n^\gamma+e^{\omega T}k^T_n L_{e^{\omega T}r}\epsilon_n^{\beta_{e^{\omega T}r}} 
		\leq c_{r,T}h_n^\gamma. \qedhere 
	\end{align*}
\end{proof}

The previous theorem might seem very abstract at first, but in many applications
verifying the assumptions and deriving explicit convergence rates is rather straightforward,
see~\cite[Section~4]{blessing25}.\ For instance, in case of the optimal control problem
studied in Subsection~\ref{sec:soc} below, it follows from~\cite[Theorem~4.3]{blessing25}
that $\gamma=\frac{1}{4}$. Furthermore, it holds $\beta_r=\alpha$ and $\Lambda\subset\Rd\times\mathbb{S}^d_+\cong\R^\frac{d(d+3)}{2}$.\
Consequently, one should choose 
\[ M_n+1\geq\mathfrak{N}_{e^{\omega T}r}(\epsilon_n)+1\asymp\epsilon_n^{-\frac{d(d+3)}{2}} 
\asymp\left(\frac{h_n^\gamma}{k^T_n}\right)^{-\frac{d(d+3)}{2\alpha}} 
\asymp h_n^{-\frac{d(d+3)(\gamma+1)}{2\alpha}}=h_n^{-\frac{5d(d+3)}{8\alpha}} \]
neurons in the envelope-neural operator.

\section{Numerical experiments}
\label{sec:numerics}

In this section, we illustrate with three numerical examples\footnote{The 
	numerical experiments have been implemented in \texttt{Python} and were 
	executed on a high-performance computing (HPC) cluster of ETH Zurich. 
	The code can be found at the URL~\url{https://github.com/sgarale/chernoff_neural}} 
how Chernoff-neural operators and envelope-neural operators can be effectively 
used to learn the Chernoff one-step operators $(I_n)_{n\in\N}$ and therefore
the corresponding strongly continuous convex monotone semigroup $(S(t))_{t \geq 0}$.
The examples follow an increasing level of abstraction and are carefully chosen
to highlight the key aspects of our approach.

\subsection{Splitting schemes for semilinear PDEs}
\label{sec:pde}

In the first example, we approximate the solution operator $f\mapsto S(t)f:=u(t,\cdot)$ 
of the semilinear partial differential equation (PDE)
\begin{equation} \label{eq:pde}
	\begin{cases}
		\partial_t u(t,x)=\frac{1}{2}\Delta u(t,x)+H\big(\nabla u(t,x)\big), & (t,x)\in (0,\infty)\times\Rd, \\
		\;\;\, u(0,x)=f(x), & x\in\Rd,
	\end{cases}    
\end{equation}
where $\Delta$ denotes the Laplacian, $\nabla$ represents the gradient and
$H\colon\Rd\to\R$ is a convex Hamiltonian whose convex conjugate is given by
\[ H^*(y):=\sup_{x\in\Rd}\big(y^\top x-H(x)\big) \quad\text{for all } y\in\Rd. \]
We split the differential operator in equation~\eqref{eq:pde} into the linear 
diffusion part $A^{\mathrm{diff}}f:=\frac{1}{2}\Delta f$ and the non-linear 
first-order part $A^{\mathrm{ham}}f:=H(\nabla f)$.\ For a fixed sequence 
$(h_n)_{n\in\N}\subset (0,\infty)$ with $h_n\to 0$, the corresponding Chernoff
one-step operators are given by
\[ (J^{\mathrm{diff}}_n f)(x):=\mathbb{E}[f(x+W_{h_n})] \quad\text{and}\quad 
(J^{\mathrm{ham}}_n f)(x):=\sup_{\lambda\in\Lambda}\big(f(x+\lambda h_n)-H^*(\lambda)h_n\big) \]
for all $n\in\N$, $f\in\Cb$ and $x\in\Rd$, 
where $(W_t)_{t\geq 0}$ is a $d$-dimensional Brownian motion satisfying $\cov(W_1)=I_d$
and $\Lambda:=\{H^*<\infty\}$. For every $f\in\Cbi$, It\^o's and Taylor's formula imply 
\[ \lim_{n\to\infty}\frac{J^{\mathrm{diff}}_nf-f}{h_n}=A^{\mathrm{diff}}f := \frac{1}{2}\Delta f
\quad\text{and}\quad \lim_{n\to\infty}\frac{J^{\mathrm{ham}}_n f-f}{h_n}=A^{\mathrm{ham}}f:=H(\nabla f). \]
For every $n\in\N$, $f\in\Cb$ and $x\in\Rd$, we define
\begin{equation} \label{eq:pde:chernoff}
	(I_n f)(x):=(J_n^{\mathrm{ham}}J_n^{\mathrm{diff}}f)(x)
	=\sup_{\lambda\in\Lambda}\big(\E[f(x+\lambda h_n+W_{h_n})]-H^*(\lambda)h_n\big).
\end{equation}
It follows from~\cite[Theorem~3.6]{blessing23} that there exists a strongly continuous
convex monotone semigroup $(S(t))_{t\geq 0}$ on~$\Cb$ given by
\[ S(t)f:=\lim_{n\to\infty}I(\pi_n^t)f \quad\text{for all } t\geq 0 \text{ and } f\in\Cb \]
whose generator is given by 
\[ Af=\frac{1}{2}\Delta f+H(\nabla f) \quad\text{for all } f\in\Cbi. \]
Furthermore, the unique (viscosity) solution of equation~\eqref{eq:pde}
is given by $u(t,\cdot):=S(t)f$ for all $t\geq 0$. In the following, we implement
the splitting scheme $u(t,\cdot) \approx I_n^{k^t_n}f$, where each diffusion step
is followed by a non-linear transport step.\ We will approximate the one-step 
operator~$I_n$ both with a Chernoff-neural operator 
$\Phi_n\in\mathcal{CN}^{\H,\rho,\L}_{\C_{\kappa_0},\C^\alpha_\kappa}$ and an 
envelope-neural operator $\Phi_n\in\mathcal{EN}^{\ReLU}_{\Ck,\Ck}$.

\begin{lemma} \label{lem:pde}
	Let $\Lambda$ be bounded and $H^*|_\Lambda$ Lipschitz continuous.\ Furthermore,
	let $\alpha:=1$, $\kappa\equiv 1$, $\kappa_0(x):=(1+|x|^{q_0})^{-1}$ and
	$\K:=\{\kappa\}\cup\{(1+|\cdot|^{q'})^{-1}\colon q'\in (2,q_0]\}$ for some 
	$q_0\in (2,\infty)$.\ Then, the operators $(I_n)_{n\in\N}$ defined by 
	equation~\eqref{eq:pde:chernoff} satisfy the 
	Assumptions~\ref{ass:chernoff},~\ref{ass:cn_uat},~\ref{ass:nisio},~\ref{ass:en} 
	and~\ref{ass:en_rate}.
\end{lemma}
\begin{proof}
	See Appendix~\ref{app:pde}.
\end{proof}

For the numerical experiment, we choose $d=1$, a bounded time horizon $T=1$, 
$n=30$, $h_n=\frac{1}{30}$ and $H(p)=|p|$ which implies $\Lambda=[-1,1]$
and $H^*=+\infty\one_{[-1,1]^c}$.\ Furthermore, we independently sample the 
parameters 
\[ a_i\sim\cU(0.5,1.5),\quad \sigma_i\sim\cU(0.5,1.5),
\quad b_i\sim\cU(-2,2) \quad\text{and}\quad c_i\sim\cU(-1,1) \]
for $i=1,\ldots,I$ with $I:=2000$ according to uniform distributions.\ Using these 
parameters, we define the functions
\begin{equation} \label{eq:pde:fcts}
	f_{i,1}(x):=c_i+a_i\left(\Gamma\left(\frac{x-b_i}{\sigma_i}\right)-\frac{1}{2}\right)
	\quad\text{and}\quad
	f_{i,2}(x):=c_i-a_i\left(\Gamma\left(\frac{x-b_i}{\sigma_i}\right)-\frac{1}{2}\right)
\end{equation}
which are split up into $90\%/10\%$ for training and testing, where $\Gamma$ denotes
the cumulative distribution function of the standard normal distribution $\mathcal{N}(0,1)$.
Since
\[ \E\left[\Gamma\left(\frac{x+h_n\lambda+W_{h_n}-b_i}{\sigma_i}\right)\right]
=\Gamma\left(\frac{x+h_n\lambda-b_i}{\sqrt{\sigma_i^2+h_n}}\right) \]
and $f_{i,1}$ and $f_{i,2}$ are increasing and decreasing, respectively, the supremum over
$\lambda\in[-1,1]$ is attained at $\lambda_{i,1}^\star=1$ and $\lambda_{i,2}^\star=-1$ 
such that
\begin{equation} \label{eq:pde:cn_out}
	\begin{aligned}
		(I_n f_{i,1})(x) &=c_i+a_i\left(\Gamma\left(\frac{x+h_n-b_i}{\sqrt{\sigma_i^2+h_n}}\right)-\frac{1}{2}\right), \\
		(I_n f_{i,2})(x) &=c_i-a_i\left(\Gamma\left(\frac{x-h_n-b_i}{\sqrt{\sigma_i^2+h_n}}\right)-\frac{1}{2}\right).
	\end{aligned}
\end{equation}
In particular, the class of functions defined by equation~\eqref{eq:pde:fcts} is 
preserved under iteration of $I_n$.\ Then, we learn the neural operators~$\Phi_n$ 
by minimizing the mean squared error (MSE)
\begin{equation} \label{eq:mse}
	\frac{1}{2|\mathcal{I}|L}\sum_{i\in\mathcal{I}}\sum_{j=1}^2\sum_{l=1}^L 
	|(I_n f_{i,j})(x_l)-(\Phi_n f_{i,j})(x_l)|^2
\end{equation}
over the training set $\mathcal{I}$, where $(x_l)_{l=1,\ldots,L}\sim\mathcal{N}(0,\sigma^2)$ 
are $L:=200$ independent and identically distributed (i.i.d.)~evaluation points with 
$\sigma:=3$.\ In order to assess the out-of-sample performance of the trained operators,
we evaluate them on the hyperbolic tangent function
\[ f^{\mathrm{th}}(x):=\frac{1}{2}\tanh(x) \]
which does not belong to the class of training functions defined by equation~\eqref{eq:pde:fcts}.\
As reference solutions, we use an explicit representation for $S(t)f_{i,j}$ which 
is similar to equation~\eqref{eq:pde:cn_out}.\ Moreover, since $f^{\mathrm{th}}$ is
increasing, it holds $(S(t)f^{\mathrm{th}})(x)=\frac{1}{2}\E[\tanh(x+t+W_t)]$
which we approximate with Gauss--Hermite quadrature.\ For the Chernoff-neural operator 
$\Phi_n\in\mathcal{CN}^{\H,\rho,\L}_{\C_{\kappa_0},\C^\alpha_\kappa}$,
we use $M=100$ neurons and $\tanh$ as activation function, where~$\H$ and~$\L$
consist of neural networks with $20$ neurons and $\tanh$ as activation function,
and apply the Adam algorithm over $10^4$ epochs with learning rate $10^{-5}$ and 
batchsize $200$ per type $j=1,2$.\ For the envelope-neural operator 
$\Phi_n\in\mathcal{EN}^{\ReLU}_{\Ck,\Ck}$, we choose $M=64$, 
a ReLU-neural network $\varphi\colon\R^M\to\R$ with 2 hidden layers of size $64$ 
and $32$ and apply the Adam algorithm over $10^4$ epochs with learning rate $10^{-5}$ 
and batchsize $500$.\ The results are reported in Figure~\ref{fig:pde}.

\begin{figure}[ht!]
	\centering
	\begin{minipage}{0.49\textwidth}
		\centering
		\includegraphics[height=5.5cm]{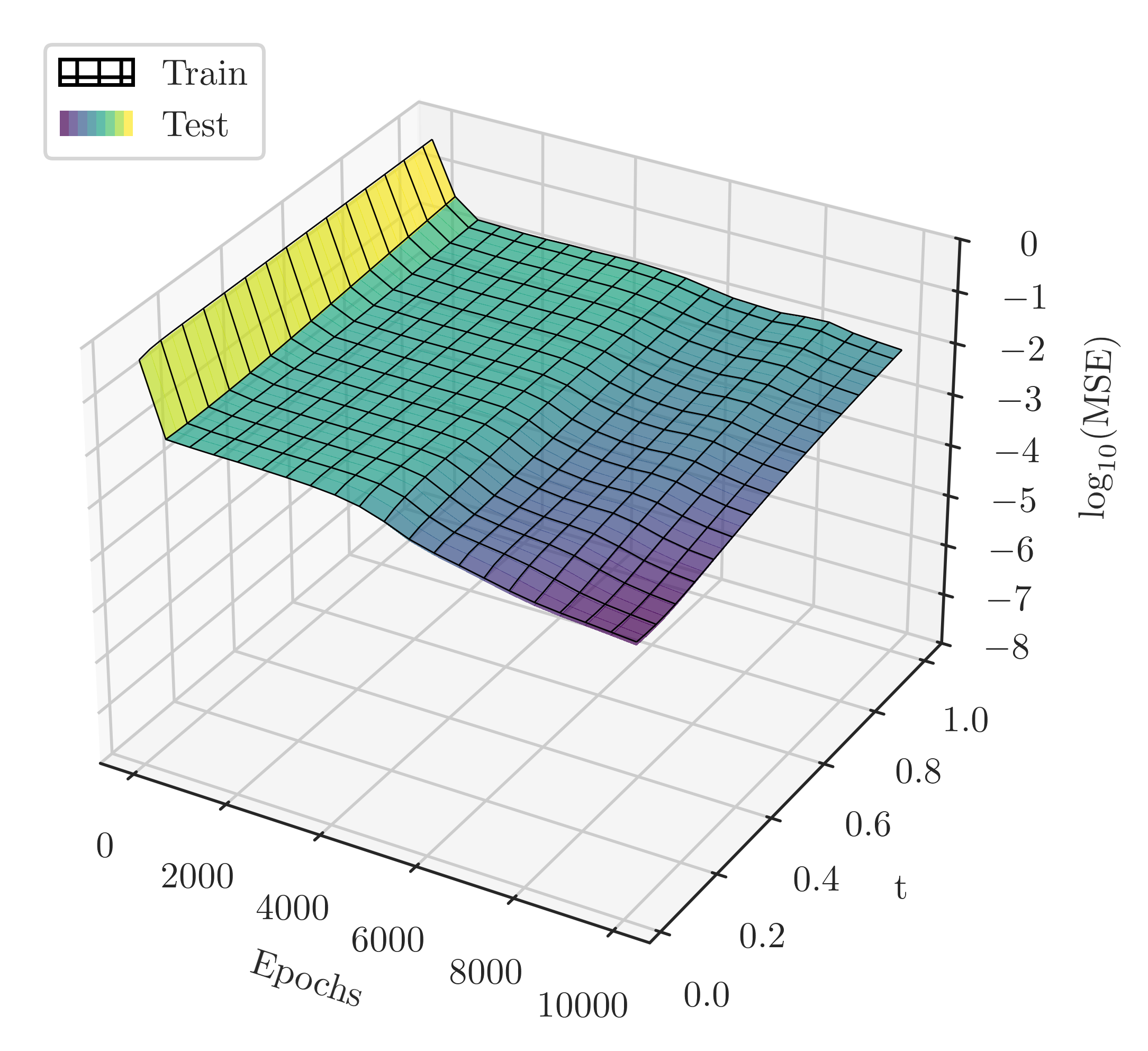}
		
		{\footnotesize ($\mathcal{C}$1) MSE~\eqref{eq:mse} between $\Phi_n^{k^t_n} f_{i,j}$ and $S(t) f_{i,j}$ along training epochs and time.}
	\end{minipage}
	\begin{minipage}{0.49\textwidth}
		\centering
		\includegraphics[height=5.5cm]{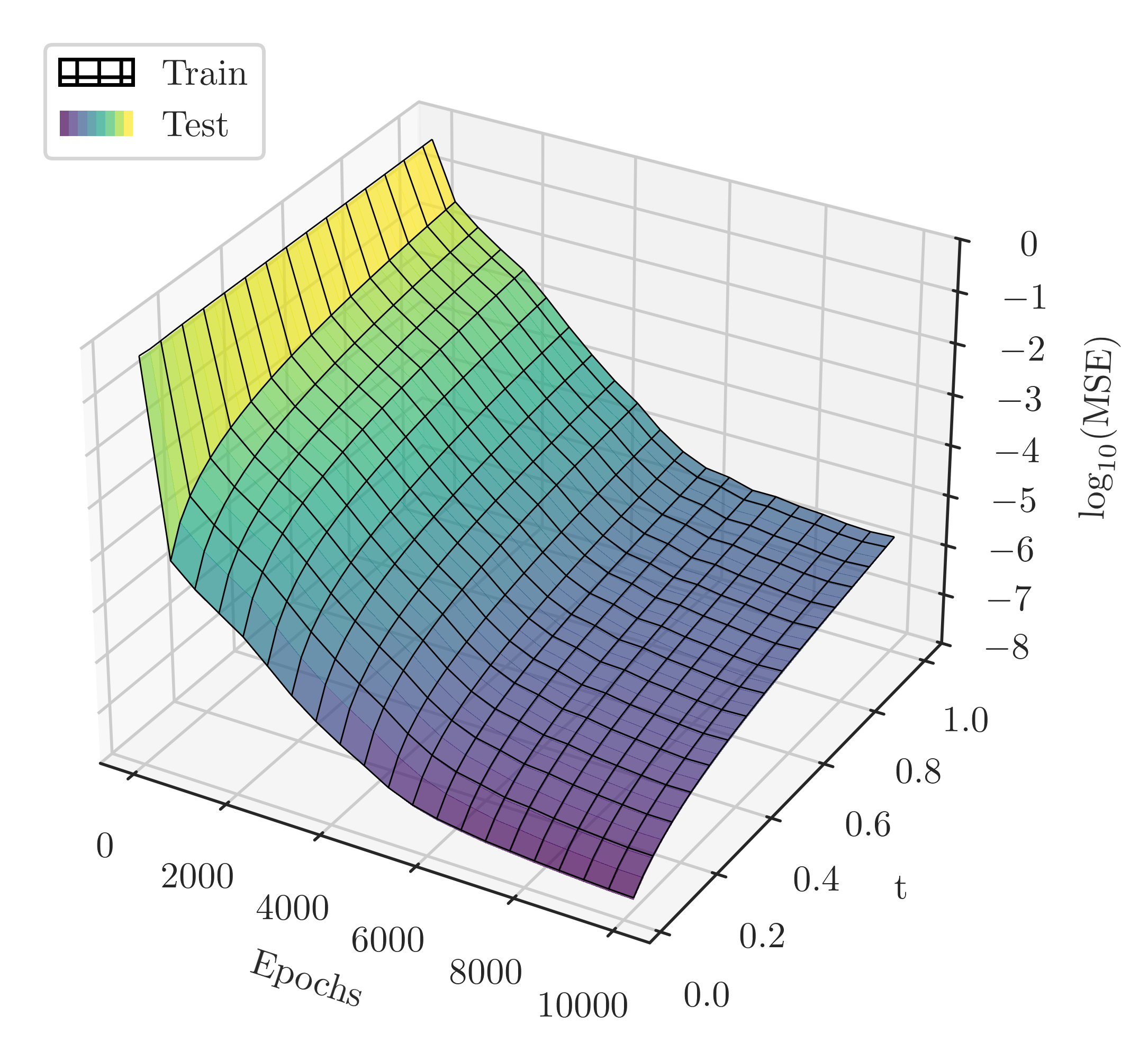}
		
		{\footnotesize ($\mathcal{E}$1) MSE~\eqref{eq:mse} between $\Phi_n^{k^t_n} f_{i,j}$ and $S(t) f_{i,j}$ along training epochs and time.}
	\end{minipage}
	\vspace{0.4cm}
	
	\begin{minipage}{0.49\textwidth}
		\centering
		\includegraphics[height=5.5cm]{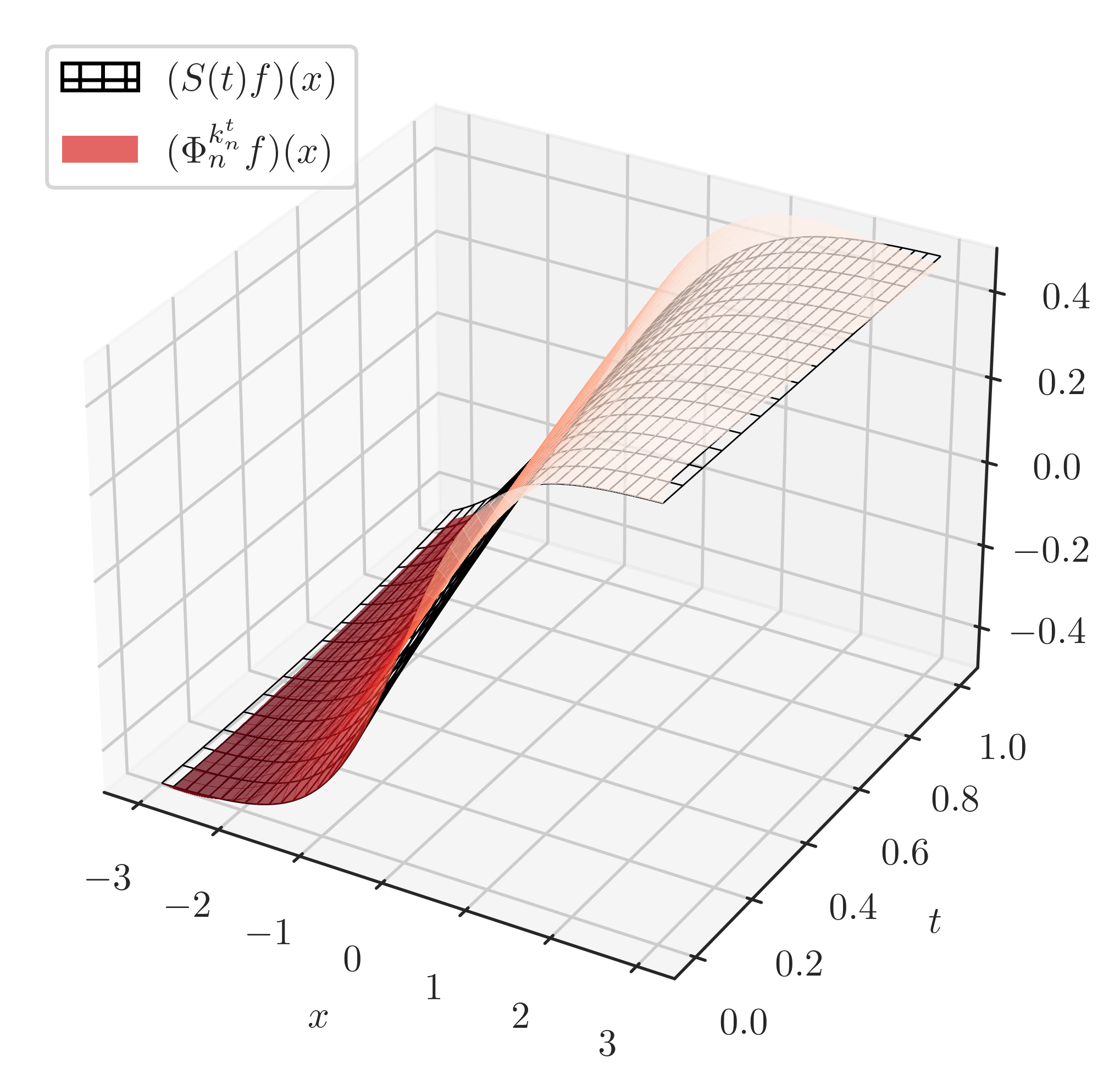}
		
		{\footnotesize ($\mathcal{C}$2) Approximation of $S(t) f^{\mathrm{th}}$ by $\Phi_n^{k^t_n} f^{\mathrm{th}}$.}
	\end{minipage}
	\begin{minipage}{0.49\textwidth}
		\centering
		\includegraphics[height=5.5cm]{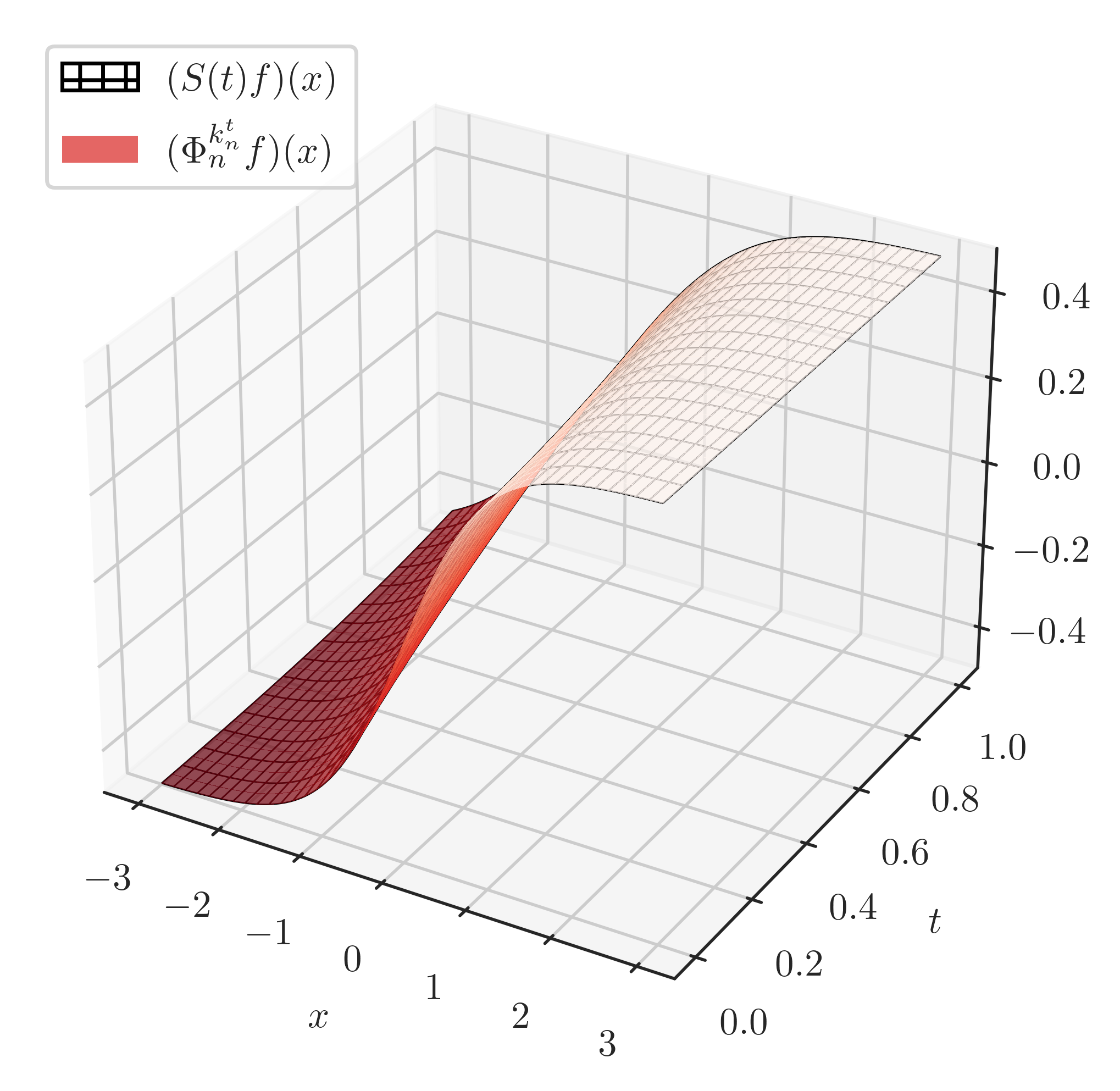}
		
		{\footnotesize ($\mathcal{E}$2) Approximation of $S(t) f^{\mathrm{th}}$ by $\Phi_n^{k^t_n} f^{\mathrm{th}}$.}
	\end{minipage}
	\vspace{0.4cm}
	
	\caption{Approximation of the solution operator $f\mapsto S(t)f:=u(t,\cdot)$ 
		of the PDE~\eqref{eq:pde} by using the Chernoff one-step operator $I_n$ defined in
		equation~\eqref{eq:pde:chernoff} which is learned with a Chernoff-neural operator 
		$\Phi_n\in\mathcal{CN}^{\H,\rho,\L}_{\C_{\kappa_0},\C^\alpha_\kappa}$ in
		($\mathcal{C}$1)--($\mathcal{C}$2) and an envelope-neural operator 
		$\Phi_n\in\mathcal{EN}^{\ReLU}_{\Ck,\Ck}$ in ($\mathcal{E}$1)--($\mathcal{E}$2). 
		In ($\mathcal{C}$1)+($\mathcal{E}$1), the MSE~\eqref{eq:mse} on the train/test set
		is evaluated.\ In ($\mathcal{C}$2)+($\mathcal{E}$2), the approximation of 
		$S(t)f^{\mathrm{th}}$ by $\Phi_n^{k^t_n}f^{\mathrm{th}}$ is shown.}
	\label{fig:pde}
\end{figure}

\subsection{Stochastic optimal control}
\label{sec:soc}

In the second example, we approximate the value function
\begin{equation} \label{eq:soc}
	(S(t)f)(x):=\sup_{(\mu_s)_s\subset\Xi,\atop (\sigma_s)_s\subset\Sigma} 
	\E\left[f\left(x+\int_0^t\mu_s\,\d s+\int_0^t\sigma_s\,\d W_s\right)\right]
\end{equation}
of a stochastic optimal control problem, where $(\mu_s)_{s\in [0,T]}$ and 
$(\sigma_s)_{s\in [0,T]}$ are predictable processes taking values in bounded subsets 
$\Xi\subset\Rd$ and $\Sigma\subset\mathbb{S}^d_+$, respectively.\ By~\cite[Theorem~6.2]{denk25},
the family $(S(t))_{t\geq 0}$ is a strongly continuous convex monotone semigroup
on~$\Ck$ with $\kappa:=(1+|\cdot|^q)^{-1}$ for any $q\geq 2$ whose generator is given by
\[ Af=\sup_{(\mu,\sigma)\in\Xi\times\Sigma}\left(\frac{1}{2}\trace\big(\sigma^2\nabla^2 f\big)+\mu^\top\nabla f\right) 
\quad\text{for all } f\in\Cbi. \]
Furthermore, it holds $S(t)f=\lim_{n\to\infty}I(\pi_n^t)f$ for all $t\geq 0$
and $f\in\Ck$, where 
\begin{equation} \label{eq:soc:chernoff}
	(I_n f)(x):=\sup_{(\mu,\sigma)\in\Xi\times\Sigma}\,\E[f(x+\mu h_n+\sigma W_{h_n})]
\end{equation}
and $(h_n)_{n\in\N}\subset (0,\infty)$ is a fixed sequence with $h_n\to 0$.\
We learn this Chernoff one-step operator by a Chernoff-neural operator 
$\Phi_n\in\mathcal{CN}^{\H,\rho,\L}_{\C_{\kappa_0},\C^\alpha_\kappa}$
and an envelope-neural operator $\Phi_n\in\mathcal{EN}^{\ReLU}_{\Ck,\Ck}$.

\begin{lemma} \label{lem:soc}
	Let $\alpha\in (0,1]$, $2\leq q<q_0<\infty$, $\kappa(x):=(1+|x|^q)^{-1}$,
	$\kappa_0(x):=(1+|x|^{q_0})^{-1}$ and $\K:=\{(1+|\cdot|^{q'})^{-1}\colon q'\in [q,q_0]\}$.\
	Then, the family $(I_n)_{n\in\N}$ defined by equation~\eqref{eq:soc:chernoff} satisfies
	the Assumptions~\ref{ass:chernoff},~\ref{ass:cn_uat},~\ref{ass:nisio},~\ref{ass:en} 
	and~\ref{ass:en_rate}.
\end{lemma}
\begin{proof}
	See Appendix~\ref{app:soc}.
\end{proof}

For the numerical experiment, we choose $d=1$, a finite time horizon $T=1$, $n=30$, 
$h_n=\frac{1}{30}$, $\Xi:=[\underline{\mu},\overline{\mu}]:=[-0.25,0.25]$, 
$\Sigma := [\underline{\sigma},\overline{\sigma}]:=[1.5, 2.0]$ and generate 
$I:=1000$ i.i.d.~strike prices $(K_i)_{i=1,\ldots,I}\sim\mathcal{N}(0,\sigma^2)$
with $\sigma=2$. We use these parameters to define the functions
\begin{equation} \label{eq:soc:fcts}
	\begin{aligned}
		f_{i,1}(x) &=(x-K_i)_+, & f_{i,2}(x) &=(K_i-x)_+, \\
		f_{i,3}(x) &=-(x-K_i)_+, & f_{i,4}(x) &=-(K_i-x)_+,
	\end{aligned}
\end{equation}
where $s_+:=\max(s,0)$, which are split up into $90\%/10\%$ for training and testing
equally among the different types. Note that the operators $I_{n,(\mu,\sigma)}$ applied 
to call/put functions are given by
\begin{align} 
	\big(I_{n,(\mu,\sigma)}(\cdot-K)_+\big)(x) &=\E[(x+\mu h_n+\sigma W_{h_n}-K)_+] \nonumber \\
	&=(x+\mu h_n\!-\!K)\,\Gamma\left(\frac{x+\mu h_n-K}{\sigma\sqrt{h_n}}\right) 
	+\sigma\sqrt{h_n}\,\gamma\left(\frac{x+\mu h_n-K}{\sigma\sqrt{h_n}}\right), \label{eq:soc:1} \\
	\big(I_{n,(\mu,\sigma)} (K-\cdot)_+\big)(x) & = \mathbb{E}[(K-x-\mu h_n-\sigma W_{h_n})_+] \nonumber \\
	&=(K-x-\mu h_n)\,\Gamma\left(\frac{K-x-\mu h_n}{\sigma\sqrt{h_n}}\right)
	+\sigma\sqrt{h_n}\,\gamma\left(\frac{K-x-\mu h_n}{\sigma\sqrt{h_n}}\right), \label{eq:soc:2}
\end{align}
where $\Gamma$ (resp., $\gamma$) denotes the cumulative distribution (resp., probability density)
function of the standard normal distribution $\mathcal{N}(0,1)$.\ Consequently, the functions 
$I_n f_{i,j}$ are explicitly given by the equations~\eqref{eq:soc:1} and~\eqref{eq:soc:2}, 
where the supremum over $(\mu,\sigma)$ is attained at the following values:
$(\mu,\sigma)=(\overline{\mu},\overline{\sigma})$ if $j=1$, $(\mu,\sigma)=(\underline{\mu},\overline{\sigma})$
if $j=2$, $(\mu,\sigma):=(\underline{\mu},\underline{\sigma})$ if $j=3$ and 
$(\mu,\sigma)=(\overline{\mu},\underline{\sigma})$ if $j=4$.\ In order to assess the 
out-of-sample performance of the trained operators, we evaluate them on the absolute
value function
\[ f^{\mathrm{ab}}(x):=|x| \]
which does not belong to the training functions given by equation~\eqref{eq:soc:fcts}.\
As reference solution, we compute $S(t)f_{i,j}$ and $S(t)f^{\mathrm{ab}}$ using a finite
difference scheme for the corresponding fully nonlinear HJB equation.\ We learn the neural 
operators by minimizing the mean squared error (MSE) defined in equation~\eqref{eq:mse} 
over the training set, where $(x_l)_{l=1,\ldots,L} \sim \mathcal{N}(0,\sigma^2)$ are $L=200$
i.i.d.~evaluation points with $\sigma=3$.\ The other parameters are chosen as in 
Section~\ref{sec:pde}, except that the Chernoff-neural operator has ReLU activation 
functions and learning rate $2\cdot 10^{-5}$ and the envelope-neural operator has $M=8$
neurons and learning rate $10^{-5}$ with the ReLU-neural network $\varphi\colon\R^M\to\R$
having two hidden layers of $8$ and $4$ neurons. Both operators are trained with batchsize
$100$ per type $j=1,\ldots,4$. The results are reported in Figure~\ref{fig:soc}.

\begin{figure}[ht!]
	\centering
	\begin{minipage}{0.49\textwidth}
		\centering
		\includegraphics[height=5.5cm]{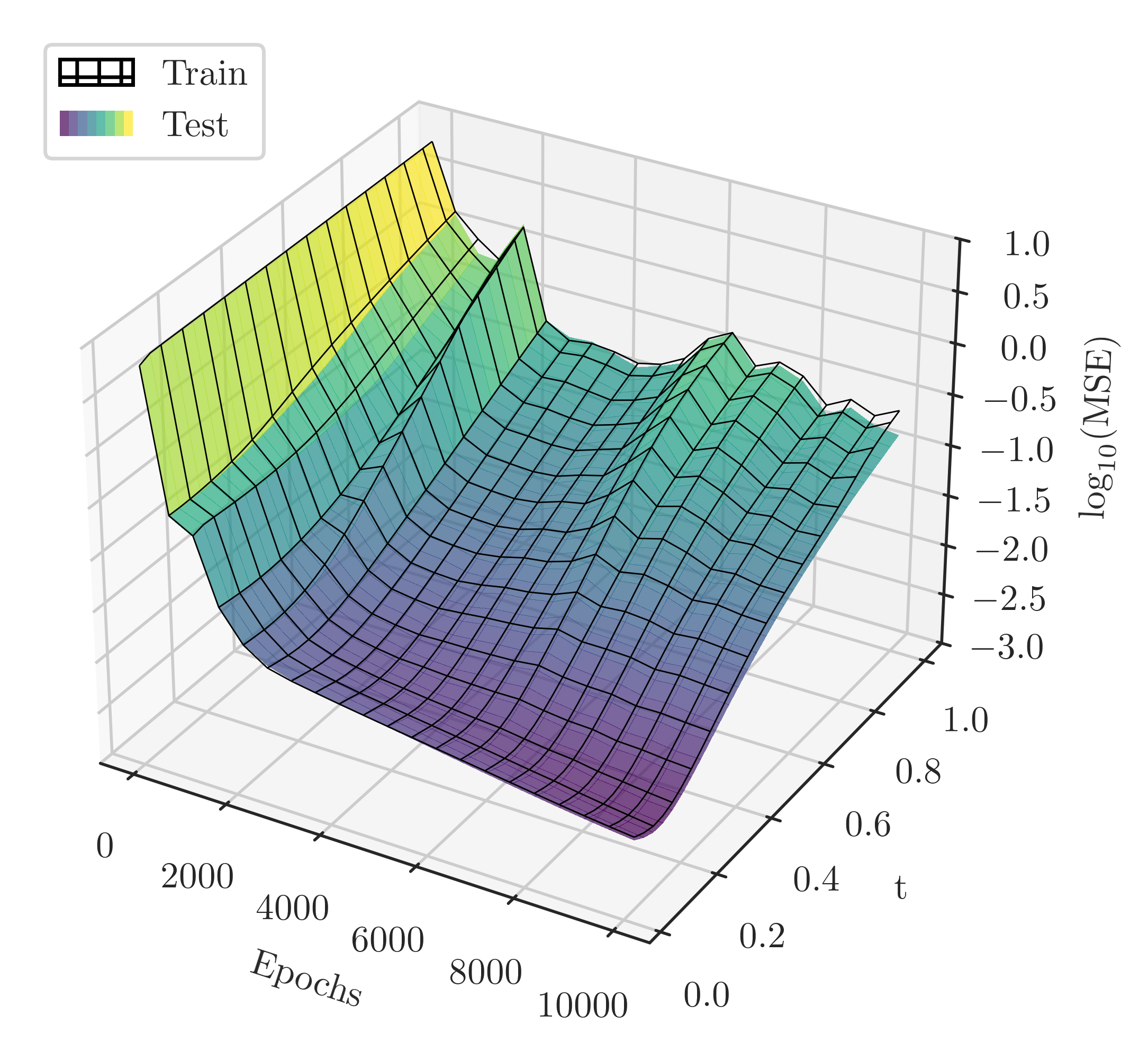}
		
		{\footnotesize ($\mathcal{C}$1) MSE~\eqref{eq:mse} between $\Phi_n^{k^t_n} f_{i,j}$ and $S(t) f_{i,j}$ along training epochs and time.}
	\end{minipage}
	\begin{minipage}{0.49\textwidth}
		\centering
		\includegraphics[height=5.5cm]{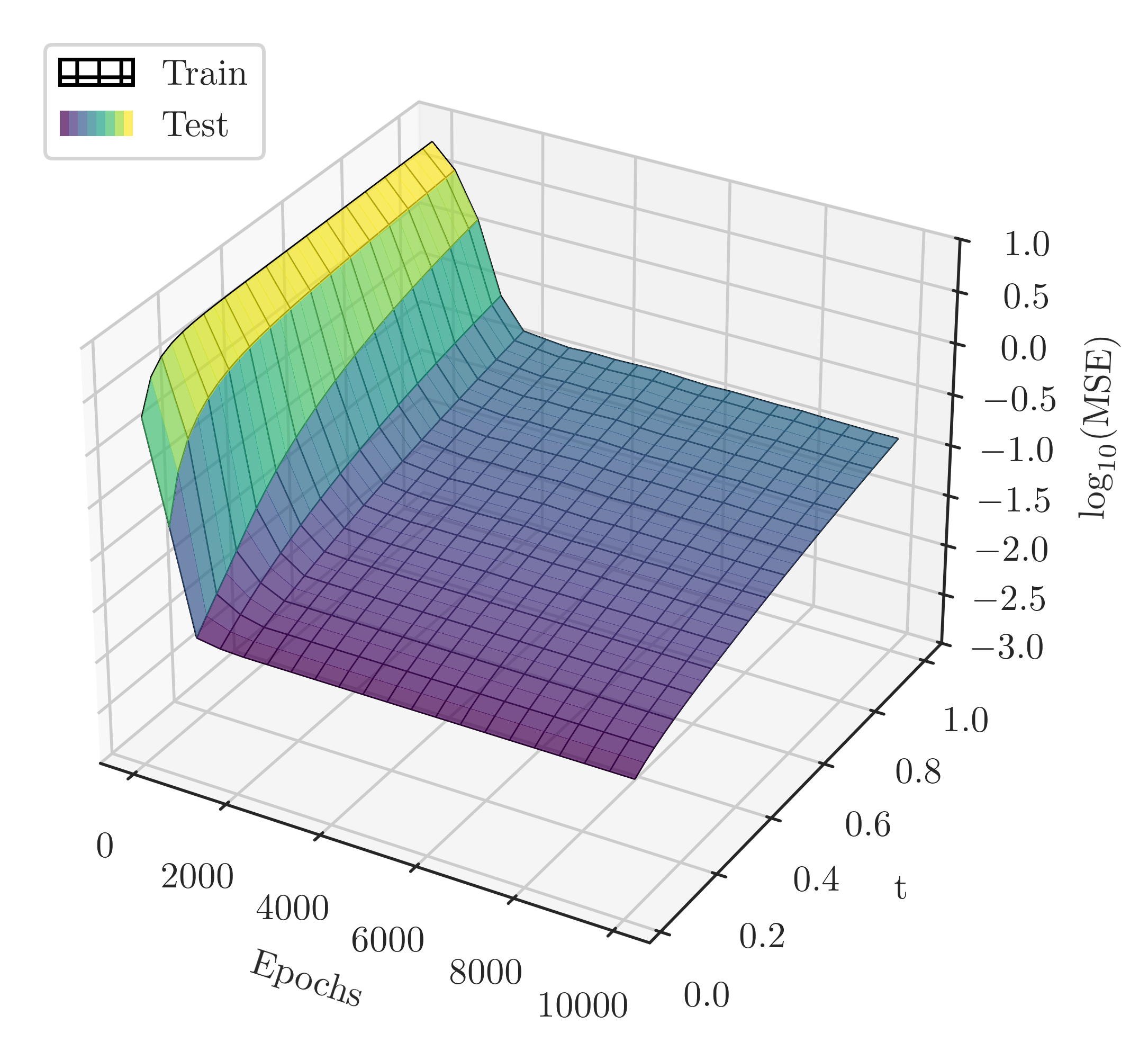}
		
		{\footnotesize ($\mathcal{E}$1) MSE~\eqref{eq:mse} between $\Phi_n^{k^t_n} f_{i,j}$ and $S(t) f_{i,j}$ along training epochs and time.}
	\end{minipage}
	\vspace{0.4cm}
	
	\begin{minipage}{0.49\textwidth}
		\centering
		\includegraphics[height=5.5cm]{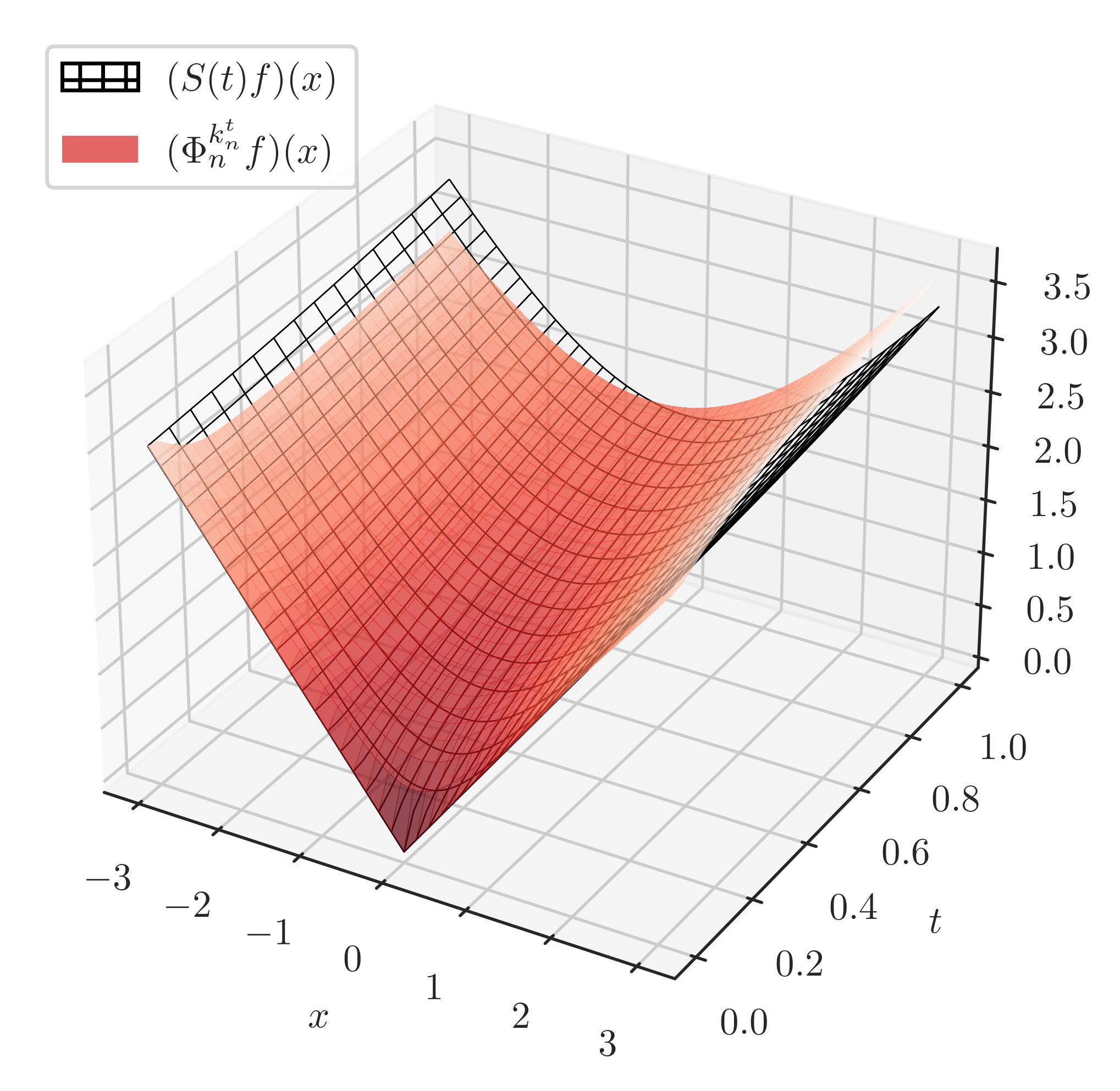}
		
		{\footnotesize ($\mathcal{C}$2) Approximation of $S(t) f^{\mathrm{ab}}$ by $\Phi_n^{k^t_n} f^{\mathrm{ab}}$.}
	\end{minipage}
	\begin{minipage}{0.49\textwidth}
		\centering
		\includegraphics[height=5.5cm]{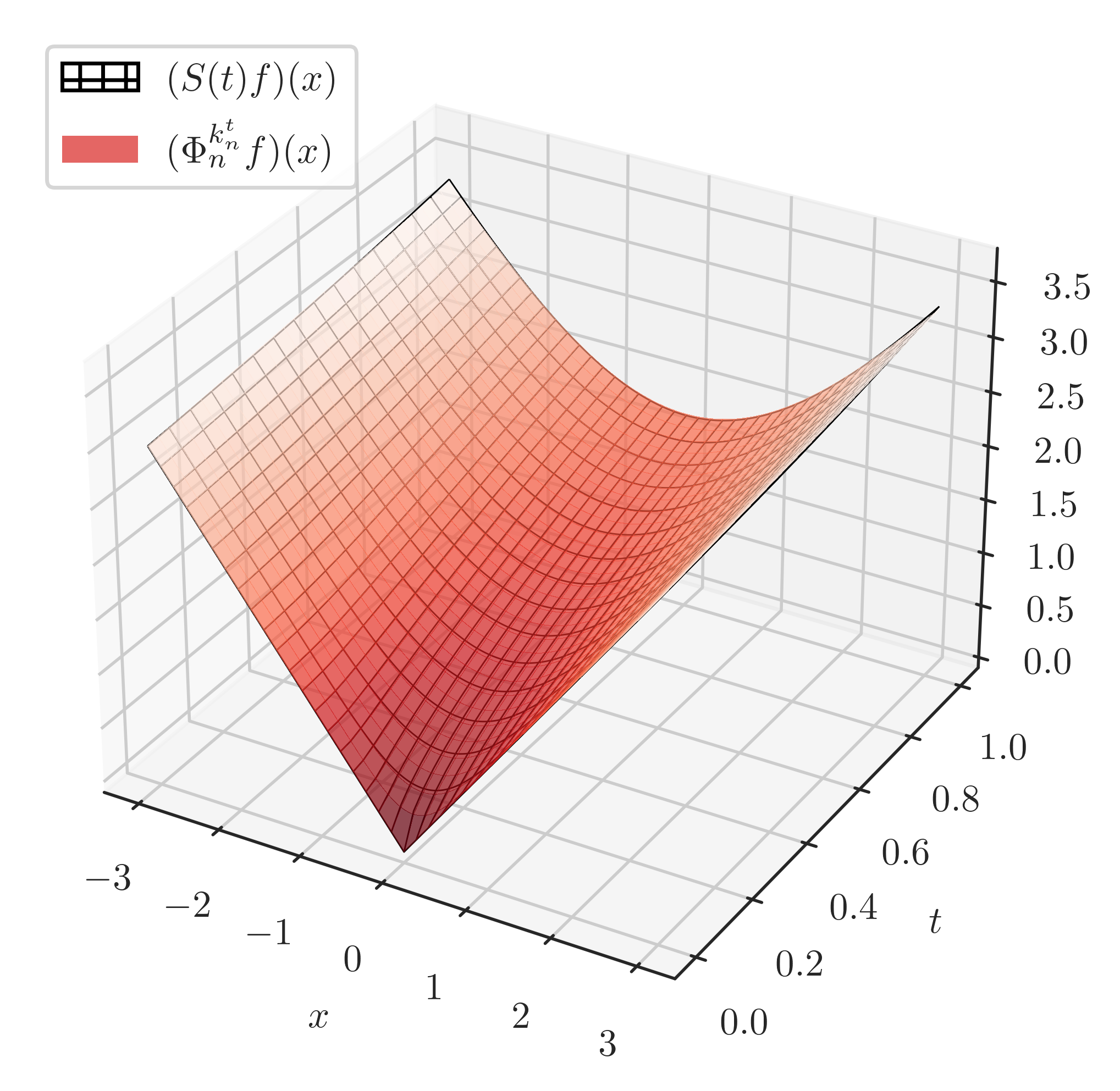}
		
		{\footnotesize ($\mathcal{E}$2) Approximation of $S(t) f^{\mathrm{ab}}$ by $\Phi_n^{k^t_n} f^{\mathrm{ab}}$.}
	\end{minipage}
	\vspace{0.4cm}
	
	\caption{Approximating the stochastic control problem $f \mapsto S(t) f$ in \eqref{eq:soc} by the Chernoff one-step operator $I_n f$ in \eqref{eq:soc:chernoff} which is learned with a Chernoff-neural operator $\Phi_n \in \mathcal{CN}^{\mathcal{H},\rho,\mathcal{L}}_{\C_{\kappa_0},\C^\alpha_\kappa}$ in ($\mathcal{C}$1)--($\mathcal{C}$2) and an envelope-neural operator $\Phi_n \in \mathcal{EN}^{\ReLU}_{\Ck,\Ck}$ in ($\mathcal{E}$1)--($\mathcal{E}$2). In ($\mathcal{C}$1)+($\mathcal{E}$1), the MSE~\eqref{eq:mse} on the train/test set is evaluated. In ($\mathcal{C}$2)+($\mathcal{E}$2), the approximation of $S(t) f^{\mathrm{ab}}$ by $\Phi_n^{k^t_n} f^{\mathrm{ab}}$ is shown.}
	\label{fig:soc}
\end{figure}

\subsection{Stochastic processes under model uncertainty}
\label{sec:wus}

In the third example, we follow~\cite{bartl2021,fuhrmann23} to study strongly 
continuous convex monotone semigroups arising from Wasserstein perturbations 
of the transition probabilities of Lévy processes.\ To that end, let $p>1$ and
$(X_t)_{t\geq 0}$ be a Lévy process with transition probabilities 
\[ p_t(x,B):=\varpi_t(\{y\in\Rd\colon x+y\in B\}) 
\quad\text{for all } t\geq 0, \, x\in\Rd \text{ and } B\in\B(\Rd), \]
where $(\varpi_t)_{t\geq 0}$ is a family of probability measures with finite
$p$-th moment such that
\begin{equation} \label{eq:wus:moments:bound}
	\lim_{t\to 0}\int_{\Rd}|y|^p\,\varpi_t(\d y)=0.
\end{equation}
In addition, let $\eta\colon\R_+\to [0,\infty]$ be a convex non-decreasing function
with $\eta(0)=0$ and $\eta\not\equiv 0$ which is locally Lipschitz continuous on 
$\{\eta<\infty\}$.\ Moreover, the function $\R_+\to [0,\infty],\; v \mapsto\eta(v^{1/p})$
is supposed to be convex which implies $\liminf_{v\to\infty}\frac{\eta(v)}{v^p}>0$.\ 
For every $n\in\N$, $f\in\Cb$ and $x\in\Rd$, we define 
\begin{equation} \label{eq:wus:chernoff}
	(I_n f)(x):=\sup_{\nu\in\cP_p(\Rd)}\left(\int_{\Rd}f(x + z)\,\nu(\d z)
	-\eta\left(\frac{\W_p(\varpi_{h_n},\nu)}{h_n}\right)h_n\right),
\end{equation}
where $\cP_p(\Rd)$ consists of all probability measures on $\B(\Rd)$ with finite
$p$-th moment,~$\W_p$ denotes the $p$-Wasserstein distance on $\cP_p(\Rd)$
and $(h_n)_{n\in\N}\subset (0,\infty)$ is a fixed sequence with $h_n\to 0$.
This Chernoff one-step operator incorporates nonparametric model uncertainty
by taking the fixed transition probabilities of the Lévy process $(X_t)_{t\geq 0}$
as reference model and weighting all other probability measures according to their 
distance to the reference model.\ For instance, in case that $\eta:=+\infty\one_{\{0\}^c}$,
the Chernoff one-step operators coincide with the linear transition semigroup
\[ (T(t)f)(x):=\int_{\Rd}f(x+y)\,\varpi_t(\d y) \]
of the Lévy process.\ Following~\cite{fuhrmann23}, it is possible to generalize 
the framework to reference dynamics of the form $X_t^x=\psi_t(x)+Y_t$ with a 
L\'evy process $(Y_t)_{t\geq 0}$ and a deterministic drift $(\psi_t)_{t\geq 0}$ 
which covers, for instance, Ornstein--Uhlenbeck processes.\
It follows from~\cite[Theorem~3.14]{fuhrmann23} that there exists a strongly continuous
convex monotone semigroup $(S(t))_{t\geq 0}$ on $\Cb$ given by 
\[ S(t)f:=\lim_{n\to\infty}I(\pi_n^t)f \quad\text{for all } t\geq 0 \text{ and } f\in\Cb. \]
Furthermore, the generator of $(S(t))_{t\geq 0}$ is given by 
\begin{equation} \label{eq:wus:generator}
	Af=Bf+\eta^*\big(|\nabla f|\big) \quad\text{for all } f\in D(B)\cap\C_b^1,
\end{equation}
where $\eta^*(w):=\sup_{v\geq 0}(vw-\eta(v))$ for all $w\geq 0$ and $B$ denotes 
the generator of $(T(t))_{t\geq 0}$.

\begin{lemma}\label{lem:wus}
	Let $\alpha:=1$, $\kappa\equiv 1$ and assume $\Cbi\subset D(B)$.\ Then, the operators 
	$(I_n)_{n\in\N}$ defined by equation~\eqref{eq:wus:chernoff} satisfy the 
	Assumptions~\ref{ass:nisio}(i)--(vii),~\ref{ass:en} and~\ref{ass:en_rate}.
\end{lemma}
\begin{proof}
	See Appendix~\ref{app:wus}.
\end{proof}

Looking at the generator in equation~\eqref{eq:wus:generator}, it becomes apparent 
that the same semigroup can be constructed by only taking parametric drift uncertainty
into consideration. Indeed, let 
\begin{equation} \label{eq:wus:nisio}
	(J_n f)(x):=\sup_{\lambda\in\Rd}\left(\int_{\Rd}f(x+\lambda h_n+y)\,\varpi_{h_n}(\d y) 
	-\eta(|\lambda|)h_n\right)
\end{equation}
for all $n\in\N$, $f\in\Cb$ and $x\in\Rd$.\ By~\cite[Section~4.5]{fuhrmann23},
there exists another strongly continuous convex monotone semigroup $(\tilde{S}(t))_{t\geq 0}$ 
on $\Cb$ given by 
\[ \tilde{S}(t)f:=\lim_{n\to\infty}J(\pi_n^t)f \quad\text{for all } t\geq 0 \text{ and } f\in\Cb \]
whose generator is given by 
\[ \tilde{A}f=\sup_{\lambda\in\Rd}\big(Bf+\lambda^T\nabla f-\eta(|\lambda|)\big)
=Bf+\eta^*(|\nabla f|) \quad\text{for all } f\in D(B)\cap\C_b^1. \]

\begin{lemma}\label{lem:wus_nisio}
	Let $\alpha:=1$, $\kappa \equiv 1$ and assume $\Cbi\subset D(B)$.\ Then, the operators  
	$(J_n)_{n\in\N}$ defined by equation~\eqref{eq:wus:nisio} satisfy the 
	Assumptions~\ref{ass:nisio},~\ref{ass:en} and~\ref{ass:en_rate}. In particular,
	\[ S(t)f=\tilde{S}(t)f \quad\text{for all } t\geq 0 \text{ and } f\in\Cb. \]
\end{lemma}
\begin{proof}
	See Appendix~\ref{app:wus_nisio}.
\end{proof}

By choosing $\nu=\varpi_{h_n}*\delta_{\lambda h_n}$ for all $\lambda\in\Rd$,
we see that $J_n f\leq I_n f$ for all $n\in\N$ and $f\in\Cb$. It has further
been shown in~\cite[Lemma~3.10 and Section~4.5]{fuhrmann23} that the iterated
operators satisfy
\[ J_n^{k^t_n}f\leq S(t)f\leq I_n^{k^t_n}f
\quad\text{for all } t\geq 0, n\in\N \text{ and } f\in\Cb.\]
In addition, if $(h_n)_{n\in\N}$ defines a sequence of refining partitions,
e.g., for $h_n:=2^{-n}$, it holds
\[ I(\pi_n^t)f\downarrow S(t)f \quad\text{and}\quad J(\pi_n^t)f\uparrow S(t)f
\quad\text{for all } t\geq 0 \text{ and } f\in\Cb \]
meaning that the first sequence is decreasing and the second one is increasing.

For our numerical example, we choose $d=1$, a bounded time horizon $T=1$, $n=30$,
$h_n=\frac{1}{30}$ and consider $(X_t)_{t\geq 0}=(W_t)_{t\geq 0}$.\ Furthermore, 
let $p=2$ and $\eta(v)=\frac{1}{6}v^6$ with convex conjugate $\eta^*(w)=\frac{5}{6}w^{6/5}$.\
Note that the mapping $v\mapsto\eta(v^{1/2})=\frac{v^3}{6}$ is convex and 
$\lim_{v\to\infty}\frac{\eta(v)}{v^2}=\infty$. Since the rescaled penalty heavily 
penalizes all measures $\nu\in\cP_p(\Rd)$ with $\W_p(\varpi_{h_n},\nu)>h_n$,
the supremum in equation~\eqref{eq:wus:chernoff} can be seen as a smoothed 
version of the supremum over a Wasserstein ball. We therefore proceed as in 
Section~\ref{sec:soc}.\ To that end, we generate $I := 1000$ i.i.d.\ midpoints
$(c_i)_{i=1,\ldots,I}\sim\cN(0,1)$ and half-widths 
$(w_i)_{i=1,\ldots,I}\sim\cU(\frac{1}{4},\frac{3}{2})$, set $K_{i,1}:=c_i-w_i$
and $K_{i,2}:=c_i+w_i$ and define the bull and bear spread functions
\begin{equation} \label{eq:wus:train}
	\begin{aligned}
		f_{i,1}(x) &=(x-K_{i,1})_+ -(x-K_{i,2})_+, & f_{i,2}(x) &=(K_{i,2}-x)_+ -(K_{i,1}-x)_+, \\
		f_{i,3}(x) &=-f_{i,1}(x), & f_{i,4}(x) &=-f_{i,2}(x),
	\end{aligned}
\end{equation}
which are split up into $90\%/10\%$ for training and testing equally among
the different types. 

We implement $I_n$ as an envelope-neural operator using the parametric representation 
of~\cite{nendel26}, where the supremum in equation~\eqref{eq:wus:chernoff} is restricted
to measures of the form $\varpi_{h_n}\circ(\id+h_n a)^{-1}$ with vector fields 
$a\in L^p(\varpi_{h_n};\R)$.\
Restricting the supremum yields the envelope-neural operator
\begin{equation} \label{eq:wus:en_param}
	(\Phi^I_n f)(x):=\max_{m=1,\ldots,M}\left(\int_{\R}f(x+y+a_m(y))\,\varpi_{h_n}(\d y) 
	-h_n\eta\left(\frac{\|a_m\|_{L^p(\varpi_{h_n};\R)}}{h_n}\right)\right)
\end{equation}
with finitely many trainable neurons $(a_m)_{m=1,\ldots,M}\subset L^p(\varpi_{h_n};\R)$.\ 
For the implementation, we choose affine functions $a_{(\varrho,\beta)}(y):=\varrho+\beta y$
with $\varrho,\beta\in\R$ for which the transported measure is given by 
$\varpi_{h_n}\circ (\id+a_{(\varrho,\beta)})^{-1}=\cN(\varrho,(1+\beta)^2 h_n)$.
Hence, for the piecewise linear payoffs defined by equation~\eqref{eq:wus:train}
the inner integral in equation~\eqref{eq:wus:en_param} is a Bachelier price and
therefore available in closed form. Moreover, the penalization term is explicitly
given by $\|a_{(\varrho,\beta)}\|^2_{L^2(\varpi_{h_n})}=|\varrho|^2+\beta^2 h_n$.
For the operator $J_n$ defined by equation~\eqref{eq:wus:nisio}, the envelope-neural 
operator is given by
\begin{equation} \label{eq:wus:en}
	(\Phi_n^J f)(x):=\max_{m=1,\ldots,M}
	\left(\int_{\R}f(x+\lambda_m h_n+y)\,\varpi_{h_n}(\d y)-\eta(|\lambda_m|)h_n\right)
\end{equation}
with finitely many trainable neurons $(\lambda_m)_{m=1,\ldots,M}\subset\R$.

The neural operators are trained by minimizing the MSE~\eqref{eq:mse} over $L=100$
i.i.d.~evaluation points $(x_l)_{l=1,\ldots,L}\sim\cN(0,\sigma^2)$ with $\sigma=3$.\
Both envelope-neural operators use $M=32$ neurons and the maximum readouts of the
equations~\eqref{eq:wus:en_param} and~\eqref{eq:wus:en}, respectively, so that 
$\Phi^I_n$ has~$64$ and $\Phi^J_n$ has~$32$ trainable parameters.\ Both are trained 
by applying the Adam algorithm over~$10^4$ epochs with learning rate $10^{-5}$
and batchsize $100$ per type $j=1,\ldots,4$.\ The training labels $I_n f_{i,j}$ of 
the envelope-neural operator $\Phi^I_n$ are computed from the dual representation 
\begin{equation} \label{eq:wus:dual}
	(I_n f)(x)=\min_{\lambda\geq 0}\left(\int_{\R}\sup_{w\in\R}
	\big(f(x+y+w)-\lambda |w|^p\big)\,\varpi_{h_n}(\d y)+G^*(\lambda)\right),
\end{equation}
see \cite[Theorem~2.4]{bartl2020}, where $G^*$ denotes the convex conjugate of 
$G(\delta):=h_n\eta (\delta^{1/p}/h_n)$.\ In our case, we have
$G(\delta)=\frac{\delta^3}{6h_n^5}$ and $G^*(\lambda)=\frac{2}{3}\sqrt{2h_n^5}\lambda^{3/2}$.\
Note that, for a one-dimensional piecewise linear payoff, the computation of the
inner supremum in equation~\eqref{eq:wus:dual} reduces to maximizing over a finite 
set of points and the integral can be computed with a one-dimensional quadrature.\
The training labels $J_n f_{i,j}$ of the envelope-neural operator $\Phi^J_n$
are obtained by evaluating the supremum in equation~\eqref{eq:wus:nisio} over
a discretized grid of drifts $\lambda\in\R$.

In order to assess the out-of-sample performance of the trained operators, 
we evaluate them on the butterfly option
\begin{equation} \label{eq:wus:butterfly}
	f^{\mathrm{bf}}(x):=(x-K_1)_+ -2(x-K_2)_+ +(x-K_3)_+ \quad\text{with}\quad
	K_1:=-1, \, K_2 := 0 \text{ and } K_3 := 1
\end{equation}
which is not monotone and therefore does not belong to any of the four classes 
of training functions defined in~\eqref{eq:wus:train}.\ As reference solution, 
we compute $S(t)f^{\mathrm{bf}}$ using a finite difference scheme for the 
corresponding HJB equation.

The results of the training are reported in Figure~\ref{fig:wus}. In Figure~\ref{fig:wus_sandwich},
we repeat the training of both neural operators for different numbers of steps~$n$ 
with step size $h_n=T/n$ and display the two terminal values $(\Phi^I_n)^{k^T_n} f^{\mathrm{bf}}$
and $(\Phi_n^J)^{k^T_n} f^{\mathrm{bf}}$ together with the finite difference 
solution $S(T) f^{\mathrm{bf}}$ at the final time $T$.\ The gap between the iterated
envelope-neural operators is decreasing in~$n$ which illustrates how the nonparametric
Wasserstein uncertainty and the parametric drift uncertainty lead to the same continuous
time limit in accordance with Lemma~\ref{lem:wus_nisio}.

\begin{figure}[ht!]
	\centering
	\begin{minipage}{0.49\textwidth}
		\centering
		\includegraphics[height=5.5cm]{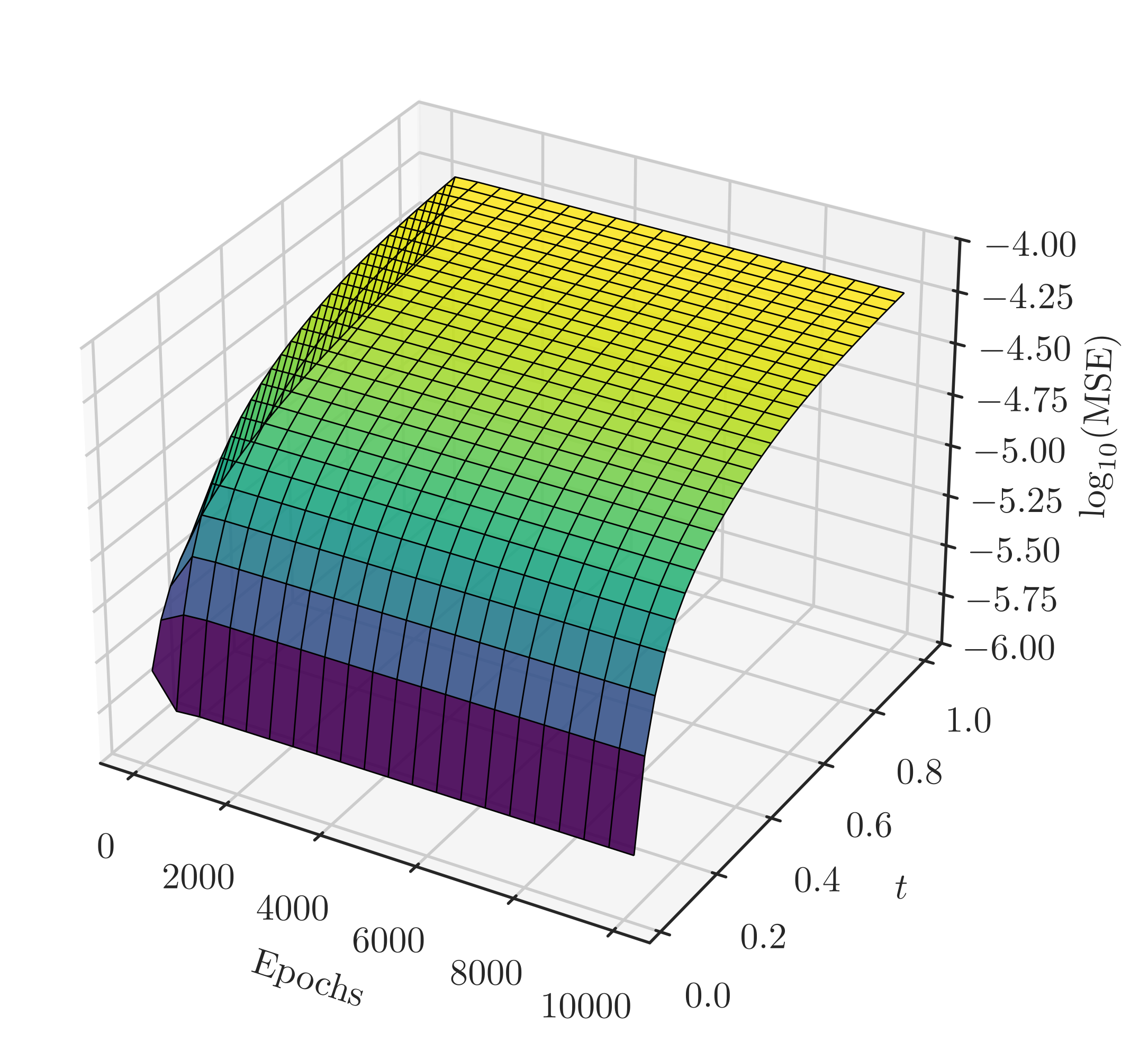}
		
		{\footnotesize ($\mathcal{I}$1) MSE between $(\Phi^I_n)^{k^t_n} f^{\mathrm{bf}}$ and $S(t) f^{\mathrm{bf}}$ along training epochs and time.}
	\end{minipage}
	\begin{minipage}{0.49\textwidth}
		\centering
		\includegraphics[height=5.5cm]{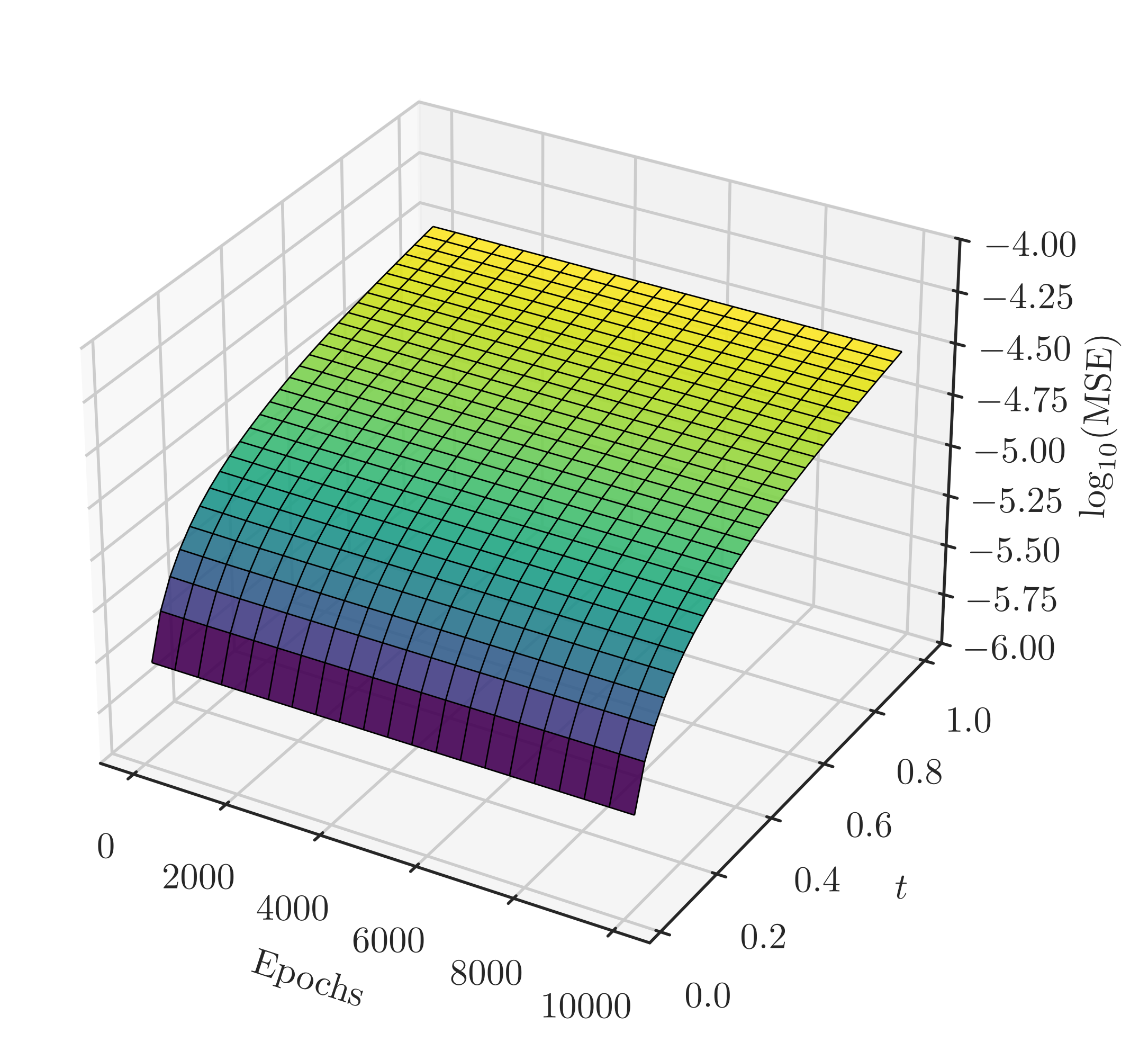}
		
		{\footnotesize ($\mathcal{J}$1) MSE between $(\Phi^J_n)^{k^t_n} f^{\mathrm{bf}}$ and $S(t) f^{\mathrm{bf}}$ along training epochs and time.}
	\end{minipage}
	\vspace{0.4cm}
	
	\begin{minipage}{0.49\textwidth}
		\centering
		\includegraphics[height=5.5cm]{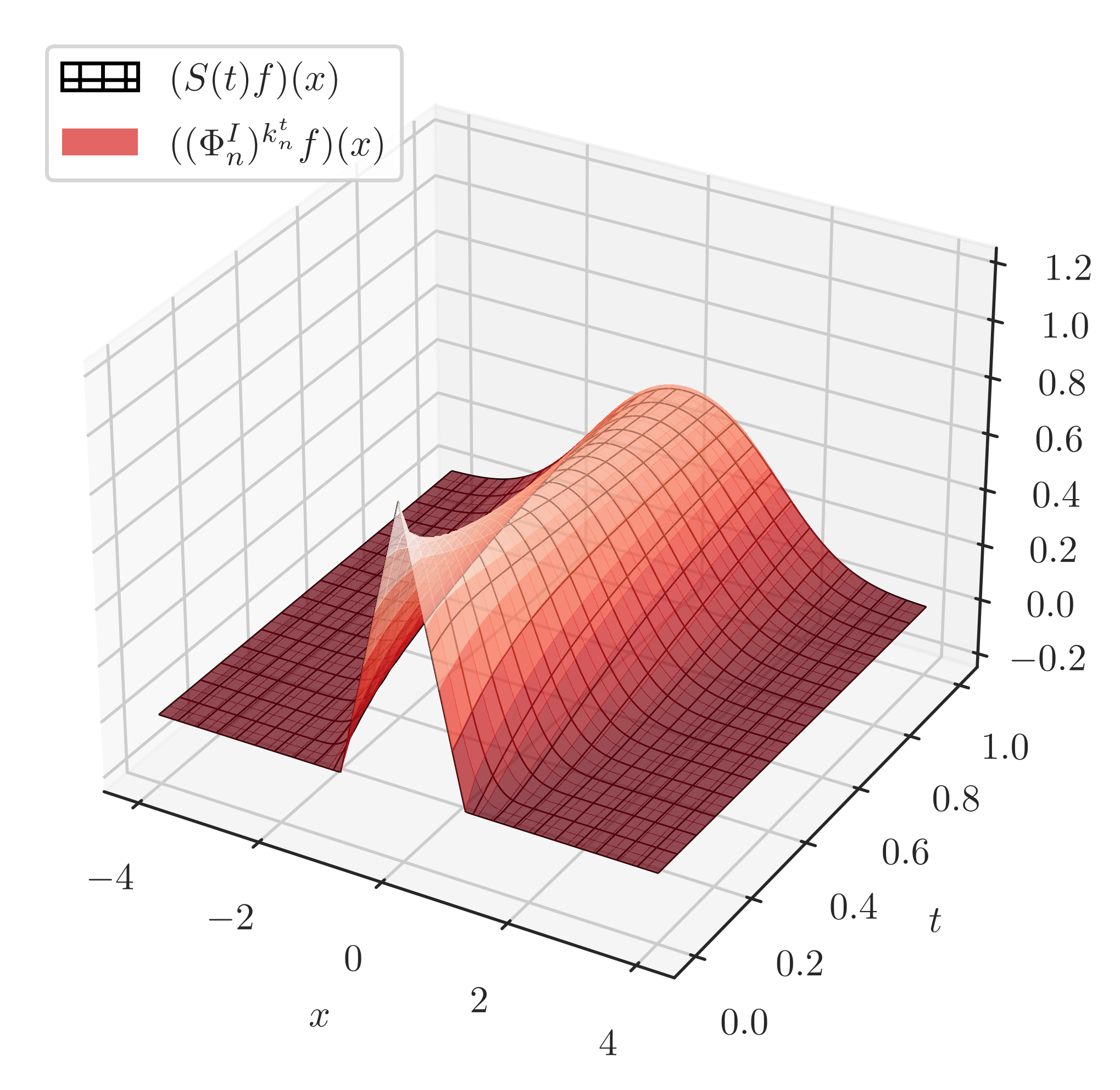}
		
		{\footnotesize ($\mathcal{I}$2) Approximation of $S(t) f^{\mathrm{bf}}$ by $(\Phi^I_n)^{k^t_n} f^{\mathrm{bf}}$.}
	\end{minipage}
	\begin{minipage}{0.49\textwidth}
		\centering
		\includegraphics[height=5.5cm]{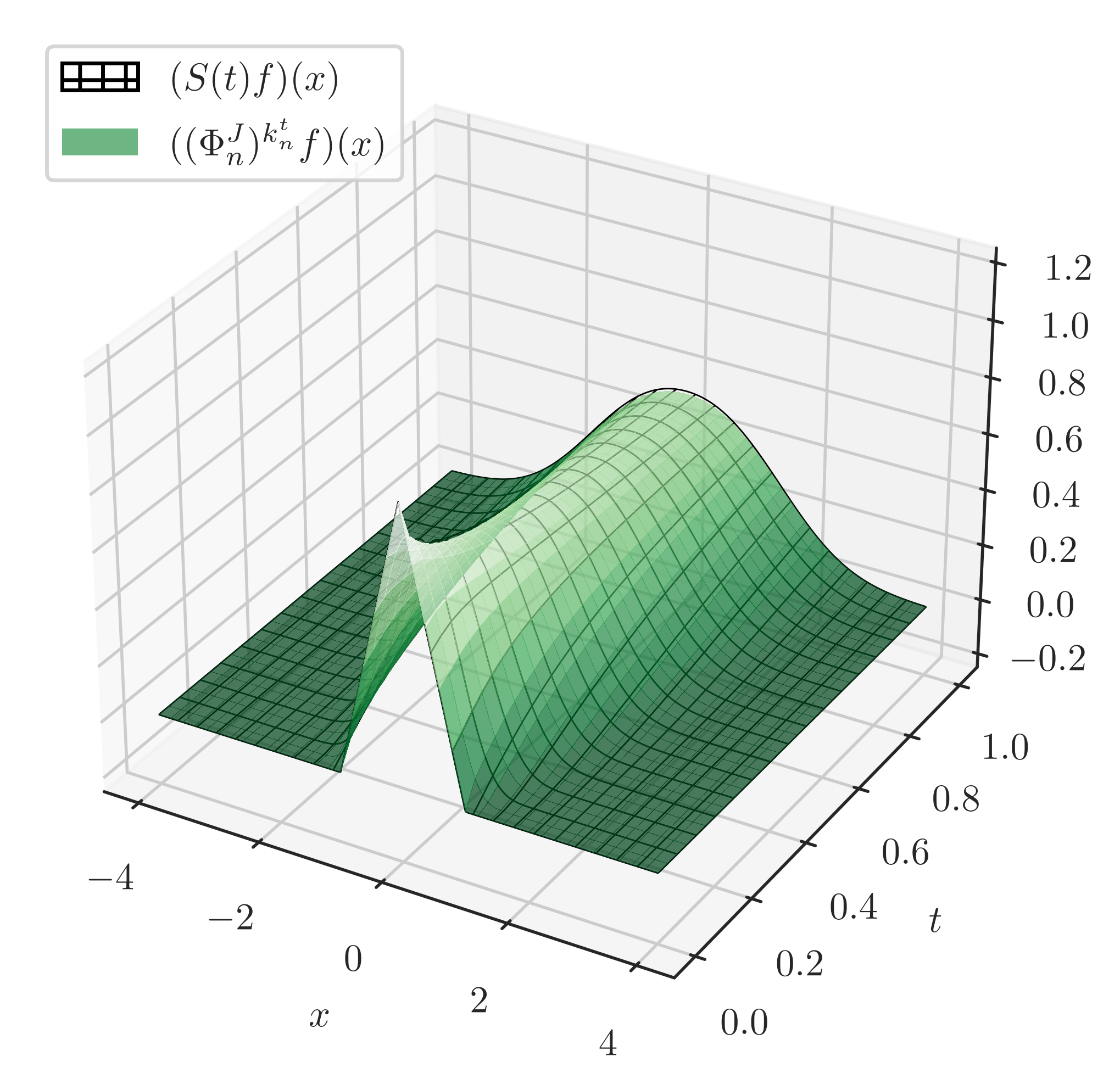}
		
		{\footnotesize ($\mathcal{J}$2) Approximation of $S(t) f^{\mathrm{bf}}$ by $(\Phi^J_n)^{k^t_n} f^{\mathrm{bf}}$.}
	\end{minipage}
	\vspace{0.4cm}
	
	\caption{Approximating the Wasserstein-perturbation semigroup $f \mapsto S(t) f$ with two envelope-neural operators, evaluated on the out-of-sample butterfly $f^{\mathrm{bf}}$ in \eqref{eq:wus:butterfly}. The operator $\Phi^I_n$ in \eqref{eq:wus:en_param} in ($\mathcal{I}$1)-($\mathcal{I}$2), and the operator $\Phi^J_n$ in \eqref{eq:wus:en} in ($\mathcal{J}$1)-($\mathcal{J}$2). In ($\mathcal{I}$1)+($\mathcal{J}$1), the MSE between the iterations $(\Phi^I_n)^{k^t_n} f^{\mathrm{bf}}$, respectively $(\Phi_n^J)^{k^t_n} f^{\mathrm{bf}}$, and $S(t) f^{\mathrm{bf}}$ is displayed along the training epochs and the time $t = k h_n$, $k = 1,\ldots,n$. In ($\mathcal{I}$2)+($\mathcal{J}$2), the approximation of $S(t) f^{\mathrm{bf}}$ by the two iterated neural operators is compared with the finite difference solution of the corresponding HJB equation over $[0,T]$.}
	\label{fig:wus}
\end{figure}

\begin{figure}[ht!]
	\centering
	\includegraphics[width=0.9\textwidth]{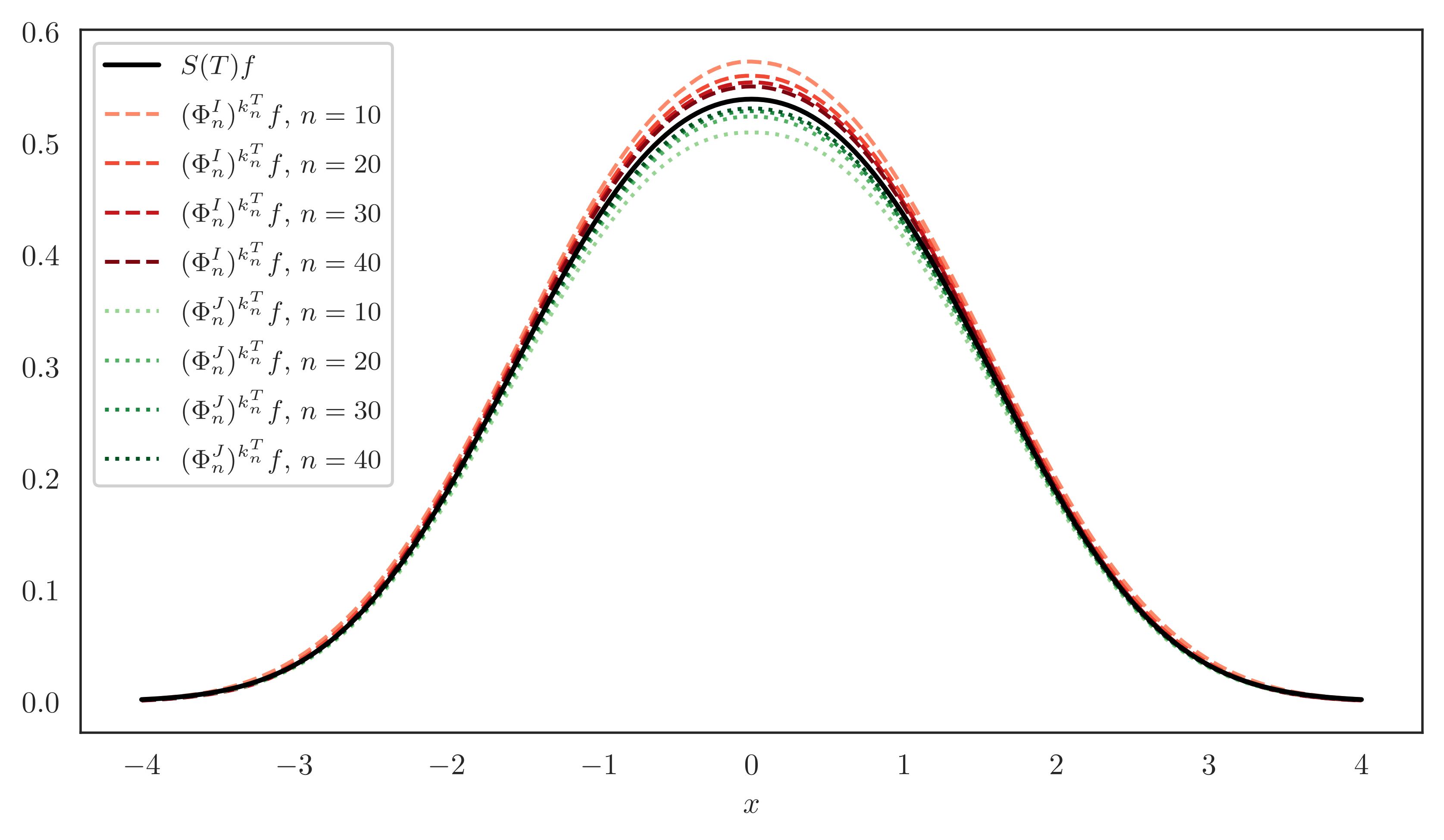}
	\caption{Terminal values $(\Phi^I_n)^{k^T_n} f^{\mathrm{bf}}$ and $(\Phi^J_n)^{k^T_n} f^{\mathrm{bf}}$ of the two envelope-neural operators of Section~\ref{sec:wus}, evaluated on the out-of-sample butterfly $f^{\mathrm{bf}}$ in \eqref{eq:wus:butterfly} and compared with the finite difference solution $S(T) f^{\mathrm{bf}}$ at the terminal time $T$. Both operators are retrained for each number of steps $n \in \{10, 20, 30, 40\}$, with step size $h_n = T/n$. The shades run from light to dark as $n$ grows.}
	\label{fig:wus_sandwich}
\end{figure}

\appendix

\section{Weighted H\"older spaces}
\label{app:holder}

We prove some basic properties of the weighted H\"older spaces from Section~\ref{sec:notation}.\
Recall that, for bounded continuous functions $\kappa',\kappa\colon\Rd\to (0,\infty)$, 
we write $\kappa'\lesssim\kappa$ if and only if $\kappa'\leq c\kappa$ for some constant $c\geq 0$,
and $\kappa'\llsim\kappa$ if and only if $\lim_{|x|\to\infty}\frac{\kappa'(x)}{\kappa(x)}=0$.
Furthermore, we define
\[ \bar{\kappa}(x,y):=\frac{2\kappa(x)\kappa(y)}{\kappa(x)+\kappa(y)}=\frac{2}{\kappa(x)^{-1}+\kappa(y)^{-1}}
\quad\text{for all } x,y\in\Rd. \]

\begin{proposition} \label{thm:hol_ban}
	Let $\alpha \in [0,1]$ and $\kappa\colon\Rd\to (0,\infty)$ be a bounded continuous function.\
	Then, the space $(\C^\alpha_\kappa,\|\cdot\|_{\alpha,\kappa})$ is complete.
\end{proposition}
\begin{proof}
	Let $(f_n)_{n\in\N}\subset\C^\alpha_\kappa$ be a Cauchy sequence.\ Since $(f_n)_{n\in\N}$ 
	is also a Cauchy sequence in the Banach space $\C_\kappa$, there exists $f\in\C_\kappa$ with 
	$\|f-f_n\|_\kappa\to 0$. By definition of the norm, we obtain
	\[ \|f\|_{\alpha,\kappa}\leq\sup_{n\in\N}\|f_n\|_{\alpha,\kappa}<\infty. \]
	Moreover, for every $\epsilon>0$, there exist $n_0\in\N$ with $\|f_m-f_n\|_{\alpha,\kappa}<\epsilon$ 
	for all $m,n\geq n_0$. Hence, for every $x,y\in\Rd$ with $x\neq y$ and $n\geq n_0$, 
	we obtain 
	\begin{align*}
		\frac{|(f-f_n)(x)-(f-f_n)(y)|}{|x-y|^\alpha}\bar{\kappa}(x,y)
		&=\lim_{m\to\infty}\frac{|(f_m-f_n)(x)-(f_m-f_n)(y)|}{|x-y|^\alpha}\bar{\kappa}(x,y) \\
		&\leq\lim_{m\to\infty}\|f_m-f_n\|_{\alpha,\kappa}\leq\epsilon. \qedhere
	\end{align*}
\end{proof}

\begin{theorem} \label{thm:hol_cont}
	Let $0\leq\alpha'\leq\alpha\leq 1$ and $\kappa'\lesssim\kappa\colon\Rd\to (0,\infty)$ be bounded 
	continuous functions. Then, the embedding $\C^\alpha_\kappa\hookrightarrow \C^{\alpha'}_{\kappa'}$ 
	is continuous.\ Moreover, if $\alpha'<\alpha$ and $\kappa'\llsim\kappa$, then 
	$\C^\alpha_\kappa\hookrightarrow\c^{\alpha'}_{\kappa'}$. 
\end{theorem}
\begin{proof}
	Let $f \in \C^\alpha_\kappa$. For every $x,y\in\Rd$ with $|x-y|\leq 1$, it follows 
	from $\alpha'\leq\alpha$ and $\kappa'\leq c\kappa$ that 
	$\bar{\kappa}'(x,y)=\frac{2}{\kappa'(x)^{-1}+\kappa'(y)^{-1}}
	\leq c\frac{2}{\kappa(x)^{-1}+\kappa(y)^{-1}}\leq c\bar{\kappa}(x,y)$ and therefore
	\[ \frac{|f(x)-f(y)|}{|x-y|^{\alpha'}}\bar{\kappa}'(x,y) 
	\leq c\frac{|f(x)-f(y)|}{|x-y|^\alpha}\bar{\kappa}(x,y). \]
	Moreover, for every $x,y\in\Rd$ with $|x-y|>1$, the inequality 
	$\bar{\kappa}'(x,y) \leq 2\min(\kappa'(x),\kappa'(y))$ implies
	\[ \frac{|f(x)-f(y)|}{|x-y|^{\alpha'}}\bar{\kappa}'(x,y) 
	\leq 2|f(x)-f(y)|\min(\kappa'(x),\kappa'(y))
	\leq 2(|f(x)|\kappa'(x)+|f(y)|\kappa'(y)). \]
	Hence, it holds $\|f\|_{\alpha',\kappa'}\leq 4c\|f\|_{\alpha,\kappa}$ for all $f\in\C^\alpha_\kappa$.
	
	Now, let $\alpha'<\alpha$ and $\kappa'\llsim\kappa$. For every $f\in\C^\alpha_\kappa$,
	it holds
	\[ \lim_{|x|\to\infty}|f(x)|\kappa'(x)
	\leq\|f\|_{\alpha,\kappa}\lim_{|x|\to\infty}\frac{\kappa'(x)}{\kappa(x)}=0. \]
	In addition, by using that $\kappa'\leq c\kappa$ implies $\bar{\kappa}'(x,y)\leq c\bar{\kappa}(x,y)$, we obtain
	\begin{align*}
		\lim_{\delta\to 0}\sup_{x\neq y,\atop |x-y|<\delta}\frac{|f(x)-f(y)|}{|x-y|^{\alpha'}}\bar{\kappa}'(x,y)
		&\leq c \lim_{\delta\to 0}\bigg(\delta^{\alpha-\alpha'} 
		\sup_{x\neq y,\atop |x-y|<\delta}\frac{|f(x)-f(y)|}{|x-y|^\alpha}\bar{\kappa}(x,y)\bigg) \\
		&\leq c\|f\|_{\alpha,\kappa}\lim_{\delta\to 0}\delta^{\alpha-\alpha'}=0. \qedhere 
	\end{align*}
\end{proof}

\begin{theorem} \label{thm:hol_cpt}
	Let $0\leq\alpha'<\alpha\leq 1$ and $\kappa'\llsim\kappa\colon\Rd\to (0,\infty)$
	be bounded continuous functions.\ Then, the embedding 
	$\C^\alpha_\kappa\hookrightarrow\C^{\alpha'}_{\kappa'}$ is compact.
\end{theorem}
\begin{proof}
	Let $(f_n)_{n\in\N}\subset\C^\alpha_\kappa$ be bounded and $C:=\sup_{n\in\N}\|f_n\|_{\alpha,\kappa}$.\
	By Arzel\`a-Ascoli's theorem, there exists a diagonal sequence, which is again 
	denoted by $(f_n)_{n\in\N}$, and $f\in\C(\Rd)$ such that 
	\[\lim_{n\to\infty}\|f-f_n\|_{\infty,K}=0 \quad\text{for all } K\Subset\Rd.\]
	For every $x,y\in\Rd$, we use $\|f_n\|_{\alpha,\kappa}\leq C$ to obtain 
	\begin{align*}
		|f(x)|\kappa(x) &=\lim_{n\to\infty}|f_n(x)|\kappa(x)\leq C, \\
		|f(x)-f(y)|\bar{\kappa}(x,y)
		&=\lim_{n\to\infty}|f_n(x)-f_n(y)|\bar{\kappa}(x,y)\leq C|x-y|^\alpha
	\end{align*}
	which guarantees that $f\in \C^\alpha_\kappa$ and $\|f-f_n\|_{\alpha,\kappa}\leq 2C$.
	It remains to show that $\|f-f_n\|_{\alpha',\kappa'}\to 0$. Let $\epsilon>0$.\ We choose
	$r>0$ with $\sup_{|x|>r}\frac{\kappa'(x)}{\kappa(x)}<\frac{\epsilon}{4C}$ and 
	$n_0\in\N$ with $\|f-f_n\|_{\infty,B_{\R^d}(r)}<\frac{\epsilon}{2\|\kappa'\|_{\infty}}$. 
	Then, for every $n\geq n_0$, we obtain
	\begin{equation} \label{eq:AA1}
		\| f-f_n\|_{\kappa'} 
		\leq\|\kappa'\|_{\infty}\sup_{x\in B_{\R^d}(r)}|(f-f_n)(x)|+2C\sup_{|x|>r}\frac{\kappa'(x)}{\kappa(x)}
		<\|\kappa'\|_{\infty}\frac{\epsilon}{2\|\kappa'\|_{\infty}}+2C\frac{\epsilon}{4C}=\epsilon.
	\end{equation}
	Furthermore, the estimates $\kappa'\leq c\kappa$ and 
	$\bar{\kappa}'(x,y)\leq 2\min(\kappa'(x),\kappa'(y))$ imply
	\begin{align*}
		&\sup_{x\neq y}\frac{|(f-f_n)(x)-(f-f_n)(y)|}{|x-y|^{\alpha'}}\bar{\kappa}'(x,y) \\
		&\leq\bigg(\sup_{x\neq y}\frac{|(f-f_n)(x)-(f-f_n)(y)|}{|x-y|^\alpha}\bar{\kappa}'(x,y)\bigg)^\frac{\alpha'}{\alpha} 
		\bigg(\sup_{x\neq y}|(f-f_n)(x)-(f-f_n)(y)|\bar{\kappa}'(x,y)\bigg)^{1-\frac{\alpha'}{\alpha}} \\
		&\leq\bigg(c\sup_{x\neq y}\frac{|(f-f_n)(x)-(f-f_n)(y)|}{|x-y|^\alpha}\bar{\kappa}(x,y)\bigg)^\frac{\alpha'}{\alpha} 
		\bigg(2\sup_{x\in\Rd}|(f-f_n)(x)|\kappa'(x)\bigg)^{1-\frac{\alpha'}{\alpha}} \\
		&\leq (2Cc)^\frac{\alpha'}{\alpha}(2\|f-f_n\|_{\kappa'})^{1-\frac{\alpha'}{\alpha}}.
	\end{align*}
	Combining the previous estimate with inequality~\eqref{eq:AA1} yields $\|f-f_n\|_{\alpha',\kappa'}\to 0$.
\end{proof}

\begin{theorem} \label{thm:hol_dense}
	Let $\alpha\in [0,1)$ and $\kappa\colon\Rd\to (0,\infty)$ be a bounded 
	continuous function. Then, the set $\Cci\subset\c^\alpha_\kappa$ is dense.
\end{theorem}
\begin{proof}
	First, we show that every $f\in\c^\alpha_\kappa$ can be approximated by a sequence
	of compactly supported functions.\ To do so, let $\chi\in\Cbi$ with $0\leq\chi\leq 1$, $\chi\equiv 1$
	on $B_{\Rd}(1)$ and $\chi\equiv 0$ on $B_{\Rd}(2)^c$.\ We define $\chi_r(x):=\chi(x/r)$, 
	$f_r(x):=\chi_r(x)f(x)$ and $h_r(x):=f(x)-f_r(x)=(1-\chi_r(x))f(x)$ for all $r>0$ and $x\in\Rd$.
	It holds $f_r,h_r\in\c^\alpha_\kappa$ for all $r>0$. Since $f\in\c^\alpha_\kappa$, 
	we have
	\begin{align} 
		m_f(r) &:=\sup_{|x|>r}|f(x)|\kappa(x) \to 0 \quad\text{ as } r\to\infty, \label{eq:thm:hol_dense:proof1a} \\
		\omega_f(\delta) &:=\sup_{x\neq y,\atop |x-y|<\delta}\frac{|f(x)-f(y)|}{|x-y|^\alpha}\bar{\kappa}(x,y)
		\to 0 \quad\text{as } \delta\to 0. \label{eq:thm:hol_dense:proof1b}
	\end{align}
	It follows from equation~\eqref{eq:thm:hol_dense:proof1a} that
	$\|h_r\|_\kappa=\|f-f_r\|_\kappa\leq m_f(r)\to 0$ as $r\to\infty$.\
	Moreover, there exists $L\geq 0$ with $\|\nabla\chi_r\|_\infty\leq L/r$ for all $r>0$.\
	For every $x,y\in\Rd$ with $0<|x-y|<\delta$, we use equation~\eqref{eq:thm:hol_dense:proof1b},
	$\chi_r(x)=\chi_r(y)=1$ in case that $|y|\leq r-\delta$ and $\bar{\kappa}(x,y)\leq 2\kappa(y)$ 
	to estimate
	\begin{align*}
		|h_r(x)-h_r(y)|\bar{\kappa}(x,y) 
		&\leq (1-\chi_r(x))|f(x)-f(y)|\bar{\kappa}(x,y)+2|\chi_r(x)-\chi_r(y)|\cdot|f(y)|\kappa(y) \\
		&\leq\omega_f(\delta)|x-y|^\alpha+\frac{2L}{r}|x-y| m_f(r-\delta)
	\end{align*}
	which implies that
	\[ \sup_{x\neq y,\atop |x-y|<\delta}\frac{|h_r(x)-h_r(y)|}{|x-y|^\alpha}\bar{\kappa}(x,y) 
	\leq\omega_f(\delta)+\frac{2L}{r}\delta^{1-\alpha}m_f(r-\delta). \]
	On the other hand, for every $x,y\in\Rd$ with $|x-y|\geq\delta$, it holds
	\[ \sup_{x\neq y,\atop |x-y|\geq\delta}\frac{|h_r(x)-h_r(y)|}{|x-y|^\alpha}\bar{\kappa}(x,y)
	\leq\frac{2\sup_{x\in\Rd}|h_r(x)|\kappa(x)}{\delta^\alpha}
	\leq \frac{2m_f(r)}{\delta^\alpha}. \]
	Letting first $r\to\infty$ and then $\delta\to 0$ yields 
	$\|f-f_r\|_{\alpha,\kappa}=\|h_r\|_{\alpha,\kappa}\to 0$.
	
	Second, we show that every $f\in\c^\alpha_\kappa$ with compact support can be
	approximated by a sequence of functions in $\Cci$.\ Since $\kappa$ is bounded away
	from zero on compact subsets, the weighted H\"older norm is equivalent to the classical
	H\"older norm on $\supp(f)$, whence a standard mollification argument yields a sequence 
	$(f_n)_{n\in\N}\subset\Cci$ with $\|f-f_n\|_{\alpha,\kappa}\to 0$, see, 
	e.g.,~\cite[Proposition 0.2.1]{lunardi95}. 
\end{proof}

\section{Weighted Stone-Weierstrass theorems for Banach scales}
\label{app:stone_wstrass}

Let $(\Theta,\preceq)$ be a pre-ordered set equipped with an additional 
relation $\prec$ satisfying $\theta_1\prec\theta_2\Rightarrow\theta_1\preceq\theta_2$ 
such that, for every $k\in\N$ and $\theta_1,\ldots,\theta_k\in\Theta$, there exist 
$l_\star,l^\star\in\{1,\ldots,k\}$ with $\theta_{l_\star}\preceq\theta_l\preceq\theta_{l^\star}$ 
for all $l\in\{1,\ldots,k\}$.\ Let $X$ be a Banach space and $(X_\theta,\|\cdot\|_\theta)_{\theta\in\Theta}$ 
be a family of Banach spaces such that, for every $\theta'\preceq\theta$, 
it holds that $X_\theta\subset X_{\theta'}\subset X$ and the inclusion 
$\iota_{X_\theta,X_{\theta'}}\colon X_\theta\hookrightarrow X_{\theta'}$ is continuous, 
and for every $\theta'\prec\theta$, the set $B_{X_\theta}(r):=\{u\in X_\theta\colon\|u\|_\theta\leq r\}$
is compact in $X_{\theta'}$.\ Moreover, for a strictly increasing continuous function 
$\psi\colon [0,\infty)\to (0,\infty)$ with $\lim_{R\to\infty}\psi(R)=\infty$, 
we define the weight functions
\[ \psi_\theta\colon X_\theta\to (0,\infty),\; u \mapsto\psi(\|u\|_\theta)
\quad\text{for all } \theta\in\Theta. \]
Then, for every $\theta'\prec\theta$ and $R>0$, the pre-image 
$\psi_\theta^{-1}((0,R])=B_{X_\theta}(\psi^{-1}(R))\subset X_{\theta,\theta'}$ is compact
and therefore $\psi_\theta$ is an admissible weight on $X_{\theta,\theta'}:=(X_\theta,\|\cdot\|_{\theta'})$ 
in the sense of~\cite[Definition~2.1]{cuchiero26}. Furthermore, let $Y$ be a vector space and
$(Y_\theta,\|\cdot\|_\theta)_{\theta\in\Theta}$ be a family of Banach spaces such that, 
for every $\theta'\preceq\theta$, it holds $Y_\theta\subset Y_{\theta'}\subset Y$ and
the inclusion $\iota_{Y_\theta,Y_{\theta'}}\colon Y_\theta\hookrightarrow Y_{\theta'}$ is continuous.

\begin{remark}
	In the proof of Theorem~\ref{thm:uat_chernoff}, we consider $\Theta:=(\beta_q,\eta_q)_{q \in Q} \subset [0,\alpha] \times \mathcal{K}$, with totally ordered $(Q,\leq)$ satisfying $\beta_{q'}<\beta_q$ and $\eta_{q'}\llsim\eta_q$ for all $q' < q$. We equip $\Theta$ with
	\[ (\beta_{q'},\eta_{q'})\preceq (\beta_q,\eta_q)\, :\Leftrightarrow\,q'\leq q
	\quad\text{and}\quad
	(\beta_{q'},\eta_{q'})\prec (\beta_q,\eta_q)\, :\Leftrightarrow\,q'<q. \]
	Moreover, we set 
	$X_{\theta,\theta'}:=\C^{\beta_q,\beta_q'}_{\eta_q,\eta_q'}=(\C^{\beta_q}_{\eta_q},\|\cdot\|_{\beta_q',\eta_q'})$ 
	and $Y_{\theta'}:=\c^{\beta_q'}_{\eta_q'}$ for $\theta':=(\beta_q',\eta_q')\prec\theta:=(\beta_q,\eta_q)$.
\end{remark}

We further introduce the intersection mapping spaces
\begin{align*}
	\bigcap\nolimits_{\theta'\prec\theta}\B_\psi(X_{\theta,\theta'};Y_{\theta'}) 
	&:=\{f\colon X\to Y\colon f\vert_{X_\theta}\in\B_\psi(X_{\theta,\theta'};Y_{\theta'}) \text{ for all } \theta'\prec\theta\in\Theta\}, \\
	\bigcap\nolimits_{\theta'\prec\theta}\Cb(X_{\theta,\theta'};Y_{\theta'}) 
	&:=\{f\colon X\to Y\colon f\vert_{X_\theta}\in\Cb(X_{\theta,\theta'};Y_{\theta'}) \text{ for all } \theta'\prec\theta\in\Theta\}
\end{align*}
which satisfy 
$\bigcap_{\theta'\prec\theta}\Cb(X_{\theta,\theta'};Y_{\theta'})\subset
\bigcap_{\theta'\prec\theta}\B_\psi(X_{\theta,\theta'};Y_{\theta'})$ 
since $\C_b(X_{\theta,\theta'};Y_{\theta'})\subset\B_\psi(X_{\theta,\theta'};Y_{\theta'})$.
We equip $\bigcap_{\theta'\prec\theta}\B_\psi(X_{\theta,\theta'};Y_{\theta'})$ 
with the locally convex topology generated by the seminorms
\[\|f\|_{\B_\psi(X_{\theta,\theta'};Y_{\theta'})}
:=\sup_{u\in X_\theta}\frac{\|f(u)\|_{\theta'}}{\psi(\|u\|_\theta)}
\quad\text{for all } \theta'\prec\theta\in\Theta. \]
A zero-neighborhood basis of $\bigcap_{\theta'\prec\theta}\B_\psi(X_{\theta,\theta'};Y_{\theta'})$
is given by sets of the form
\[ \bigg\{f\in\bigcap_{\theta'\prec\theta}\B_\psi(X_{\theta,\theta'};Y_{\theta'})\colon
\max_{l=1,\ldots,k}\|f\|_{\B_{\psi_{\theta_l}}(X_{\theta_l,\theta_l'};Y_{\theta_l'})}<\epsilon\bigg\}, \]
where $\epsilon>0$, $k\in\N$ and $\theta_l'\prec\theta_l\in\Theta$, $l = 1,\ldots,k$.\
Recall that a vector space $\A$ of mappings $a\colon X\to\R$ is a \emph{subalgebra} 
if $a_1\cdot a_2\in\A$ for all $a_1,a_2\in\A$, where $(a_1\cdot a_2)(u):=a_1(u)a_2(u)$.\ 
The set~$\A$ is \emph{point separating} if, for every $u_1,u_2\in X$ with $u_1\neq u_2$, 
there is $a\in\A$ with $a(u_1)\neq a(u_2)$.\ Finally, $\A$ is called 
\emph{nowhere vanishing} if, for every $u\in X$, there is $a\in\A$ with $a(u)\neq 0$.

\begin{theorem}[$\R$-valued weighted Stone-Weierstrass] \label{thm:stone_wstrass_R}
	Let $\A\subset\bigcap_{\theta'\prec\theta}\Cb(X_{\theta,\theta'};\R)$ 
	be a point separating nowhere vanishing subalgebra. Then, the set~$\A$
	is dense in $\bigcap_{\theta'\prec\theta}\B_\psi(X_{\theta,\theta'};\R)$.
\end{theorem}
\begin{proof}
	We follow the proof of~\cite[Theorem~3.9]{cuchiero26}.\
	Let $f\in\bigcap_{\theta'\prec\theta}\B_\psi(X_{\theta,\theta'};\R)$, $\epsilon>0$, 
	$k\in\N$ and $\theta_l'\prec\theta_l\in\Theta$ for $l=1,\ldots,k$. 
	By~\cite[Lemma~2.7(i)]{cuchiero26}, there exists $R_1>0$ with
	\begin{equation} \label{eq:thm:stone_wstrass_R:proof0}
		\max_{l=1,\ldots,k}\sup_{u\in X_{\theta_l}\setminus K_{l,R_1}} 
		\frac{1+|f(u)|}{\psi(\|u\|_{\theta_l})}<\frac{\epsilon}{5},
	\end{equation} 
	where $K_{l,R_1} := \psi_{\theta_l}^{-1}((0,R_1])$ denotes the compact pre-image of $\psi_{\theta_l}(\cdot) := \psi(\Vert \cdot \Vert_{\theta_l})$. Moreover, we define $M:=(\min_{l=1,\ldots,k}\inf_{u\in X_{\theta_l}}\psi(\|u\|_{\theta_l}))^{-1}$, let $b:=1+\max_{l=1,\ldots,k}\sup_{u\in K_{l,R_1}}|f(u)|$ and choose $R_2\geq\max\{R_1,\frac{5b}{\epsilon}\}$.\ Furthermore, there exists $\theta'\in\Theta$
	with $\theta'\preceq\theta_l'$ for all $l\in\{1,\ldots,k\}$. 
	Since the set $K_2:=\bigcup_{l=1}^k K_{l,R_2}\subset X_{\theta'}$ is compact,
	we can apply the classical Stone--Weierstrass theorem~\cite{stone48} to 
	$\A\vert_{K_2}\subset\C(K_2,\|\cdot\|_{\theta'})$ to obtain $a\in\A$ with
	\begin{equation} \label{eq:thm:stone_wstrass_R:proof1}
		\sup_{u\in K_2}|f(u)-a(u)|<\min\left\{1,\frac{\epsilon}{5M}\right\}.
	\end{equation}
	We use that $g(s):=\max\{\min(s,b),-b\}$
	satisfies $g(a(u))=a(u)$ for all $u\in K_1:=\bigcup_{l=1}^k K_{l,R_1}$,
	the inequality $|a(u)|\leq 1+|f(u)|$ for all $u\in K_2$ and the 
	inequalities \eqref{eq:thm:stone_wstrass_R:proof0}--\eqref{eq:thm:stone_wstrass_R:proof1}
	to estimate
	\begin{align}
		&\max_{l=1,\ldots,k}\|f-g\circ a\|_{\B_{\psi_{\theta_l}}(X_{\theta_l,\theta_l'};\R)}
		=\max_{l=1,\ldots,k}\sup_{u\in X_{\theta_l}}\frac{|f(u)-g(a(u))|}{\psi(\|u\|_{\theta_l})}
		\nonumber \\
		&\leq\max_{l=1,\ldots,k}\sup_{u\in X_{\theta_l}\setminus K_{l,R_1}}\frac{|f(u)|}{\psi(\|u\|_{\theta_l})}
		+\max_{l=1,\ldots,k}\sup_{u\in K_{l,R_1}}\frac{|f(u)-g(a(u))|}{\psi(\|u\|_{\theta_l})} 
		\nonumber \\
		& \quad\; +\max_{l=1,\ldots,k}\sup_{u\in K_{l,R_2}\setminus K_{l,R_1}} 
		\frac{|g(a(u))|}{\psi(\|u\|_{\theta_l})} 
		+ \max_{l=1,\ldots,k}\sup_{u\in X_{\theta_l}\setminus K_{l,R_2}} 
		\frac{|g(a(u))|}{\psi(\|u\|_{\theta_l})} \nonumber \\
		&\leq\max_{l=1,\ldots,k}\sup_{u\in X_{\theta_l}\setminus K_{l,R_1}} 
		\frac{1+|f(u)|}{\psi(\|u\|_{\theta_l})} 
		+M\sup_{u\in K_{l,R_1}}|f(u)-a(u)| \nonumber \\
		&\quad\; +\max_{l=1,\ldots,k} \sup_{u\in K_{l,R_2}\setminus K_{l,R_1}} 
		\frac{1+|f(u)|}{\psi(\|u\|_{\theta_l})} 
		+ \max_{l=1,\ldots,k}\sup_{u\in X_{\theta_l}\setminus K_{l,R_2}} 
		\frac{b}{\psi(\|u\|_{\theta_l})} \nonumber \\
		&<\frac{\epsilon}{5}+M\frac{\epsilon}{5M}+\frac{\epsilon}{5} +\frac{b}{R_2}
		\leq\frac{4\epsilon}{5}. \label{eq:thm:stone_wstrass_R:proof2}
	\end{align}
	Moreover, by the Weierstrass theorem~\cite{weierstrass85}, there exists a polynomial~$p$ on $\R$ with
	\[ |g(s)-p(s)|<\frac{\epsilon}{5M} \quad\text{for all } |s|\leq 
	c:= \max_{l=1,\ldots,k}\sup_{u\in X_{\theta_l}}|a(u)|. \]
	Without loss of generality, we assume $p(0)=0$, implying
	\begin{align}
		&\max_{l=1,\ldots,k}\|g\circ a-p\circ a\|_{\B_{\psi_{\theta_l}}(X_{\theta_l,\theta_l'};\R)}
		=\max_{l=1,\ldots,k}\sup_{u\in X_{\theta_l}}\frac{|g(a(u))-p(a(u))|}{\psi(\|u\|_{\theta_l})} 
		\nonumber \\
		&\leq M\max_{l=1,\ldots,k}\sup_{u\in X_{\theta_l}}|g(a(u))-p(a(u))|
		\leq M\sup_{|s|\leq c}|g(s)-p(s)|
		<\frac{M\epsilon}{5M}=\frac{\epsilon}{5}. \label{eq:thm:stone_wstrass_R:proof3}
	\end{align}
	Combining the inequalities~\eqref{eq:thm:stone_wstrass_R:proof2}
	and~\eqref{eq:thm:stone_wstrass_R:proof3} yields
	$\max_{l=1,\ldots,k}\|f-p\circ a\|_{\B_\psi(X_{\theta_l,\theta_l'};\R)}<\epsilon$.
	Since~$\A$ is an algebra, it holds that $p\circ a\in\A$ and the claim follows.
\end{proof}

Following~\cite[Section~3.2]{cuchiero26}, Theorem~\ref{thm:stone_wstrass_R} 
can be generalized to $\A\subset\bigcap_{\theta'\prec\theta}\B_\psi(X_{\theta,\theta'};\R)$ 
consisting of unbounded mappings.\ Given a vector space~$\A$ of mappings $a\colon X\to\R$,
a vector space~$\W$ of mappings $w\colon X\to Y$ is an \emph{$\A$-submodule} if
$a\cdot w\in\W$ for all $a\in\A$ and $w\in\W$, where $(a\cdot w)(u):=a(u)w(u)$.

\begin{theorem}[Vector-valued weighted Stone--Weierstrass] \label{thm:stone_wstrass_vector}
	Let $\A\subset\bigcap_{\theta'\prec\theta}\Cb(X_{\theta,\theta'};\R)$ be a point
	separating nowhere vanishing subalgebra and let 
	$\W\subset\bigcap_{\theta'\prec\theta}\Cb(X_{\theta,\theta'};Y_{\theta'})$ be an 
	$\A$-submodule such that $\W(u):=\{w(u)\colon w\in\W\}$ is a dense subset of $Y_{\theta'}$
	for all $u\in X_\theta$ and $\theta',\theta\in\Theta$ with $\theta'\prec\theta$.\ 
	Then, the set~$\W$ is dense in
	$\bigcap_{\theta'\prec\theta}\B_\psi(X_{\theta,\theta'};Y_{\theta'})$.
\end{theorem}
\begin{proof}
	Let $f\in\bigcap_{\theta'\prec\theta}\B_\psi(X_{\theta,\theta'};Y_{\theta'})$, 
	$\epsilon>0$, $k\in\N$ and $\theta_l'\prec\theta_l\in\Theta$ for $l=1,\ldots,k$.\
	It follows from~\cite[Lemma~2.7(i)]{cuchiero26} that there exists $R_1>0$ with
	\begin{equation} \label{eq:thm:stone_wstrass_vector:proof1}
		\max_{l=1,\ldots,k}\sup_{u\in X_{\theta_l}\setminus K_{l,R_1}} 
		\frac{1+\|f(u)\|_{\theta_l'}}{\psi(\|u\|_{\theta_l})}<\frac{\epsilon}{4},
	\end{equation}
	where $K_{l,R_1}:=\psi_{\theta_l}^{-1}((0,R_1])$.\
	Let $M:=(\min_{l=1,\ldots,k}\inf_{u\in X_{\theta_l}}\psi(\|u\|_{\theta_l}))^{-1}$.\
	We choose $\theta'\in\Theta$ with $\theta'\preceq\theta_l'$ for all $l\in\{1,\ldots,k\}$
	and observe that $K_1:=\bigcup_{l=1}^k K_{l,R_1}$ is compact in $X_{\theta'}$.\ Moreover,
	for every $u\in K_1$, we define $L_u:=\{l\in\{1,\ldots,k\}\colon u\in X_{\theta_l}\}$
	and choose $l_u\in L_u$ with $\theta_l'\preceq\theta_{l_u}'$ for all $l\in L_u$.\
	Since $\W(u)$ is a dense subset of $Y_{\theta_{l_u}'}$, there exists $w_u\in\W$ with
	\[ \|f(u)-w_u(u)\|_{\theta_{l_u}'}<\frac{\epsilon}{4c_u M}, \]
	where $c_u :=1+\max_{l\in L_u}\|\iota_{\theta_{l_u}',\theta_l'}\|_{L(Y_{\theta_{l_u}'};Y_{\theta_l'})}$.
	Consequently, for every $l\in L_u$, it holds that
	\[ \|f(u)-w_u(u)\|_{\theta_l'}
	\leq\|\iota_{\theta_{l_u}',\theta_l'}\|_{L(Y_{\theta_{l_u}'};Y_{\theta_l'})}\|f(u)-w_u(u)\|_{\theta_{l_u}'}
	<c_u\frac{\epsilon}{4c_uM}=\frac{\epsilon}{4M} \]
	and therefore
	\[ \frac{\|f(u)-w_u(u)\|_{\theta_l'}}{\psi(\|u\|_{\theta_l})}<M\frac{\epsilon}{4M}=\frac{\epsilon}{4}. \]
	In addition, for every $u\in K_1$ and $l\in\{1,\ldots,k\}$, the set
	\[ F_{u,l}:=\left\{v\in X_{\theta_l}\colon
	\frac{\|f(v)-w_u(v)\|_{\theta_l'}}{\psi(\|v\|_{\theta_l})}\geq\frac{\epsilon}{4}\right\} \]
	is compact in $X_{\theta'}$.\ Indeed, by applying~\cite[Lemma~2.7(i)]{cuchiero26} to~$f$ 
	and using that $w_u$ is uniformly bounded, we obtain 
	\[ \lim_{R\to\infty}\sup_{v\in X_{\theta_l}\setminus K_{l,R}} 
	\frac{\|f(v)-w_u(v)\|_{\theta_l'}}{\psi(\|v\|_{\theta_l})}=0 \]
	which implies that $F_{u,l}\subset K_{l,R}$ for sufficiently large $R>0$.\ Since $K_{l,R}$
	is compact and $F_{u,l}$ is closed, the set $F_{u,l}$ is compact.\ In particular,
	it follows from $u\notin F_{u,l}$ that $U_u:=X_{\theta'}\setminus\bigcup_{l=1}^k F_{u,l}$
	is an open neighborhood of~$u$ in~$X_{\theta'}$ satisfying
	\begin{equation} \label{eq:thm:stone_wstrass_vector:proof2}
		\frac{\|f(v)-w_u(v)\|_{\theta_l'}}{\psi(\|v\|_{\theta_l})}<\frac{\epsilon}{4}
		\quad\text{for all } v\in U_u\cap X_{\theta_l}.
	\end{equation}
	Since $(U_u)_{u\in K_1}$ is an open cover of the compact set $K_1$, there exist
	$u_1,\ldots,u_N \in K_1$ satisfying $K_1\subset\bigcup_{n=1}^N U_{u_n}$ and a 
	partition of unity $(g_n)_{n=1,\ldots,N}\subset\C(K_1)$ subordinate to 
	$(U_{u_n})_{n=1,\ldots,N}$ which means that $0\leq g_n\leq 1$, $\sum_{n=1}^N g_n=1$
	and $\supp(g_n)\subset U_{u_n}$.\ By using again that $K_1$ is compact and 
	$\supp(g_n)\subset U_{u_n}$, we can extend each $g_n\in\C(K_1)$ to some 
	$\tilde{g}_n\in\C(X_{\theta'})$ with $0\leq\tilde{g}_n\leq 1$, $\tilde{g}_n\vert_{K_1}=g_n$
	and $\supp(\tilde{g}_n)\subset U_{u_n}$.\ By further choosing $\tilde{g}_0\in\C(X_{\theta'})$
	with $0\leq\tilde{g}_0\leq 1$, $\tilde{g}_0\vert_{K_1}=0$ and $\tilde{g}_0\vert_{X_{\theta'}\setminus K_1}>0$,
	the normalized functions $\bar{g}_n:=\frac{\tilde{g}_n}{G}$ with $G:=\sum_{n=0}^N\tilde{g}_n$
	for $n=0,\ldots,N$ form a partition of unity on $X_{\theta'}$ satisfying $\bar{g}_0\vert_{K_1}=0$,
	$\bar{g}_n\vert_{K_1}=g_n$ and $\supp(\bar{g}_n)\subset U_{u_n}$ for $n=1,\ldots,N$.
	Let $w_n:=w_{u_n}$ for $n=1,\ldots,N$.\ Since $\bar{g}_0(u)>0$ implies $u\notin K_1$
	and therefore $u\in X_{\theta_l}\setminus K_{l,R_1}$ and $\bar{g}_n(u)>0$ implies $u\in U_{u_n}$,
	it follows from inequality~\eqref{eq:thm:stone_wstrass_vector:proof1} and 
	inequality~\eqref{eq:thm:stone_wstrass_vector:proof2} that
	\begin{align}
		& \max_{l=1,\ldots,k}\sup_{u\in X_{\theta_l}}
		\frac{\big\|f(u)-\sum_{n=1}^N\bar{g}_n(u)w_n(u)\big\|_{\theta_l'}}{\psi(\|u \|_{\theta_l})}
		\nonumber \\
		& \leq\max_{l=1,\ldots,k}\sup_{u\in X_{\theta_l}}\bar{g}_0(u) 
		\frac{\|f(u)\|_{\theta_l'}}{\psi(\|u\|_{\theta_l})} 
		+\max_{l=1,\ldots,k}\sup_{u\in X_{\theta_l}}\sum_{n=1}^N\bar{g}_n(u) 
		\frac{\|f(u)-w_n(u)\|_{\theta_l'}}{\psi(\|u\|_{\theta_l})} \nonumber \\
		&<\frac{\epsilon}{4}+\frac{\epsilon}{4}=\frac{\epsilon}{2} \label{eq:thm:stone_wstrass_vector:proof3}
	\end{align}
	Moreover, for every $n\in\{1,\ldots,N\}$, we apply Theorem~\ref{thm:stone_wstrass_R}
	to obtain $a_n\in\A$ with
	\begin{equation} \label{eq:thm:stone_wstrass_vector:proof4}
		\max_{l=1,\ldots,k}\|\bar{g}_n-a_n\|_{\B_\psi(X_{\theta_l,\theta_l'};\R)}<\frac{\epsilon}{2c_{w_n}N},
	\end{equation}
	where $c_{w_n}:=1+\max_{l=1,\ldots,k}\sup_{u\in X_{\theta_l}}\|w_n(u)\|_{Y_{\theta_l'}}$.
	Finally, the inequalities~\eqref{eq:thm:stone_wstrass_vector:proof3} 
	and~\eqref{eq:thm:stone_wstrass_vector:proof4} imply 
	\begin{align*}
		&\max_{l=1,\ldots,k}\bigg\|f-\sum_{n=1}^N a_n w_n\bigg\|_{\B_\psi(X_{\theta_l,\theta_l'};Y_{\theta_l'})} \\
		&\leq\max_{l=1,\ldots,k}\bigg\|f-\sum_{n=1}^N\bar{g}_n w_n\bigg\|_{\B_\psi(X_{\theta_l,\theta_l'};Y_{\theta_l'})} 
		+\sum_{n=1}^N\max_{l=1,\ldots,k}\|\bar{g}_n-a_n\|_{\B_\psi(X_{\theta_l,\theta_l'};\R)}
		\sup_{u\in X_{\theta_l}}\|w_n(u)\|_{Y_{\theta_l'}} \\
		&<\frac{\epsilon}{2}+\sum_{n=1}^N\frac{\epsilon}{2c_{w_n}N}c_{w_n}=\epsilon. 
	\end{align*}
	Since $\W$ is an $\A$-submodule, it holds that $\sum_{n=1}^N a_n\cdot w_n\in\W$ and the claim follows.
\end{proof}

Following~\cite[Section~3.3]{cuchiero26}, Theorem~\ref{thm:stone_wstrass_vector}
can be generalized to $\W\subset\bigcap_{\theta'\prec\theta}\B_\psi(X_{\theta,\theta'};Y_{\theta'})$ 
and $\A\subset\bigcap_{\theta'\prec\theta}\B_\psi(X_{\theta,\theta'};\R)$ 
consisting of unbounded mappings.

\section{Proof of Lemmas~\ref{lem:pde}--\ref{lem:wus_nisio}}
\label{app:ass}

\begin{lemma} \label{lem:diff_bds}
	Let $q\geq 2$, $\mu\in\Rd$, $\sigma\in\R^{d\times m}$, $\Sigma:=\sigma\sigma^\top\in\mathbb{S}^d_+$
	and $(W_t)_{t\geq 0}$ be a $m$-dimensional Brownian motion. Define 
	$X_t:=\mu t+\sigma W_t$ for all $t\geq 0$ and 
	$\omega_q:=q|\mu|+\frac{q}{2}\trace(\Sigma)+\frac{q(q-2)}{2}|\Sigma|$. Then,
	\[ \E[1+|x+X_t|^q]\leq e^{\omega_q t}(1+|x|^q) 
	\quad\text{for all } t\geq 0 \text{ and } x\in\Rd. \]
\end{lemma}
\begin{proof}
	Define $V_q(x):=1+|x|^q$. The generator of the It\^o process $(X_t)_{t\geq0}$ is given by
	\[ (\L f)(x):=\mu^\top\nabla f(x)+\frac{1}{2}\trace\big(\Sigma\nabla^2 f(x)\big). \]
	Since $\nabla V_q(x)=q|x|^{q-2}x$ and $\nabla^2 V_q(x)=q|x|^{q-2}I_d+q(q-2)|x|^{q-4}xx^\top$ 
	for all $x\neq 0$, which can be continuously extended to zero, we can use 
	$|x|^{q-1}\leq 1+|x|^q$ and $|x|^{q-2}\leq 1+|x|^q$ to obtain
	\begin{align*}
		(\L V_q)(x) 
		&=q|x|^{q-2}\mu^\top x+\frac{q}{2}|x|^{q-2}\trace(\Sigma)+\frac{q(q-2)}{2}|x|^{q-4}x^\top\Sigma x \\
		&\leq q|\mu|\cdot |x|^{q-1}+\frac{q}{2}\big(\trace(\Sigma)+(q-2)|\Sigma|\big)|x|^{q-2}
		\leq\omega_q V_q(x).
	\end{align*}
	For every $t\geq 0$ and $x\in\Rd$, it follows from It\^o's formula that
	\[ \E[V_q(x+X_t)]=V_q(x)+\int_0^t\E[\L V_q(x+X_s)]\,\d s 
	\leq V_q(x)+\omega_q\int_0^t\E[V_q(x+X_s)]\,\d s. \]
	Finally, Gronwall's lemma implies $\E[V_q(x+X_t)]\leq e^{\omega_qt}V_q(x)$
	for all $t\geq 0$ and $x\in\Rd$.
\end{proof}

\subsection{Proof of Lemma~\ref{lem:pde}}
\label{app:pde}

Regarding the verification of the Assumptions~\ref{ass:chernoff} and~\ref{ass:nisio}
as well as the extension $I_n\colon\C_{\kappa_0}\to\C_{\kappa_0}$, we refer
to~\cite[Theorem~3.6]{blessing23}, whereas Assumption~\ref{ass:cn_uat}(i) is clearly satisfied.\
Furthermore, by Lemma~\ref{lem:diff_bds}, there exists a constant $\omega\geq 0$ with
\begin{align*}
	\|I_{n,\lambda}f\|_{\alpha',\kappa'} 
	&\leq\max\left\{\sup_{x\in\Rd}\E[|f(x+Z_{n,\lambda})|]\kappa'(x),
	\sup_{x\neq y}\frac{\E[|f(x+Z_{n,\lambda})-f(y+Z_{n,\lambda})|]}{|x-y|^{\alpha'}}\bar{\kappa}'(x,y)\right\} \\
	&\leq\|f\|_{\alpha',\kappa'}\max\left\{\sup_{x\in\Rd}\E\left[\frac{\kappa'(x)}{\kappa'(x+Z_{n,\lambda})}\right], 
	\sup_{x\neq y}\E\left[\frac{\bar{\kappa}'(x,y)}{\bar{\kappa}'(x+Z_{n,\lambda},y+Z_{n,\lambda})}\right]\right\} \\
	&\leq e^{\omega h_n}\|f\|_{\alpha',\kappa'}
\end{align*}
for all $n\in\N$, $\lambda\in\Lambda$, $\alpha'\in [0,\alpha]$, $\kappa'\in\K$ 
and $f\in\C^{\alpha'}_{\kappa'}$, where $(I_{n,\lambda}f)(x):=\E[f(x+Z_{n,\lambda})]$
and $Z_{n,\lambda}:=\lambda h_n+W_{h_n}$.\ Indeed, for $\kappa'(x)=(1+|x|^{q'})^{-1}$,
it holds $\kappa'(x)^{-1}=1+|x|^{q'}$ and 
$\bar{\kappa}'(x,y)^{-1}=\frac{\kappa'(x)^{-1}+\kappa'(y)^{-1}}{2}=\frac{2+|x|^{q'}+|y|^{q'}}{2}$.\ 
Using the linearity of $I_{n,\lambda}$ and taking the supremum over $\lambda\in\Lambda$,
this shows that the Assumptions~\ref{ass:cn_uat}(ii)--(iii) and~\ref{ass:en}(i) are valid.\
Denoting by $c\geq 0$ the Lipschitz constant of $H^*|_\Lambda$, for every $n\in\N$, 
$\lambda_1,\lambda_2\in\Lambda$ and $f\in\C^\alpha_\kappa=\Lipb$, it holds
\begin{align*}
	&\|(I_{n,\lambda_1}f-\eta(\lambda_1))
	-((I_{n,\lambda_2} f)-\eta(\lambda_2)h_n)\|_\kappa \\
	&\leq\sup_{x\in\Rd}\E[|f(x+Z_{n,\lambda_1})-f(x+Z_{n,\lambda_2})|]+|H^*(\lambda_1)-H^*(\lambda_2)|h_n
	\\
	&\leq (\Vert f \Vert_{\alpha,\kappa}+c)|\lambda_1-\lambda_2|h_n.
\end{align*}
Hence, Assumption~\ref{ass:en_rate} is valid with $L_r:=(c+r)\sup_{n\in\N}h_n$, $d_r:=|\cdot|$ and 
$\beta_r:=1$ for all $r\geq 0$.

\subsection{Proof of Lemma~\ref{lem:soc}}
\label{app:soc}

Regarding the verification of the Assumptions~\ref{ass:chernoff} and~\ref{ass:nisio}
as well as the extension $I_n\colon\C_{\kappa_0}\to\C_{\kappa_0}$, we refer
to~\cite[Theorem~6.2]{denk25}, whereas Assumption~\ref{ass:cn_uat}(i) is clearly satisfied.
Furthermore, by Lemma~\ref{lem:diff_bds}, there exists a constant $\omega\geq 0$ with
\begin{align*}
	&\|I_{n,(\mu,\sigma)}f\|_{\alpha',\kappa'} \\
	&\leq\max\left\{\sup_{x\in\Rd}\E[|f(x+Z_{n,(\mu,\sigma)})|]\kappa'(x), 
	\sup_{x\neq y}\frac{\E[|f(x+Z_{n,(\mu,\sigma)})-f(y+Z_{n,(\mu,\sigma)})|]}
	{|x-y|^{\alpha'}}\bar{\kappa}'(x,y)\right\} \\
	&\leq\|f\|_{\alpha',\kappa'}\max\left\{\sup_{x\in\Rd}
	\E\left[\frac{\kappa'(x)}{\kappa'(x+Z_{n,(\mu,\sigma)})}\right],
	\sup_{x\neq y}\E\left[\frac{\bar{\kappa}'(x,y)}
	{\bar{\kappa}'(x+Z_{n,(\mu,\sigma)},y+Z_{n,(\mu,\sigma)})}\right] \right\} \\
	&\leq e^{\omega h_n} \Vert f \Vert_{\alpha',\kappa'}
\end{align*}
for all $n\in\N$, $(\mu,\sigma)\in\Xi\times\Sigma$, $\alpha'\in [0,\alpha]$, $\kappa'\in\K$ 
and $f\in\C^{\alpha'}_{\kappa'}$, where $Z_{n,(\mu,\sigma)}:=\mu h_n+\sigma W_{h_n}$
and $(I_{n,(\mu,\sigma)}f)(x):=\E[f(x+Z_{n,(\mu,\sigma)})]$.\
Using the linearity of $I_{n,(\mu,\sigma)}$ and taking the supremum over 
$(\mu,\sigma)\in\Xi\times\Sigma$ yields that the Assumptions~\ref{ass:cn_uat}(ii)--(iii) 
and~\ref{ass:en}(i) are satisfied.\ Finally, for every $n\in\N$, 
$(\mu_1,\sigma_1), (\mu_2,\sigma_2)\in\Xi\times\Sigma$ and $f\in\C^\alpha_\kappa$,
it follows from $\frac{1+|x+z|^q}{1+|x|^q}\leq 2^{q-1}(1+|z|^q)$, Lemma~\ref{lem:diff_bds},
H\"older's inequality and Jensen's inequality that 
\begin{align*}
	&\|I_{n,(\mu_1,\sigma_1)}f-I_{n,(\mu_2,\sigma_2)}f\|_\kappa 
	\leq\sup_{x\in\Rd}\E\big[|f(x+Z_{n,(\mu_1,\sigma_1)})-f(x+Z_{n,(\mu_2,\sigma_2)})|\big]\kappa(x) \\
	&\leq\sup_{x\in\Rd}\|f\|_{\alpha,\kappa}\E\left[\frac{|Z_{n,(\mu_1,\sigma_1)}-Z_{n,(\mu_2,\sigma_2)}|^\alpha}
	{\bar{\kappa}(x+Z_{n,(\mu_1,\sigma_1)},x+Z_{n,(\mu_2,\sigma_2)})}\right]\kappa(x) \\
	&\leq 2^{q-2}\|f\|_{\alpha,\kappa}\E\big[|Z_{n,(\mu_1,\sigma_1)}-Z_{n,(\mu_2,\sigma_2)}|^\alpha
	(2+|Z_{n,(\mu_1,\sigma_1)}|^q+|Z_{n,(\mu_2,\sigma_2)}|^q)\big] \\
	&\leq 2^{q-2}\|f\|_{\alpha,\kappa}\E\big[|Z_{n,(\mu_1,\sigma_1)}-Z_{n,(\mu_2,\sigma_2)}|^{2\alpha}\big]^\frac{1}{2}
	\E\big[(2+|Z_{n,(\mu_1,\sigma_1)}|^q+|Z_{n,(\mu_2,\sigma_2)}|^q)^2\big]^\frac{1}{2} \\
	&\leq 2^{q-2}\|f\|_{\alpha,\kappa}\E\big[|Z_{n,(\mu_1,\sigma_1)}-Z_{n,(\mu_2,\sigma_2)}|^2\big]^\frac{\alpha}{2} \E\big[(2+|Z_{n,(\mu_1,\sigma_1)}|^q+|Z_{n,(\mu_2,\sigma_2)}|^q)^2\big]^\frac{1}{2}\\
	&\leq C2^{q-2}\|f\|_{\alpha,\kappa}d((\mu_1,\sigma_1),(\mu_2,\sigma_2))^\alpha,
\end{align*}
where $d((\mu_1,\sigma_1),(\mu_2,\sigma_2)):=\sqrt{|\mu_1-\mu_2|^2+|\sigma_1-\sigma_2|^2}$ and
\[ C:=\sup_{n\in\N}\sup_{(\mu_1,\sigma_1),(\mu_2,\sigma_2)\in\Xi\times\Sigma}
\E\big[(2+|Z_{n,(\mu_1,\sigma_1)}|^q+|Z_{n,(\mu_2,\sigma_2)}|^q)^2\big]^\frac{1}{2}(h_n^2+h_n)^\frac{\alpha}{2}. \]
Hence, Assumption~\ref{ass:en_rate} is valid with $L_r:=2^{q-2}Cr$, 
$d_r:=d$ and $\beta_r:=\alpha$ for all $r\geq 0$.

\subsection{Proof of Lemma~\ref{lem:wus}}
\label{app:wus}

Regarding the verification of the Assumptions~\ref{ass:nisio}(i)--(vii), we 
refer to~\cite{fuhrmann23}. Since the function $\eta\colon\R_+\to [0,\infty]$ 
is non-decreasing, it holds
\[ I_n f=\sup_{(\nu,a)\in\Lambda}\big(I_{n,(\nu,a)}f-\eta_n(\nu,a)h_n\big) \]
for all $n\in\N$ and $f\in\Cb$, where $\Lambda:=\cP_p(\Rd)\times\R_+$ and 
\[ \eta_n(\nu,a):=\begin{cases}
	\eta(a), & \text{if } \mathcal{W}_p(\varpi_{h_n},\nu) \leq ah_n, \\
	\infty, & \text{otherwise.}
\end{cases} \]
Using $\alpha=1$ and $\kappa\equiv 1$, it holds for every $n\in\N$, 
$(\nu,a)\in\Lambda$ and $f\in\Lipb$ that
\begin{align*}
	\|I_{n,(\nu,a)}f\|_{\alpha,\kappa} 
	&\leq\max\left\{\sup_{x\in\Rd}\int_{\Rd}|f(x+z)|\,\nu(\d z), 
	\sup_{x\neq y}\frac{\int_{\Rd}|f(x+z)-f(y+z)|\,\nu(\d z)}{|x-y|}\right\} \\
	&\leq\max\left\{\|f\|_\infty,\|f\|_{\alpha,\kappa}\sup_{x\neq y}\frac{|(x+z)-(y+z)|}{|x-y|}\right\}
	=\|f\|_{\alpha,\kappa}.
\end{align*}
This shows that Assumption~\ref{ass:en}(i) is valid.\ Moreover, for every $r\geq 0$, 
there exists $R_r\geq 0$ with
\begin{align*}
	(I_n f)(x) &=\sup_{\{\nu\colon\W_p(\varpi_{h_n},\nu)\leq R_rh_n\}}
	\left(\int_{\Rd}f(x+z)\,\nu(\d z)-\eta\left(\frac{\W_p(\varpi_{h_n},\nu)}{h_n}\right)h_n\right) \\
	&=\sup_{(\nu,a)\in\Lambda_r}\big((I_{n,(\nu,a)}f)(x)-\eta_n(\nu,a)h_n\big)
\end{align*}
for all $n\in\N$, $f\in\Lipb(r)$ and $x\in\Rd$, see~\cite[Lemma~3.3]{fuhrmann23}, where
\[ \Lambda_r:=\{\nu\in\cP_p(\Rd)\colon\W_p(\nu,\delta_0)\leq C_r\}\times [0,R_r]
\quad\text{with}\quad C_r:=\sup_{n\in\N}\,(R_rh_n+\W_p(\varpi_{h_n},\delta_0)). \]
Equation~\eqref{eq:wus:moments:bound}
guarantees that $C_r<\infty$ and $p>1$ implies that $\Lambda_r$ is compact w.r.t.\ the metric 
\[ d((\nu_1,a_1),(\nu_2,a_2)):=\W_1(\nu_1,\nu_2)+|a_1-a_2|. \]
We use the Kantorovich--Rubinstein inequality, that $\eta\colon [0,R_r]\cap\{\eta<\infty\}\to\R_+$ 
is Lipschitz continuous with some constant $c_r\geq 0$ and $\bar{h}:=\sup_{n\in\N}h_n<\infty$  
to estimate
\begin{align*}
	&\big\|\big(I_{n,(\nu_1,a_1)}f-\eta_n(\nu_1,a_1)h_n\big) 
	-\big(I_{n,(\nu_2,a_2)}f-\eta_n(\nu_2,a_2)h_n\big)\big\|_\infty \\
	& \leq\sup_{x\in\Rd}\left|\int_{\Rd}f(x+z)\,\nu_1(\d z)-\int_{\Rd}f(x+z)\,\nu_2(\d z)\right| 
	+|\eta(a_1)-\eta(a_2)|h_n \\
	&\leq r\W_1(\nu_1,\nu_2)+\bar{h}c_r|a_1-a_2| \\
	&\leq\max\{r,\bar{h}c_r\} d((\nu_1,a_1),(\nu_2,a_2))
\end{align*}
for all $n\in\N$, $r\geq 0$, $f\in\Lipb(r)$ and 
$(\nu_1,a_1),(\nu_2,a_2)\in\Lambda_{r,n}:=\Lambda_r\cap\{\eta_n<\infty\}$.\
This shows that Assumption~\ref{ass:en_rate} is valid with $L_r:=\max\{r,\bar{h}c_r\}$,
$d_r:=d$ and $\beta_r:=1$ for all $r\geq 0$.

\subsection{Proof of Lemma~\ref{lem:wus_nisio}}
\label{app:wus_nisio}

Regarding the verification of Assumption~\ref{ass:nisio}, we refer to~\cite{fuhrmann23}, and
the verification of Assumption~\ref{ass:en}(i) is the same as in the proof of
Lemma~\ref{lem:wus}. Furthermore, we use $\eta(0)=0$ to obtain
\begin{align*}
	&\big((J_{n,\lambda}f)(x)-\eta(|\lambda|)h_n\big)-\big((J_{n,0}f)(x)-\eta(0)h_n\big) \\
	&=\int_{\Rd}f(x+\lambda h_n+y)-f(x+y)\,\varpi_{h_n}(\d y)-\eta(|\lambda |)h_n
	\leq\big(r|\lambda|-\eta(|\lambda|)\big)h_n
\end{align*}
for all $n\in\N$, $r\geq 0$, $f\in\Lipb(r)$ and $\lambda,x\in\Rd$, where 
\[ (J_{n,\lambda}f)(x):=\int_{\Rd}f(x+\lambda h_n+y)\,\varpi_{h_n}(\d y). \]
Since $\liminf_{v\to\infty}\frac{\eta(v)}{v^p}>0$ with $p>1$, there exists 
a totally bounded set $\Lambda_r\subset\Rd$ with
\[ J_n f=\sup_{\lambda\in\Lambda_r}\big(J_{n,\lambda}f-\eta(|\lambda|)h_n\big) \]
for all $n\in\N$, $r\geq 0$ and $f\in\Lipb(r)$.\ We use that $\eta(|\cdot|)\colon\Lambda_r\to\R_+$
is Lipschitz continuous with some constant $c_r\geq 0$ and $\bar{h}:=\sup_{n\in\N}h_n<\infty$
to estimate 
\begin{align*}
	&\big\|\big((J_{n,\lambda_1}f)-\eta(|\lambda_1|)h_n\big)
	-\big(J_{n,\lambda_2}f-\eta(|\lambda_2|)h_n\big)\big\|_\infty \\
	&\leq\sup_{x\in\Rd}\left|\int_{\Rd}f(x+\lambda_1 h_n+y)\,\varpi_{h_n}(\d y)
	-\int_{\Rd}f(x+\lambda_2 h_n+y)\,\varpi_{h_n}(\d y)\right|
	+\big|\eta(|\lambda_1|)-\eta(|\lambda_2|)\big|h_n \\
	&\leq\bar{h}(r+c_r)|\lambda_1-\lambda_2|
\end{align*}
for all $n\in\N$, $r\geq 0$, $f\in\Lipb(r)$ and $\lambda_1,\lambda_2\in\Lambda_r$.\
This shows that Assumption~\ref{ass:en_rate} is valid with $L_r:=\bar{h}(r+c_r)$, $d_r:=|\cdot|$
and $\beta_r:=1$ for all $r\geq 0$.

\medskip

\noindent\textbf{Acknowledgments.} P.~Schmocker is partly supported by the FinsureTech Hub of ETH Zurich. 

\bibliographystyle{abbrv}
\bibliography{mybib}

\end{document}